\documentclass[12pt]{amsart}

\usepackage[left=3cm,marginpar=2.5cm,tmargin=2.5cm,bmargin=2.5cm]{geometry}

\usepackage[all]{xy}
\usepackage{amscd}
\usepackage{mathrsfs}
\usepackage{amssymb}
\usepackage{amsthm}
\usepackage{amsmath}
\usepackage{enumerate}
\usepackage{indentfirst}
\usepackage{dsfont}
\usepackage{microtype}

\usepackage{xcolor}

\usepackage[obeyspaces,hyphens,spaces]{url}
\usepackage{hyperref}

\usepackage{siunitx}

\makeatletter
\def\@xocreysp{ } %
\makeatother
\usepackage{color,xcolor}
\newcommand{\mpar}[1]{}
\newcommand{\micropar}[1]{}
\newcommand\fixit[1]{\ignorespaces}

\usepackage{comment}

\newtheorem{thm}{Theorem}[section]
\newtheorem{conj}[thm]{Conjecture}
\newtheorem{cor}[thm]{Corollary}
\newtheorem{lem}[thm]{Lemma}
\newtheorem{prop}[thm]{Proposition}

\theoremstyle{definition}
\newtheorem{defn}[thm]{Definition}
\newtheorem{ex}[thm]{Example}
\theoremstyle{remark}
\newtheorem{rem}[thm]{Remark}

\numberwithin{equation}{section}

\newtheorem{remark}[thm]{Remark}

\newcommand{\CaA}{\mathcal A}
\newcommand{\CaB}{\mathcal B}

\newcommand{\CaO}{\mathcal O}

\newcommand{\bbF}{\mathbb F}

\newcommand{\bbR}{\mathbb R}

\newcommand{\bbZ}{\mathbb Z}

\newcommand{\rd}{\mathrm d}

\newcommand{\xo}{x}

\newcommand{\sfG}{\mathsf G}

\newcommand{\sfT}{\mathsf T}
\newcommand{\sfM}{\mathsf M}

\newcommand{\sfU}{\mathsf U}

\newcommand{\scrT}{(\mathscr T)}
\newcommand{\scrE}{(\mathscr E)}

\newcommand{\val}{\operatorname{ord}}
\newcommand{\Ad}{\mathrm{Ad}}
\newcommand{\ad}{\mathrm{ad}}
\newcommand{\pid}{\mathfrak p}
\newcommand{\orbi}{\mathcal O}
\newcommand{\orb}{\mathscr O}

\newcommand{\bG}{\mathbf G}

\newcommand{\bT}{\mathbf T}

\newcommand{\bU}{\mathbf U}
\newcommand{\bS}{\mathbf S}

\newcommand{\bZ}{\mathbf Z}

\newcommand{\lieG}{\mathfrak g}

\newcommand{\lieT}{\mathfrak t}

\newcommand{\blieG}{\boldsymbol{\lieG}}

\newcommand{\blieT}{\boldsymbol{\lieT}}

\newcommand{\datum}{\Sigma}

\newcommand{\Bd}{\CaB}

\newcommand{\Apt}{\CaA}

\newcommand{\dth}{\rd}

\newcommand{\rtm}{r}

\newcommand{\supp}{\mathrm{Supp}}
\newcommand{\reg}{{\mathrm{reg}}}

\newcommand{\Hom}{\mathrm{Hom}}

\newcommand{\cind}{\textrm{c-}\mathrm{ind}}

\newcommand{\Gal}{\mathrm{Gal}}

\newcommand{\vol}{\mathrm{vol}}

\newcommand{\midvsp}{\vspace{7pt}}

\newcommand{\A}{\mathbb{A}}

\newcommand{\bbE}{\mathbb E}

\newcommand{\cF}{{\mathcal{F}}}

\newcommand{\cA}{{\mathcal{A}}}

\newcommand{\fkg}{{\mathfrak g}}
\newcommand{\fkk}{{\mathfrak k}}

\newcommand{\fkS}{{\mathfrak S}}

\newcommand{\Z}{\mathbb{Z}}

\newcommand{\Q}{\mathbb{Q}}

\newcommand{\F}{\mathbb{F}}

\newcommand{\R}{\mathbb{R}}

\newcommand{\C}{\mathbb{C}}

\newcommand{\scusp}{\mathrm{sc}}

\newcommand{\elp}{\mathrm{ell}}
\newcommand{\hs}{\mathrm{hs}}

\newcommand{\ord}{\mathrm{ord}}

\newcommand{\pl}{\hat{\mu}^{\mathrm{pl}}}

\newcommand{\bs}{\backslash}

\newcommand{\semis}{\mathrm{ss}}

\newcommand{\cnt}{\mathrm{count}}
\newcommand{\St}{\mathrm{St}}
\newcommand{\Mot}{\mathrm{Mot}}
\newcommand{\rk}{\mathrm{rk}}

\newcommand{\rank}{{\rm rank}}
\def\Gal{{\rm Gal}}

\newcommand{\spec}{{\rm spec}}

\newcommand{\tr}{{\rm tr}\,}

\newcommand{\disc}{{\rm disc}}

\newcommand{\geom}{{\rm geom}}

\newcommand{\GL}{{\rm GL}}

\newcommand{\alg}{{\rm alg}}
\newcommand{\Ram}{{\rm Ram}}

\newcommand{\triv}{{\mathbf{1}}}
\newcommand{\EP}{{\rm{EP}}}

\newcommand{\ur}{{\mathrm{ur}}}

\newcommand{\Yu}{{\rm Yu}}

\newcommand{\cH}{{\mathcal{H}}}

\newcommand{\cO}{\mathcal{O}}

\newcommand{\can}{{\rm can}}

\newcommand{\Irr}{{\rm Irr}}

\newcommand{\Res}{\mathrm{Res}}

\def\hat{\widehat}

\def\lg{\langle}
\def\rg{\rangle}
\def\hra{\hookrightarrow}
\def\ra{\rightarrow}

\def\ol{\overline}
\def\tilde{\widetilde}

\newcommand{\SL}{\mathrm{SL}}

\def\benu{\begin{enumerate}}
\def\eenu{\end{enumerate}}
\def\beq{\begin{equation}}
\def\eeq{\end{equation}}
\def\bit{\begin{itemize}}
\def\eit{\end{itemize}}
\DeclareMathOperator {\PGL} {PGL}
\newcommand{\SB}{\backslash}
\DeclareMathOperator {\Mtr}  {tr}

\begin{document}

\title[Asymptotics of supercuspidal representations]{Asymptotic behavior of supercuspidal representations and Sato-Tate equidistribution for families}

\author{Ju-Lee Kim}
\address{Department of Mathematics, Massachusetts Institute of Technology,
77 Massachusetts Avenue, Cambridge, MA 02139, USA}
\email{julee@math.mit.edu}
\author{Sug Woo Shin}
\address{Department of Mathematics, UC Berkeley, Berkeley, CA 94720, USA$//$ Korea Institute for Advanced Study, 85 Hoegiro,
Dongdaemun-gu, Seoul 130-722, Republic of Korea}
\email{sug.woo.shin@berkeley.edu}
\author{Nicolas Templier}
\address{Department of Mathematics, Cornell University,
Ithaca, NY 14850, USA}
\email{templier@math.cornell.edu}

\date{\today}
\subjclass[2010]{20G25 22E35 11F55}

\begin{abstract}
We establish properties of families of automorphic representations as we vary prescribed supercuspidal representations at a given finite set of primes. For the tame supercuspidals constructed by J.-K.~Yu we prove the limit multiplicity property with error terms. Thereby we obtain a Sato-Tate equidistribution for the Hecke eigenvalues of these families. The main new ingredient is to show that the orbital integrals of matrix coefficients of tame supercuspidal representations with increasing formal degree on a connected reductive $p$-adic group tend to zero uniformly for every noncentral semisimple element.

\end{abstract}

\maketitle
\thispagestyle{empty}

\tableofcontents

\section{Introduction}\label{s:intro}

\subsection{Limit multiplicity}\label{sub:limit-mult}
We begin this introduction by explaining our results on counting discrete automorphic
representations.
Let $G$ be a semisimple group over a totally real field $F$. Write
$F_\infty:=F\otimes_\Q \R$.
Consider a sequence of lattices $\{\Gamma_j\}_{j\ge 1}$ in $G(F_\infty)$ whose covolumes
tend to infinity as $j\to \infty$. For an irreducible unitary representation $\pi_\infty$
denote by $m(\pi_\infty,\Gamma_j)$ the multiplicity of $\pi_\infty$ occurring in the
discrete spectrum of $L^2(\Gamma_j\backslash G(F_\infty))$.
DeGeorge-Wallach~\cite{DeGeor-Wall,DeGeor-Wall-II} in the compact case and
Rohlfs-Speh~\cite{RS87} and Savin~\cite{Sav89} in the arithmetic non-compact case
 proved that if $\{\Gamma_j\}$ is a normal series whose intersection is the identity, then
\begin{equation}\label{DW}
\lim_{j\to \infty}
\frac{m(\pi_\infty,\Gamma_j)}{\vol(\Gamma_j \SB G)} = \deg(\pi_\infty).
\end{equation}
Here $\deg(\pi_\infty)$ is the formal degree which by convention is non-zero if and only if $\pi_\infty$ is square-integrable. By different methods it is shown in~\cite{BG7:betti11} that~\eqref{DW} holds if $\{\Gamma_j\}$ is Benjamini-Schramm convergent and uniformly discrete (which recovers the compact case but not the non-compact case).

Our goal is to investigate refinements where instead of the lattice $\Gamma_j$ we impose
a prescribed supercuspidal representation $\sigma_j$.
Let $u$ be a finite place, and consider a sequence
$\{\sigma_j\}_{j\ge1}$ of tame supercuspidal representations of $G(F_u)$ whose formal
degrees tend to infinity as $j\to \infty$. For an irreducible algebraic representation
$\xi$ of $G(F_\infty)$ with regular highest weight, let $\Pi_\infty(\xi)$ denote the
$L$-packet of square-integrable representations of $G(F_\infty)$ whose infinitesimal and
central characters are dual to those of $\xi$.
Write $m(\xi,\sigma_j)$ for the number of discrete automorphic representations $\pi$ in
$L^2(G(F)\backslash G(\A_F))$, counted with automorphic multiplicity $m(\pi)$, such that
$\pi_\infty\in\Pi_\infty(\xi)$, $\pi_u\simeq\sigma_j$, and $\pi$ is unramified at all
finite places away from $u$. By results of Harish-Chandra the cardinality
$|\Pi_\infty(\xi)|$ is equal to the order of the Weyl group of $G(F_\infty)$ divided by
the order of the Weyl group of a maximal compact subgroup of $G(F_\infty)$. Informally
one of our main results (Corollary \ref{cor:limit-multiplicity}, cf.
\eqref{e:intro-lim-mult} below) states,
provided that the residue characteristic of $F_u$ is sufficiently large, that
\begin{equation}\label{i:limit-mult}
\lim_{j\to \infty}
\frac{m(\xi,\sigma_j)}{\deg(\sigma_j)} = c |\Pi_\infty(\xi)|\dim \xi
\end{equation}
for a positive constant $c$ independent of $\xi$. We refer to \S\ref{sub:simple-TF}
for
the determination of $c$ which is related to the Tamagawa number of $G$.

Note that the integer $|\Pi_\infty(\xi)|\dim \xi$ is equal to $\sum_{\pi_\infty\in
\Pi_\infty(\xi)} \deg(\pi_\infty)$ up to a multiplicative constant depending only on the
Haar measure on $G(F_\infty)$.
Compared with~\eqref{DW}, we are averaging over $\pi_\infty$ ranging in the $L$-packet
$\Pi_\infty(\xi)$ for
technical simplicity in the trace formula; this simplification does not interfere with
the new phenomena at the finite prime $u$ that we are concentrating on.

\begin{ex}\label{ex:Gross}
 For the group $\PGL(2)$ consider discrete automorphic representations that ramify above
 a single prime $q$ and are unramified elsewhere. Let $\mathcal{D}_k$ be the discrete
 series representation of $\PGL(2,\R)$ of weight $k\ge 2$ (necessarily even). For each
 simple supercuspidal representation $\sigma$ of $\PGL(2,\Q_q)$ with $q > 3$, there is an
 \emph{exact} multiplicity formula
\begin{equation}\label{Gross:q3}
m(\mathcal{D}_k,\sigma) + m(\mathcal{D}_k,\sigma') = \frac{k-1}{12}(q^2-1),
\end{equation}
where $\sigma'$ is the other representation
with the same affine generic character as $\sigma$, see~~\cite{Gross:prescribed}. (The
assumptions in
~\cite{Gross:prescribed} differ slightly, but one can verify that the same argument
applies, the key point being that $\PGL(2,\Q)$ has no $q$-torsion.)

Simple supercuspidal representations for $\PGL_2(\Q_q)$ coincide with the representations
of $\GL_2(\Q_q)$ of conductor $q^3$ and trivial central character. There are $2(q-1)$
distinct simple supercuspidal $\sigma$'s, partitioned into $q-1$ pairs
$\{\sigma,\sigma'\}$, thus~\eqref{Gross:q3} leads to the
observation~\cite{Gross:prescribed} that for any even integer $k\ge 2$ and any prime
$q> 3$, the dimension of the space $S_k(q^3)^{\mathrm{new}}$ of newforms of weight $k$
and level
$\Gamma_0(q^3)$
is\footnote{This formula can also be established from
\[
 \dim S_k(q^3)^{\mathrm{new}} =
\dim S_k(q^3)
- 2 \dim S_k(q^2)
+ \dim S_k(q),
\]
and the dimension formulas of $S_k(N)$ derived from
Riemann-Roch~\cite{Martin:dimension}.}
\begin{equation*}
m(\mathcal{D}_k,\Gamma_0(q^3))^{\mathrm{new}}=\dim S_k(q^3)^{\mathrm{new}} =
\frac{k-1}{12} (q+1)(q-1)^2.
\end{equation*}
Since the formal degree of $\mathcal{D}_k$ is $\frac{k-1}{12}$, this is a strong form of
the limit multiplicity property~\eqref{DW} as $q\to \infty$.

On the other hand, we establish by the same method of proof as~\eqref{i:limit-mult},
see also~\cite{Weinstein09}, the following asymptotic
\begin{equation*}
m(\mathcal{D}_k,\sigma) \sim \frac{k-1}{24}(q^2-1), \quad \text{as $k,q\to \infty$},
\end{equation*}
 because the formal degree of simple
supercuspidals is $\deg(\sigma)=\deg(\sigma')=\frac{q^2-1}{2}$.
Thus, the asymptotic~\eqref{i:limit-mult} and the exact formula~\eqref{Gross:q3} are
consistent, although none implies the other.
\end{ex}

\subsubsection*{Depth aspect}
Recall the notion of depth of an admissible representation~\cite{MP94}.
As a special case of Conjecture~\ref{c:asymptotic-char} below, we conjecture that
for a
sequence $(\sigma_j)_{j\ge 1}$ of supercuspidal representations of a $p$-adic group,
the
condition
$\deg(\sigma_j) \to \infty$ is equivalent to $\operatorname{depth}(\sigma_j)\to \infty$.
In our present context of tame supercuspidals, this is easy to verify (see,
e.g., the
proof of Proposition~\ref{p: power-saving}).

We may refer to the asymptotic~\eqref{i:limit-mult} as a limit multiplicity
result in the depth aspect.
Indeed we view it as an analogue at finite places of the
limit multiplicity in the ``weight aspect"~\cite{Shi-plan,ST:Sato-Tate}.
Indeed in the weight aspect, the roles of $\sigma_j$  and $\xi$ are
interchanged,
namely
$\sigma_j$ remains fixed and $\xi\ra\infty$ (in the sense that the highest weight for
$\xi$ gets arbitrarily far from the walls), whereas the above families
have fixed $\xi$ and $\deg(\sigma_j)\ra\infty$. In fact, we
also
 establish hybrid results in this paper, where both $\xi$ and $\sigma_j$ tend
to infinity.
For example our results below allow us to obtain an error bound for
~\eqref{i:limit-mult}, saving powers for both $\dim(\xi)$ and $\deg(\sigma_j)$ (Corollary
\ref{c:disc-seires-case-restated} below).

\subsection{Quantitative equidistribution for a family}\label{s:quantitative}
In the same context as before, for simplicity, let $G$ be a split semisimple group over a totally real field $F$ with trivial center. (In
the main text $G$ need not be either split or semisimple with trivial center.) Let each of $S_0$ and $S$ be a finite set of finite places of
$F$ such that $S\neq \emptyset$ (but $S_0$ could be empty) and $S_0\cap S=\emptyset$. Denote by $S_\infty$ the set of infinite places of $F$
and put $\fkS:=S_\infty \cup S_0 \cup S$. Set $F_{S_0}:=\prod_{v\in S_0} F_v$ and $\A^S_F:=\prod'_{v\notin S} F_v$. Let $K_{S_0}$ be an open compact subgroup of $G(F_{S_0})$, and $K^\fkS$ an open compact subgroup of $G(\A^\fkS_F)$ which is the product of hyperspecial subgroups over $v\notin \fkS$ (which arise from a global Chevalley group for $G$ over $\Z$). We will consider
\begin{itemize}
  \item irreducible algebraic representations $\xi$ of $G(F_\infty)$ with regular highest weight,
  \item irreducible supercuspidal representations $\sigma$ of $G(F_S)$.
\end{itemize}
For a technical reason we will impose the condition that $\xi\in \Irr^{\reg}_C(G(F_\infty))$ for a fixed constant $C\ge1$ (see \S\ref{sub:bounds} for details; the error bound in the theorem depends on $C$).

Let $\cF(\xi,\sigma,K_{S_0})$ be the multi-set of discrete automorphic representations $\pi$, counted with multiplicity $m(\pi)\dim( \pi_{S_0})^{K_{S_0}}$ (a number occurring naturally in the limit multiplicity problem), such that $\pi_\infty\in \Pi_\infty(\xi)$, $\pi_S\simeq \sigma$, and $(\pi^\fkS)^{K^\fkS}\neq 0$.
We let both $\xi$ and $\sigma$ vary, which puts discrete series representations at infinite places (grouped in $L$-packets) and supercuspidal representations at finite places on an equal footing.
Write $m(\xi,\sigma,K_{S_0}):=|\cF(\xi,\sigma,K_{S_0})|$. Fix a Haar measure on
$G(\A_F)$.

Our main result in a simpler form is the following Sato-Tate equidistribution for the family $\cF$.
See Theorem \ref{t:disc-series-case} and Corollaries \ref{cor:limit-multiplicity} and \ref{c:disc-seires-case-restated} for precise statements.
In the special case where $\sigma$ is fixed, our result generalizes \cite[Thm.~9.19]{ST:Sato-Tate}.

\begin{thm}\label{t:sato-tate}
Suppose that the residue characteristic of every $v\in S$ is sufficiently large (in a way depending on $G$).
We have the limit multiplicity formula as $\dim(\xi),\deg(\sigma)\ra\infty$
\begin{equation}\label{e:intro-lim-mult}
 m(\xi,\sigma,K_{S_0}) \sim c\cdot \dim(\xi)\deg(\sigma)
 \end{equation}
for an explicit constant $c>0$.
Moreover there exist $\nu,A>0$ depending only on $G$
 such that for every $\xi$ and $\sigma$ as above, and for every function $\phi:G(\A^\fkS_F)\to \C$ which is the characteristic function of a $K^\fkS$-double coset, we have the asymptotic formula
\begin{equation}\label{e:sato-tate}
\sum_{\pi\in \cF(\xi,\sigma,K_{S_0})}
\tr \pi^\fkS(\phi)
=
m(\xi,\sigma,K_{S_0}) \phi(e)
+ O(m(\xi,\sigma,K_{S_0})^{1-\nu} \|\phi\|_1^A ).
\end{equation}
The multiplicative constant depends on $G$, $C$, $S$, $K_{S_0}$ but is independent of $\xi$, $\sigma$ and $\phi$.
\end{thm}

\begin{ex}
Suppose $G=\PGL(2)$ and $F=\Q$. We take $S_0=\emptyset$ and $S$ any non-empty finite set
of sufficiently large finite primes. We are counting, for even integer weights
$k \ge 2$, and irreducible
supercuspidal representations $\sigma$ of $\PGL(2,\Q_S)$, the number
$m(\mathcal{D}_k,\sigma)$ of cusp forms $f\in \cF(\mathcal{D}_k,\sigma)$ of weight $k$
unramified outside of $S$ and with local component $\sigma$ at $S$.
The limit multiplicity asymptotic~\eqref{e:intro-lim-mult} recovers~\cite{Weinstein09} as
in Example~\ref{ex:Gross}:
\[
 m(\mathcal{D}_k,\sigma)  \sim \frac{k-1}{12} \deg(\sigma) \quad \text{as $k,\dim(\sigma)\to \infty$}.
\]
The second assertion~\eqref{e:sato-tate} on Sato-Tate equidistribution is new already in
this case of $\PGL(2)$. For example, if the function $\phi$ is a Hecke operator
$T_n$  for some integer
$n\ge 1$ not divisible by any prime in $S$, then:
\begin{equation}\label{GL2-Sato-Tate}
\sum_{f\in \cF(\mathcal{D}_k,\sigma)} \frac{a_n(f)}{\sqrt{n}}
= m(\mathcal{D}_k,\sigma) \delta_{n=\square} + O(n),
\end{equation}
where $\delta_{n=\square}$ is one if $n$ is a perfect square and zero otherwise. Note
that we normalize $a_n(f)$ in such a way that
Deligne's bound reads $|a_p(f)|\le 2$. The above precise error term
(corresponding to $\nu=1$ and $A=0$) is derived from the Sally-Shalika character
formula~\cite[App.~A]{KST-localconstancy}.
Similarly as in~\S\ref{sub:limit-mult}, the equidistribution~\eqref{GL2-Sato-Tate} is a
refinement of an earlier result. Namely, Serre~\cite{Serre:pl} considered the trace of
the
Hecke operator $T_n$ on the space $S_k(N)$ of cusp forms $f$ of weight $k$ and level
$\Gamma_0(N)$, and established for $n\ge 1$ coprime with $N$,
\[
\sum_{f \in \cF(\mathcal{D}_k,\Gamma_0(N))}
\frac{a_n(f)}{\sqrt{n}}
=
m(\mathcal{D}_k,\Gamma_0(N))  \delta_{n=\square} + O(n).
\]
\end{ex}

\subsubsection*{Remark}
Interestingly, the depth aspect families studied in this paper are rather thin compared
to
the families formed by varying a lattice subgroup. In favorable situations, and
assuming that $S_0=\emptyset$, the global root number of $\pi\in
\cF(\xi,\sigma,\emptyset)$ depends only on $\xi$ and $\sigma$.
This almost never happens for thicker families arising from limit multiplicity
problems where the whole lattice subgroup $\Gamma_j$ varies.

\subsubsection*{Conductor vs depth}
We view the sets $\cF(\xi,\sigma,K_{S_0})$ with varying tame supercuspidal
representations $\sigma$ of $G(F_S)$ as forming a harmonic family in the sense
of~\cite{SST}. Then Theorem~\ref{t:sato-tate} essentially gives us the Sato--Tate
equidistribution for families stated as Conjecture~1 in~\cite{SST}. One difference is
that the formulation of~\cite[Conj.~1]{SST} involves analytic conductors whereas our
results are expressed in terms of formal degrees. The relation between formal degree and
conductor is not yet established in general, this is a problem closely related to
that of the
depth
preservation in the local Langlands correspondence~\cite{Yu09}. %

\subsection{Bounds towards Ramanujan}
We can deduce from Theorem~\ref{t:sato-tate} an average bound towards Ramanujan.
For every place $v\not \in \fkS$ and every $\theta>0$, there is $\varrho>0$ such
that
\begin{equation}\label{ram}
\# \{
\pi \in \cF(\xi,\sigma,K_{S_0}),\
\log_v |\alpha(\pi_v)| > \theta
\} \ll m(\xi,\sigma,K_{S_0})^{1-\varrho},
\end{equation}
where $\alpha(\pi_v)$ is the Satake parameter of the unramified representation $\pi_v$.
It is unitarily normalized so that $|\alpha(\pi_v)| = 1$ if and only if $\pi_v$ is
tempered.
The multiplicative constant and the exponent $\varrho$ are
independent of $\xi$ and $\sigma$ (they depend only on $G$, $C$, $S$, $K_{S_0}$,
$\theta$, $v$).
The proof proceeds in the same way as for~\cite[Cor.~1.8]{MT:gln}. Namely we first
construct a function $\phi_1$ which is a bi-$K_v$-invariant function on $G(F_v)$ such
that $\Mtr \pi_v(\phi_1)$ is uniformly small for $|\alpha(\pi_v)|\le 1$ and uniformly
large for $|\alpha(\pi_v)|\ge q_v^{\theta}$.
Then we apply Theorem~\ref{t:sato-tate} for $\phi:=(\phi_1 * \phi^\vee_1)^{*k}$ with
the integer $k\ge 1$ chosen proportional to $\log m(\xi,\sigma,K_{S_0})$,
see~\cite[\S3]{MT:gln}.

The estimate~\eqref{ram} shows that exceptions to the Ramanujan bound are sparse.
For quasi-split classical groups the Ramanujan bound may be reduced to the self-dual or
conjugate self-dual case of general linear groups via work of Arthur~\cite{Arthur} and
Mok~\cite{Mok}. The latter case is settled when cuspidal automorphic representations are
cohomological (over totally real fields in the self-dual case; over CM fields in the
conjugate self-dual case) by \cite{Clo91,Kot92b,Shi11,Clozel:purity,Car12}. In particular
the Ramanujan conjecture is known for the representations $\pi \in
\cF(\xi,\sigma,K_{S_0})$ if $G$ is a split classical group.\footnote{Here we use the fact
that $\xi$ has \emph{regular} highest weight. This forces the representation of Arthur's
$\SL(2)$ in the global Arthur parameter to be trivial by examining the infinitesimal
character at infinite places. %
}
For exceptional groups $G$ very little is known and even a formulation of the Ramanujan conjecture is delicate, see~\cite{Sar05} and~\cite{Shahidi:packet-Ramanujan} for recent treatments.

\subsection{Trace formula and tame supercuspidal coefficients}
We find that the limit multiplicity and quantitative equidistribution described above are related to  asymptotic properties of orbital integrals. The first step in the proof of Theorem~\ref{t:sato-tate} is to express the left-hand side of~\eqref{e:sato-tate} as the spectral side of the trace formula for a suitably chosen test function. Since the weight $\xi$ is regular and $\sigma$ is supercuspidal we can use the simple trace formula.

There exist test functions $f_\sigma$ that single out the given supercuspidal
representation $\sigma$ in the trace formula, obtained by forming matrix
coefficients.\footnote{In general $f_\sigma$ is compactly supported only modulo center of
$G$, but the center is finite as $G$ is semisimple. In the main text we work with
reductive groups with compact center but see Remark \ref{r:non-cpt-center}.}  In our
situation J.-K.~Yu's construction gives $\sigma$ as compactly induced from a finite
dimensional representation on a compact open modulo center subgroup of $G(F_S)$. (Every
$\sigma$ arises in this way if the residue characteristics of places in $S$ are
sufficiently large by the exhaustion theorem~\cite{Kim07}.) This provides an explicit
$f_\sigma$ which is essential for our purpose.

We can now explain in more details the geometric side in the application of the trace formula. The geometric side is a sum over conjugacy classes of semisimple elements $\gamma \in G(F)$ of a volume term times a global orbital integral. The global orbital integral is a product of orbital integrals at ramified places in $S$, for which the main contribution is $\orbi_\gamma(f_\sigma)$, and orbital integrals at unramified places. %

Here we are varying the supercuspidal coefficient $f_\sigma$ which is unlike the usual applications of the trace formula where it is fixed. A general approach to this situation appears in~\cite{ST:Sato-Tate} in the weight aspect and we can use the results of~\cite{ST:Sato-Tate} to estimate most of the terms in the geometric side of the trace formula, except for $\orbi_\gamma(f_\sigma)$ which is new.

For the proof of~\eqref{i:limit-mult} we establish that $|\orbi_\gamma(f_\sigma)|=o(\deg(\sigma))$ as $\deg(\sigma)\to \infty$, and for any fixed $\gamma$. The proof of Theorem~\ref{t:sato-tate} is much more difficult due to the uniformity in $\phi$. As in~\cite{ST:Sato-Tate} the number of terms in the geometric side is unbounded, and uniform estimates for orbital integrals are needed. Moreover the estimate for $\orbi_\gamma(f_\sigma)$ has to be made quantitative and uniform in $\gamma$ which we discuss in the next subsection.

\subsection{Asymptotic behavior of orbital integrals}
We have seen in the previous subsection that our approach leads to the problem of establishing uniform bounds for orbital integrals of supercuspidal coefficients. In general it would be desirable to develop a quantitative theory of orbital integrals. This is for example advocated in the introduction of~\cite{DeBacker-Sally}. Our present problem of establishing uniform bounds for $\orbi_\gamma(f_\sigma)$ goes in this direction.

Theorem~\ref{t:asymptotic-orb-int} below states that there exists a constant $\eta<1$ depending only on the group $G(F_S)$ such that for all noncentral elements $\gamma$ and all tame supercuspidal representation $\sigma$ of $G(F_S)$, we have
\begin{equation}\label{Ogamma}
D(\gamma)^{\frac{1}{2}}\,\left| \orbi_\gamma(f_\sigma)\right| \ll \deg(\sigma)^\eta,
\end{equation}
where $D(\gamma)$ is the discriminant of $\gamma$.
This result is the technical heart of the paper.

The properties of $\orbi_\gamma(f_\sigma)$ are related to the trace character
$\Theta_\sigma(\gamma)$. In fact the two quantities agree if $\gamma$ is regular
elliptic, and we derive some consequences in \S\ref{s:asymptotic-behavior-sc}. However it
should also be noted that for our application it is essential to include the case where
$\gamma$ is \emph{non-regular} elliptic (in which case $\Theta_\sigma(\gamma)$ is
undefined).
For explicit computations of $\Theta_\sigma(\gamma)$ for regular semisimple $\gamma$, we refer to \cite{AS09,DS:supercuspidal,Kaletha:supercuspidal}.
In some special cases $\orbi_\gamma(f_\sigma)$ can be computed exactly, especially if one allows an additional average of $\sigma$ (over an $L$-packet). In fact one could allow $\sigma$ to be not only supercuspidal but also discrete series representations. Notably if $\sigma$ is the Steinberg representation, then Kottwitz~\cite{Kot88} constructed an Euler-Poincar\'e function $f^{\mathrm{EP}}$ which is a pseudo-coefficient for $\sigma$. In this case~\eqref{Ogamma} holds with $f_\sigma=f^{\mathrm{EP}}$ in the horizontal aspect as the residue characteristics of places in $S$ grow to infinity, see Section~\ref{s:Steinberg}.

Though exact formulas for orbital integrals and for trace characters are extremely
difficult to obtain beyond some special cases, we manage to prove the desired
asymptotic~\eqref{Ogamma}. We indicate an outline of our proof. It follows from Yu's
construction that the function $f_\sigma$ can be chosen as a matrix coefficient and is
supported on an explicit open compact subgroup $J\subset G(F_S)$. We recall in
Section~\ref{s:Yu-construction} how $J$ is constructed from a generic $G$-datum. From
this we reduce the estimate to the orbital integral of the characteristic function of a
larger compact open subgroup $L_s$ which is generated by a principal congruence subgroup
and a parahoric subgroup of a proper twisted Levi subgroup.
We conclude the proof in Section~\ref{s:proof} based on a
detailed analysis of Moy-Prasad subgroups.

\subsection{Prescribed Steinberg representations}
In a direction somewhat orthogonal to our main results described above, we have developed the case of families with prescribed Steinberg representations.
We let the group $G$ and the finite sets $S$, $S_0$ of finite places be as before (\S\ref{s:quantitative}). Let $\mathrm{St}_S$ be the Steinberg representation of $G(F_S)$. We consider the multi-set $\cF(\xi,\mathrm{St}_S,K_{S_0})$ of discrete automorphic representations $\pi$, counted with multiplicity $m(\pi)\dim(\pi_{S_0})^{K_{S_0}}$ such that $\pi_\infty\in \Pi_\infty(\xi)$, $\pi_S\simeq \St_S$, and  $(\pi^\fkS)^{K^\fkS}\neq 0$.
We let $S$ vary and refer to  $\cF(\xi,\mathrm{St}_S,K_{S_0})$ as an horizontal family. This is to be compared with the previous vertical families $\cF(\xi,\sigma,K_{S_0})$ where $S$ was fixed and $\sigma$ was a varying supercuspidal representation of $G(F_S)$.

In this case the needed estimates for orbital integrals can be deduced from results of Kottwitz on Euler-Poincar\'e functions~\cite{Kot88}. %
We establish the Sato-Tate equidistribution for these horizontal families. The main point is that our method~\cite{ST:Sato-Tate} described above for vertical families applies almost without change to these horizontal families, but with the simplification that the rather subtle bound on orbital integrals from Sections~\ref{s:orbital-integrals} and \ref{s:proof} is replaced with easier bounds such as~\eqref{e:St-case-2} below.
In the following example we explain the significance of the result for classical modular forms and refer to Theorem~\ref{th:steinbergs} for the precise statement in general.

\begin{ex}\label{ex:ILS} Consider again the group $\PGL(2)$. The Steinberg representation
$\mathrm{St}_q$ and the quadratic twist of the Steinberg representation
$\mathrm{St}_q\otimes \chi_q$ are the two representations of $\PGL(2,\Q_q)$ of level
$\Gamma_0(q)$, thus
\[
m(\mathcal{D}_k,\mathrm{St_q}) + m(\mathcal{D}_k,\mathrm{St_q}\otimes \chi_q) = \dim
S_k(q)^{\mathrm{new}} = \dim S_k(q) - 2 \dim S_k(1). \]
We note that $\cF(\mathcal{D}_k,\mathrm{St}_q)$ (resp. $\cF(\mathcal{D}_k,\mathrm{St}_q\otimes \chi_q)$) is the set of cuspidal modular forms of weight $k$, level $\Gamma_0(q)$ with global root number $1$ (resp. $-1$), see e.g.~\cite{Cogdell2004}. Iwaniec--Luo--Sarnak~\cite[Cor. 2.14]{ILS00} proved that the following asymptotic holds:
\[
m(\mathcal{D}_k,\mathrm{St_q})\sim \frac{k-1}{24}q, \quad \text{as $k,q\to \infty$}
\]
and similarly for $ m(\mathcal{D}_k,\mathrm{St_q}\otimes \chi_q)$. We shall discuss this in a more general context in Section~\ref{s:Steinberg} and revisit the $\PGL(2)$ case in Example~\ref{ex:PGL2}.
\end{ex}

\subsection{Notation}

  Let $F$ be a number field, $S$ (any) finite set of places of $F$, and $S_\infty$ the set of all infinite places of $F$. Then set $F_S:=\prod_{v\in S} F_v$, $F_\infty:=F\otimes_\Q \R$, $\A^S_F:=\prod'_{v\notin S} F_v$, and $\A^{S,\infty}_F:=\prod'_{v\notin S\cup S_\infty} F_v$, where $\prod'$ denotes the restricted product over all places $v$ under the given constraint. Now let $G$ be a connected reductive group over $F$. The center of $G$ is denoted $Z(G)$, the maximal $\Q$-split torus in
the center of $\Res_{F/\Q} G$ is $A_G$, and $A_{G,\infty}:=A_G(\R)^0$. Write $G_\infty$ for $\Res_{F_\infty/\R} (G\times_F F_\infty)$.

   Let $\cH(G(\A_F^{S}))=C^\infty_c(G(\A_F^{S}))$
    denote the space of locally constant compactly supported $\C$-valued functions on $G(\A_F^{S})$. Similarly $\cH(G(F_S))$ is defined. The unitary dual of $G(F_S)$ is denoted $G(F_S)^{\wedge}$. Its Plancherel measure is written as $\pl_S$. We typically write $\phi_S$ for an element of $\cH(G(F_S))$ and $\hat \phi_S$ for the associated function $\pi_S\mapsto \tr \pi_S(\phi_S)$ on $G(F_S)^\wedge$. %

  When $\pi$ is an admissible representation of a $p$-adic group $G$, write $\Theta_\pi$ for its Harish-Chandra character.
We write $[g,h]:=ghg^{-1}h^{-1}$ for $g,h\in G$.

\subsection{Acknowledgment}
We thank late Paul Sally for helpful discussions. We are grateful to the referees for their numerous corrections and suggestions.
S.W.S. is partially supported by NSF grant DMS-1449558/1501882 and a Sloan Fellowship. N.T. is partially supported by NSF grant DMS-1200684/1454893.

\section{Yu's construction of supercuspidal representations}\label{s:Yu-construction}

In this section we review the construction of supercuspidal representations of a $p$-adic reductive group from the so-called generic
data due to Jiu-Kang Yu and recall from~\cite{Kim07} that his construction
exhausts all supercuspidal representations provided the residue characteristic of the
base field is sufficiently large. The construction yields a supercuspidal representation
concretely as a compactly induced representation, and this will be an important input in
the next section.

\subsection{Notation and definitions}\label{sub:notation Bd}

  The following local notation will be in use until Section~\ref{s:proof}.
  Let $p$ be a prime. Let $k$ be a finite extension of $\Q_p$. Denote by $q$ the cardinality of the residue field of $k$.
  Let $\bG$ be a connected reductive group over $k$, whose Lie algebra is denoted
  $\blieG$. Denote the center of $\bG$ by $\mathbf Z_{\bG}$. Write $G$ and $\lieG$ for
  $\bG(k)$ and $\blieG(k)$, respectively.
For a tamely ramified extension $E$ of $k$,
denote by $\Bd(\bG,E)$ the extended building of $\bG$ over $E$.
If $\bT$ is a maximal $E$-split $k$-torus, let
$\Apt(\bT,\bG,E)$ denote the apartment associated to $\bT$ in $\Bd(\bG,E)$.
It is known that for any tamely ramified
 Galois extension $E'$ of $E$,
$\Apt(\bT,\bG,E)$ can be identified with
the set of all $\Gal(E'/E)$-fixed points in
$\Apt(\bT,\bG,E')$. Likewise, $\Bd(\bG,E)$ can be embedded into $\Bd(\bG,E')$
and its image is equal to the set of the Galois fixed points in $\Bd(\bG,E')$
\cite{Rou77, Pra01}.

For $(x,r)\in\Bd(\bG,E)\times\bbR$, there is
a filtration lattice $\blieG(E)_{x,r}$ and a subgroup $\bG(E)_{x,r}$ if
$\rtm\ge0$ defined by Moy--Prasad~\cite{MP94}.
We shall normalize the valuation of $E$ to extend the valuation of $k$, which is
a different convention than in~\cite{MP94} (where the valued group is normalized
to be $\Z$).
Our convention is so that
for a tamely ramified Galois extension $E'$ of $E$
and $x\in\Bd(\bG,E)\subset\Bd(\bG,E')$, we have \cite[Prop. 1.4.1]{Adl98}:
\[
\blieG(E)_{x,r}=\blieG(E')_{x,r}\cap\blieG(E).
\]
If $r>0$, we also have
\[
\bG(E)_{x,r}=\bG(E')_{x,r}\cap\bG(E).
\]
For simplicity, we put $\lieG_{x,\rtm}:=\blieG(k)_{x,\rtm}$, etc,
and $\Bd(G):=\Bd(\bG,k)$.
For $\rtm\in\bbR$ and $x\in\Bd(G)$ we will also use the following notation:
\begin{itemize}
\item
$\lieG_{x,r^+}:=\cup_{s>r} \lieG_{x,s}$, and if $r\ge 0$,
$G_{x,r^+}:=\cup_{s>r} G_{x,s}$.
\item
$\lieG^\ast_{x,r}:=\left\{\chi\in\lieG^\ast
\mid\chi(\lieG_{x,(-r)^+})\subset\pid_k\right\}$, where $\pid_k$ is the maximal ideal of the ring of integers of $k$.
\item
$\lieG_r:=\cup_{y\in\Bd(G)} \lieG_{y,r}$ and
$\lieG_{r^+}:=\cup_{s>r} \lieG_s$
\item
$G_r:=\cup_{y\in\Bd(G)} G_{y,r}$ and
$G_{r^+}:=\cup_{s>r} G_s$ for $r\ge0$.
\end{itemize}
Lastly, for $x\in\Bd(G)$, we denote the stabilizer of $x$ in $G$ by $G_{[x]}$.

\subsection{Generic $G$-datum}

Yu's construction of supercuspidal representations starts with a {\it generic $G$-datum}, which we recall. The reader is referred to \cite{Yu01} for further details and any notions undefined here.

\begin{defn}\label{defn: generic G-datum}\rm
A {\it generic $G$-datum} is a quintuple
$\datum=(\vec\bG,\xo,\vec\rtm,\vec\phi,\rho)$ satisfying the following:

\item{${\mathbf{D}}1.$}
$\vec{\bG}=(\bG^0,\bG^1,\cdots,\bG^d=\bG)$ is a tamely ramified twisted
Levi sequence such that $\bZ_{\bG^0}/\bZ_{\bG}$ is anisotropic.

\midvsp

\item{${\mathbf{D}}2.$}
$\xo\in \Bd(G^0) = \Bd(\bG^0,k)$.

\midvsp

\item{${\mathbf{D}}3.$}
$\vec\rtm=(\rtm_0,\rtm_1,\cdots,\rtm_{d-1},\rtm_d)$ is
a sequence of positive real numbers
with $0<\rtm_0<\cdots<\rtm_{d-2}< \rtm_{d-1}\le\rtm_d$ if $d>0$,
$0\le\rtm_0$ if $d=0$.

\midvsp

\item{${\mathbf{D}}4.$}
$\vec\phi=(\phi_0,\cdots,\phi_d)$ is a sequence of quasi-characters,
where $\phi_i$ is a generic quasi-character of $G^i$ (see \cite[\S9]{Yu01} for the definition
of generic quasi-characters). When $d=0$, $\phi_0$ is trivial on $G_{\xo,\rtm_0^+}$, but, nontrivial on $G_{x,\rtm_0}$. When $d\ge1$,
$\phi_i$ is trivial on $G^i_{\xo,\rtm_i^+}$, but
non-trivial on $G^i_{\xo,\rtm_i}$ for $0\le i\le d-1$.
If $\rtm_{d-1}<\rtm_d$,
$\phi_d$ is nontrivial on $G^d_{\xo,\rtm_d}$ and trivial on
$G^d_{\xo,\rtm{}^+_d}$; otherwise, $\phi_d=1$.

\midvsp

\item{${\mathbf{D}}5.$}
$\rho$ is an irreducible representation of $G^0_{[\xo]}$,
the stabilizer in $G^0$ of the image $[\xo]$ of $\xo$
in the reduced building
of $\bG^0$,
such that $\rho|G^0_{\xo,0^+}$ is isotrivial
and $c\textrm{-Ind}_{G^0_{[\xo]}}^{G^0}\rho$ is irreducible and supercuspidal.
\end{defn}

In $\mathbf{D}5$, note that $G^0_x$ is compact while $G^0_{[x]}$ is only compact mod center.
Recall from \cite[p.585]{Yu01} that there is a canonical sequence of embeddings
\[
\Bd(\bG^0,E)\hookrightarrow\Bd(\bG^1,E)\hookrightarrow
\cdots\hookrightarrow\Bd(\bG^d,E).\]
Hence, $\xo$ can be regarded as a point of each of $\Bd(G^i)=\Bd(\bG^i,k)$.

Also $\mathbf{D}5$ implies that $\xo$ is rational as a building point of $\Bd(G)$ because
it is a vertex of $\Bd(G^0)$. This will become important in Hypothesis $\scrE$.(ii) below.

Given a generic $G$-datum $\datum=(\vec\bG,\xo,\vec\rtm,\vec\phi,\rho)$, we introduce an
open compact-mod-center subgroup of $G$
$$J_\Sigma:= G^0_{[\xo]}G^1_{\xo,s_0}\cdots G^{d-1}_{\xo,s_{d-2}}G^{d}_{\xo,s_{d-1}},$$
where we set $s_i:=r_i/2$ for each $i$.
Yu constructs a finite dimensional representation $\rho_\Sigma$ of $J_\Sigma$ from the datum. His key result is that

\begin{thm} [Yu]\label{t:Yu}
$\pi_{\datum}=c\textrm{-}\mathrm{Ind}_{J_\Sigma}^G\rho_{\datum}$
is irreducible and thus supercuspidal.
\end{thm}

Fix a positive Haar measure $\vol_G$ on $G$ and denote the formal degree of $\pi_\Sigma$ by $\deg(\pi_\Sigma)$.
It is well-known (see for example \cite{CMS90} or  \cite[Lem
2.9.(i)]{KST-localconstancy}) that we have
\begin{equation}\label{bound-on-d}
\deg(\pi_\Sigma)=\dfrac{\dim(\rho_\Sigma)}{\vol_{G/Z}(J_\Sigma/Z)}.
\end{equation}
The construction of $\rho_\Sigma$ is complicated, but in what follows we shall only
need the inequality $\dim(\rho_\Sigma)\le q^{\dim G}$, see~\cite[Lem
2.9.(ii)]{KST-localconstancy} for a proof.
 For later reference, we write $s_\Sigma:=s_{d-1}$.

\subsection{Supercuspidal representations via compact induction}\label{sub:sc-cpt-ind}

Denote by $\Irr(G)$ the set of (isomorphism classes of) irreducible smooth representations of $G$. Write $\Irr^2(G)$ (resp. $\Irr^\scusp(G)$) for the subset of square-integrable (resp. supercuspidal) members.
  Define $\Irr^{\Yu}(G)\subset \Irr^\scusp(G)$ to be the subset of all supercuspidal representations which are of the form $\pi_\Sigma$ as above. Write $\Irr^{\cind}(G)\subset \Irr^\scusp(G)$ for the set of $\pi$ compactly induced from a representation on an open compact-mod-center subgroup of $G$. We have that
  $$\Irr^{\Yu}(G) \subset \Irr^{\cind}(G) \subset \Irr^\scusp(G),$$
  where the first inclusion comes from Theorem \ref{t:Yu}. The second inclusion is expected to be an equality but not known in general; see \cite[\S2.6]{KST-localconstancy} for references to partial results by Bushnell, Kutzko, Stevens and others in this direction.
  The main result of \cite{Kim07} says that the above inclusions are equalities under a
  rather explicit set of four hypotheses (namely (H$k$), (HB), (HGT), and
  (H$\mathcal{N}$) in \cite[\S3.4]{Kim07}); in particular the equalities hold when $p$ is
  greater than some lower bound depending only on the absolute root datum of $G$ and the
  absolute ramification index of $k$.\footnote{While this paper was under review, Fintzen \cite{Fin18b} announced the proof that $\Irr^{\Yu}(G)= \Irr^\scusp(G)$ only assuming that $p$ does not divide the order of the Weyl group of $G$.}

\section{Orbital integrals of pseudo-coefficients}\label{s:orbital-integrals}

We keep the notation
  from the last
  section and assume that $G=\bG(k)$ has \emph{compact} center throughout this section
  and the next section. (We will briefly explain how to carry over the results of the
  current section to the non compact center case in Remark \ref{r:non-cpt-center} below.)
  For $\pi\in \Irr^{\Yu}(G)$
  attached to a generic $G$-datum we will construct an explicit coefficient $f_\pi$ of
  $\pi$ and study the asymptotic behavior of the orbital integral of $f_\pi$ on
  noncentral semisimple elements as $\deg(\pi)\ra\infty$ (note that we use $\pi$ instead
  of $\sigma$ to denote a representation of $G$). The result admits an interpretation as
  an asymptotic formula for character values, cf. \S\ref{s:asymptotic-behavior-sc} below,
  and will be applied in \S\ref{s:plancherel} to obtain an equidistribution theorem for
  families of automorphic representations.

\subsection{Pseudo-coefficients}

As before we have $G=\bG(k)$ and write $Z$ for the center of $G$.
Let us recall the definition of pseudo-coefficients, cf. \cite[A.4]{DKV84}.  %
\begin{defn}\label{defn: pc}
  Let $\pi\in \Irr^2(G)$. A function $f_\pi\in \cH(G)$ is said to be a
  \textbf{pseudo-coefficient} of $\pi$ if $\tr \pi' (f_\pi)=\delta_{\pi,\pi'}$ for every
  tempered $\pi'\in \Irr(G)$.%
\end{defn}

The existence of $f_\pi$ follows from the trace Paley-Wiener theorem, cf. \cite[Prop
1]{Clo86}. (For real groups this is due to Clozel and Delorme \cite[Cor, p.213]{CD90}.)
 To make $f_\pi$ explicit, one can employ Bruhat-Tits buildings as in \cite[\S
 III.4]{SS97} for any $\pi\in \Irr^2(G)$ or proceed as in Lemma \ref{l:sc-coeff} below
 for $\pi\in \Irr^{\cind}(G)$.

Although $f_\pi$ is not unique, the orbital integrals of $f_\pi$ and the trace values of
$f_\pi$ against irreducible admissible representations of $G$ are uniquely determined by
the condition of Definition \ref{defn: pc}. So,  by \cite[Thm. 0]{Kaz86}, its orbital
integrals are
uniquely determined. Note that $f_\pi$ is a cuspidal function in the sense that the trace
of every induced representation from a proper parabolic subgroup is zero against $f_\pi$.
(This fact is built into the construction of \cite[Prop 1]{Clo86}.)
Moreover the orbital integrals of $f_\pi$ are well known to encode the elliptic character
values of $\pi$ (recall that $\gamma\in G$ is said to be elliptic if the centralizer
$Z_G(\gamma)$
is compact). When $\gamma$ is regular elliptic, we will use the Haar measure on the
compact group $G_\gamma$ assigning total volume 1 in the definition of the orbital
integral below.

\begin{prop}\label{p:char-orb-int} If $\gamma\in G$ is regular semisimple, we have
  \beq\label{e:char-orb-int}
   \orbi_\gamma(f_\pi)=\left\{\begin{array}{cl}
     \Theta_{\pi^\vee}(\gamma), & ~\gamma:~\mbox{elliptic},\\
     0,& ~\gamma:~\mbox{non-elliptic}.
   \end{array}   \right.
\eeq
  Moreover $\orbi_\gamma(f_\pi)=0$ for every $\gamma$ that is (non-regular) non-elliptic
  semisimple.
\end{prop}

\begin{proof}
  The first assertion can be derived from \cite[Thm 5.1]{Art93} specialized to the $M=G$
  case. %
 The last assertion is Lemma III.4.19 of \cite{SS97} (noting that the Euler-Poincar\'e
 function in that lemma is a pseudo-coefficient in view of Proposition III.4.1 and
 Theorem III.4.6 of \textit{loc. cit.}).

\end{proof}

\subsection{Explicit supercuspidal coefficients}\label{sub:explicit-coeff}

In the following lemma, we construct an explicit matrix coefficient (which is also a
pseudo-coefficient) associated to a compactly induced supercuspidal representation.

\begin{lem}\label{l:sc-coeff} Let $\rho$ be a finite dimensional admissible
representation of an open compact subgroup $J$ of $G$. Suppose $\pi:=\cind_J^G\rho$ is
irreducible (thus supercuspidal). %
Let
\[
f_\pi(g)=\begin{cases}\frac1{\vol(J)}\Theta_{\rho^\vee}(g) &\textrm{if }g\in J\\
0&\textrm{otherwise}\end{cases}
\]
where $\Theta_{\rho^\vee}$ is the character of $\rho^\vee$. Then, we have
\begin{enumerate}
\item[(i)]
$\tr \pi\rq{} (f_{\pi})=\delta_{\pi,\pi\rq{}}$ for every $\pi'\in \Irr(G)$,
\item [(ii)]
$f_\pi(1)=\frac{\dim(\rho^\vee)}{\vol(J)}=\frac{\dim(\rho)}{\vol(J)}=\deg(\pi)$,
\item[(iii)]
$\supp(f_\pi)\subset J$.
\end{enumerate}
In particular, $f_\pi$ is a pseudo-coefficient and $\dfrac{|f_\pi|}{\deg(\pi)}\le  1_J$,
the characteristic function of $J$.
\end{lem}

\begin{proof}
Assertion (iii) is immediate, and (ii) is~\cite[Lem~2.9]{KST-localconstancy}. Assertion
(i) follows from the fact that $f_\pi$ is a matrix coefficient of $\pi^\vee$ and from
Frobenius reciprocity which implies $\Hom_J(\rho,\pi\rq{})=\delta_{\pi,\pi\rq{}}$.
\end{proof}

We will need the following hypotheses on the group $\bG$ for Theorem
\ref{t:asymptotic-orb-int}.
Hypothesis $\scrT$ will be used in Lemma~\ref{lem: good product}.
Hypothesis $\scrE$ will be used in the proof of Proposition~\ref{lem: decreasing} and
Lemma~\ref{lem: normalize}.
Observe that the hypotheses are inherited by tame twisted Levi subgroups.

\subsection{Hypotheses}\label{hypotheses}
\begin{itemize}
\item[$\scrT$]
For any tame maximal torus $\bT$, and $r>0$, every nontrivial coset in $T_r/T_{r^+}$
contains a $\bG$-good element.
\item[$\scrE$]
There is a tamely ramified extension $E$ of $k$ such that
\begin{itemize}
\item[(i)]
every $k$-torus in $\bG$ splits over $E$;
\item[(ii)]
the $k$-\emph{order}, in the sense of~\cite[\S3.3]{RY14}, of every
building point $\xo\in \Bd(G)$, which is a vertex in the building of some tamely ramified
twisted
Levi subgroup of $\bG$
divides the ramification index of $E$ over $k$;
\item[(iii)]
for every $r\ge1$, the exponential map induces a homeomorphism
\[
\exp:\blieG(E)_{x,r}\rightarrow
\bG(E)_{x,r},
\]
and an abelian group homomorphism
$\bG(E)_{x,r}\left/\bG(E)_{x,r^+}\right.\simeq
\blieG(E)_{x,r}\left/\blieG(E)_{x,r^+}\right.$.
\end{itemize}
\end{itemize}

We recall the notion~\cite{AS08} of $\bG$-\emph{good} elements (we shall simply write
{\it good} when there is no risk of confusion). Define the
\emph{depth} $\dth_\bT(\gamma)$ of a compact element $\gamma\in T_0$  as the
unique $r\in \R_{\ge0}$ such that $\gamma\in T_r\setminus T_{r^+}$.
A compact element $\gamma\in T_0$ of positive depth $r>0$ is $\bG$-good  if
for each $\alpha\in\Phi$, either $\alpha(\gamma)=1$ or
$\val(\alpha(\gamma)-1)=r$. A compact semisimple element $\gamma$ of depth $0$ is
$\bG$-good if it is
absolutely semisimple.

Hypothesis $\scrE$ may appear complicated at first. In fact, it achieves
different purposes simultaneously:
\begin{itemize}
\item If the tame extension $E$ of $k$ satisfies  $\scrE$.(i), (resp. $\scrE$.(ii)), and
$E'/E$ is
a larger tame extension, then $E'$ also satisfies $\scrE$.(i), (resp. $\scrE$.(ii)),
because the ramification index of $E'$ over $k$ is divisible by the ramification index of
$E$ over $k$.
\item On the other hand, the radius of convergence of the exponential map in
$\scrE$.(iii) decreases with the
ramification index. In particular,
Hypothesis $\scrE$.(iii) may be satisfied for the extension $E$ of $k$, and
fail for larger extensions $E'/E$.
\end{itemize}

To show that Hypotheses $\scrT$ and $\scrE$ hold true for sufficiently large $p$, we begin with a few general group-theoretic lemmas.

\begin{lem}\label{l:split}
  Fix a root datum $\mathcal R$. Then there exists a constant $d_{\mathcal R} >0$
  such that for every field $K_0$ and every connected
  reductive
  group $\bG$ over $K_0$ with absolute root datum $\mathcal R$, the group $\bG$ splits over
  an
  extension of $K_0$ with degree at most $d_{\mathcal R}$.
  In fact every maximal torus of $\bG$ defined over $K_0$ splits over an extension with degree at most $d_{\mathcal R}$.
\end{lem}

\begin{proof}
Let $\bT$ be a maximal torus of $\bG$ defined over $K_0$. Then $\bT$ splits over a finite separable extension $K/K_0$
such that $\Gal(K/K_0)$ acts faithfully on $X^*(\bT)$. Choose a $\Z$-basis for $X^*(\bT)$ to identify $X^*(\bT)\simeq \Z^r$
so as to obtain a group embedding
$$\Gal(K/K_0)\hra \GL(r,\Z).$$
By Minkowski's
lemma (see e.g.~\cite{Serre:bounds-finite}), $\GL(r,\Z)$ has only finitely many finite
subgroups up to conjugacy (for a fixed $r$), hence $[K:K_0]$ admits a bound $d_r\in
\Z_{>0}$
only in terms of $r$. The lemma is proved with $d_{\mathcal R}:=d_r$, as $r$ depends only on
the root datum $\mathcal R$.
 \end{proof}

When $K_0$ is a local field, we have an explicit and relatively tight bound on $d_{\mathcal R}$ assuming that the residual characteristic is not too small.

\begin{lem}\label{l:tame-split}
  Fix a root datum $\mathcal R$.
  Let  $M_{\mathcal R}$ be the largest integer $m$ such that $\varphi(m)\le
  \rank(\mathcal R)$, and $p_{\mathcal R}:=\rank(\mathcal R)$.
  Then for every local field $K_0$ of residual characteristic
  $p>p_{\mathcal R}$,
  every $K_0$-torus $\bT$ in a connected
  reductive
  group $\bG$ over $K_0$ with absolute root datum $\mathcal R$, splits over a tamely
  ramified extension of $K_0$ of degree $\le M_{\mathcal R}^2$.
\end{lem}

\begin{proof}
We will deduce this from the argument for Lemma~\ref{l:split}.
In the proof of that lemma, retaining the same notation, it suffices that the order of the finite subgroup
$\Gal(K/K_0)\hookrightarrow \GL(r,\Z)$ be coprime to $p$ for $K/K_0$ to be tame.
By Minkowski's lemma~\cite{Serre:bounds-finite}, this holds if $p>r+1$.
Moreover, the exponent $m$ of every finite subgroup of $\GL(r,\Z)$ satisfies
$\varphi(m)\le
r$, because the eigenvalues of an element of order $m$ inside $\GL(r,\Z)$ consist of all
$m$-th roots of unity (see~\cite[App.~B, claim (1) p.256]{Wald:tordue-asterisque} for a
similar argument). So $K/K_0$ is a finite tamely ramified extension of exponent $m$ with
$m\le M_{\mathcal R}$.
If $K'$ denotes the maximal unramified subextension of $K/K_0$ then $K/K'$ and $K'/K_0$ are cyclic so $[K:K']$ and $[K':K_0]$ divide $m$.
Therefore $[K:K_0]\le m^2$.
\end{proof}

\begin{lem}\label{l:k-embed}
  Fix a root datum $\mathcal R$. Then there exists a constant $N_{\mathcal R} >0$
  such that for every field $K_0$ of characteristic zero and every connected reductive
  group $\bG$ over $K_0$ with absolute root datum $\mathcal R$, there exists a
  $K_0$-embedding
  $\bG\hra GL(n)$ with $n\le N_{\mathcal R}$.
\end{lem}

\begin{proof}
  Denote by $\bG_0$ the Chevalley group over $k_0:=\Q$ determined by $\mathcal R$.
  Fix a $k_0$-embedding $\bG_0\hra \GL_{k_0}(V_0)$ with $V_0$ a finite dimensional
  space over $k_0$ once and for all. We will prove the lemma with $N_{\mathcal R}:=d_{\mathcal R}\dim_{{k_0}}V_0$
  with $d_{\mathcal R}$ as in Lemma \ref{l:split}.

  Let $\bG$ and $K_0$ be as in the current lemma.
  By Lemma \ref{l:split}, there exists $K/K_0$ with $[K:K_0]\le d_{\mathcal R}$ such that
  $\bG\otimes_{K_0} K$ is a split group, so that $\bG_0\otimes_{k_0} K\simeq \bG\otimes_{K_0} K$.
  By base change, we have a $K$-embedding $\bG_0\otimes_{k_0} K\hra
  \GL_K(V_0\otimes_{k_0} K)$.
  Let $\mathrm{Res}_{K/K_0}$ denote the Weil restriction of scalars with respect to $K/K_0$.
  Using the obvious embedding $\bG\hra \Res_{K/K_0}(\bG\otimes_{K_0} K)$, we obtain a chain of
  $K_0$-embeddings
  $$\bG\hra \Res_{K/K_0}(\bG\otimes_{K_0} K)\simeq \Res_{K/K_0}(\bG_0\otimes_{k_0} K)\hra
  \Res_{K/K_0}\GL_K(V_0\otimes_{k_0} K)\hra \GL_k (V_0\otimes_{k_0} K).$$
  The last embedding follows from viewing $V_0\otimes_{k_0} K$ as a $K_0$-vector space. Since
  $\dim_{K_0} (V_0\otimes_{k_0} K) = [K:K_0]\dim_{{k_0}}V_0 \le d_{\mathcal R}\dim_{{k_0}}V_0=N_{\mathcal R}$, the
  proof is complete.
\end{proof}

\begin{remark} For certain kinds of groups $\bG$, one can
proceed more directly as follows:
\begin{itemize}
\item If $\bG$ is semisimple, then we have
$\operatorname{Aut}(\bG)\simeq \operatorname{Inn}(\bG)\rtimes
\operatorname{Aut}({\mathcal R})$, and the action of
$\Gal(\overline{K_0}/K_0) $ on $\bT$ factors through $W\rtimes \operatorname{Aut}({\mathcal
R})$. Thus it is sufficient that $p_{\mathcal R}$ be larger than every prime factor of
the Weyl group $W$ and the automorphism group $\operatorname{Aut}({\mathcal R})$.

\item If $\bG$ is already split over a tamely ramified extension of $K_0$, then pick one
tame $K_0$-torus $\bT$. To ensure that all the other $K_0$-tori are tame, it suffices that
the image of $H^1(K_0,N_{\bG}(\bT))\to
\operatorname{Hom}(\Gal(\overline{K_0}/K_0),W)$ consists of elements of
order coprime to
$p$,
which happens if $p_{\mathcal R}$ is
larger than every prime factor of $W$.
\end{itemize}
\end{remark}

\begin{prop}\label{p:hypotheses}
Hypotheses $\scrE$ and $\scrT$ are satisfied if $p$ is
sufficiently large, depending only on the absolute root datum of $\bG$ and the absolute
ramification index of $k$ over $\Q_p$.
\end{prop}

\begin{proof}
Let $\mathcal{R}$ be the absolute root datum of $\bG$.
Assuming $p>p_{\mathcal R}$, choose a tamely ramified extension $K/k$ with $[K:k]\le
M^2_{\mathcal R}$ as in Lemma~\ref{l:tame-split}.
We shall construct $E$ as an extension of
$K$, thus $\scrE$.(i) will be automatically satisfied.

Concerning Hypothesis $\scrE$.(ii), the key is to observe that every $x\in
\Bd(\bG,k)$ contained in a generic datum is a \emph{vertex} of the sub-building
$\Bd(\bG^0,k)$ attached to a tamely ramified twisted Levi subgroup $\bG^0\subset \bG$.
Let $\bT^0$ be a maximal $k$-torus in $\bG^0$ so that $x\in \Bd(\bG^0,k)\cap \Apt(\bG,\bT^0,K)$.
In particular $x$ is a $\Gal(K/k)$-fixed point of $\mathcal{A}(\bG,\bT^0,K)$.
Let $\bS^0\subset \bT^0$ be the maximal $k$-split subtorus of $\bT^0$.

Similarly as in~\cite[\S2.6]{RY14}, if we fix an alcove and consider the finitely many affine roots $\psi_i$ whose zero loci bound
an alcove and which are positive valued on the alcove, then there is a linear relation
$\sum_i
b_i \psi_i=1$ with $b_i\in \Z_{>0}$. Let $\ell_{\psi_i}$ be defined as in \cite[3.1]{MP94}. Put $b_{\bG}:=\mathrm{l.c.m}(\ell_{\psi_i}b_i)$.
Since $x$ is a vertex of $\Bd(\bG^0,k)$, we have $\psi(x)\in \frac{1}{b_\bG}\Z$ for
every affine
$k$-root $\psi\in
\Psi(\bG^0,\bS^0,k)$.
We view $\Psi(\bG^0,\bS^0,k)$, resp. $\Psi(\bG,\bT^0,K)$, as affine linear functions on
$X_*(\bS^0,k)\otimes_\Z \R$, resp. $X_*(\bT^0,K)\otimes_\Z \R$.
We need to recall the details of how $\Bd(\bG^0,k)$ injects in $\Bd(\bG,k)$, see p.585,
Rem.2.11, Rem.3.4 of~\cite{Yu01}.
The affine isomorphism $\mathcal{A}(\bG^0,\bT^0,K)\simeq \mathcal{A}(\bG,\bT^0,K)$ is
canonically defined only up to translation by $X_*(\bZ_{\bG^0},K)\otimes_\Z \R$.
The choice in~\cite{Yu01} is made essentially in such a way that the isomorphism be
$\Gal(K/k)$-equivariant.
Recall that $\bZ_{\bG^0}$ is anisotropic because $\bZ_\bG$ is anisotropic. Hence the
origin is the only $\Gal(K/k)$-fixed point in
$X_*(\bZ_{\bG^0},K)\otimes_\Z \R$, and the isomorphism is uniquely determined.
To rigidify the situation, we identify
the apartments $\mathcal{A}(\bG^0,\bS^0,k)$ and $\mathcal{A}(\bG,\bT^0,K)$ with
$X_*(\bS^0,k)\otimes_\Z\R$ and $X_*(\bT^0,K)\otimes_\Z\R$, respectively.
Since the center
$\bZ_{\bG^0}$ is anisotropic, we have that the $\Q$-span of $\Psi(\bG^0,\bS^0,k)$ is equal
to the space of all affine linear functions on $X_*(\bS^0,k)\otimes_\Z \Q$.
There is a positive integer $e$ such that the $\Z$-span of
$\frac{1}{e}\Psi(\bG^0,\bS^0,k)$
contains the image of $\Psi(\bG,\bT^0,K)$ under restriction to
$X_*(\bS^0,k)\otimes_\Z \Q$.
We have $\psi(x)\in \frac{1}{b_\bG e}\Z$ for every affine root $\psi \in \Psi(\bG,\bT^0,K)$.
Since the integers $b_{\bG}$ and $e$ depend only on the relative root system attached to
$(\bG^0,\bS^0,k)$, the absolute root system $\mathcal{R}$ attached to
$(\bG,\bT^0)$, and the action of $\Gal(K/k)$, and since there are finitely many root
systems of given rank (possibly
non-reduced,
and possibly reducible), we have $b_{\bG}e | e_{\mathcal R}$ for
some positive integer constant $e_{\mathcal R}$ depending only on $\mathcal R$.
If $p\nmid e_{\mathcal R}$, then a totally tamely ramified extension $E/K$ of degree
$e_{\mathcal R}$ will satisfy $\scrE$.(ii).
\begin{comment}
{\color{red} Let $\bT^0$ be a maximal $k$-torus in $\bG^0$ so that $x\in \Bd(\bG^0,k)\cap \Apt(\bG,\bT^0,K)$.
Let $\bS^0\subset \bT^0$ be a maximal $k$-split subtorus of $\bT^0$.
If $k$-rank of $\bS^0$ is 0, $\bT^0$ is $k$-anisotropic and $x$ is the unique Galois invariant vertex in $\Apt(\bG,\bT^0,K)$. Set $b$ to be the l.c.m. of the vertices in $\Apt(\bG,\bT^0,K)$. We also set $\tilde e=1$.}
\end{comment}

For $\scrE$.(iii),
one needs to show that for $X, Y\in\blieG(E)_{x,r}$,
\[\tag{$\ast$}
\log(\exp(X)\exp(Y))\equiv X+Y\pmod{\blieG(E)_{x,r^+}}.
\]
The Campbell-Hausdorff formula in the form given by Dynkin is
\[
\log(\exp(X)\exp(Y))=\sum_{d=1}^\infty \frac{(-1)^{d+1}}{d} Z_d = X + Y +
\sum_{d=2}^\infty \frac{(-1)^{d+1}}{d} Z_d,
\]
with
\[
Z_d=
\sum_{\substack{s_i+t_i\ge 1\\ s_d+t_d=1 }}
\frac{\ad(X)^{s_1}\ad(Y)^{t_1}\cdots\ad(X)^{s_{d-1}}\ad(Y)^{t_{d-1}}(Y)
(X^{s_d}Y^{t_d})}
{\sum^d_{i=1}(s_i+t_i)
\cdot
\prod^d_{i=1} r_i! s_i!
},
\]
where the convention is that $X^{s_d}Y^{t_d}$ is equal to $X$ if $s_d=1,t_d=0$, and is
equal to $Y$ if
$s_d=0,t_d=1$.

Since $[\blieG(E)_{x, a},\blieG(E)_{x, b}]\subset \blieG(E)_{x, a+b}$,
each summand in $Z_d$ is in
$\blieG(E)_{x, r'}$ where
\[
r'=r\sum\nolimits_i(s_i+t_i)-\ord_E(d)-\ord_E\left(\sum\nolimits_i(s_i+t_i)\right)
-\sum\nolimits_i\ord_E(s_i!)
-\sum\nolimits_i\ord_E(t_i!).
\]
By Legendre formula
\[
\ord_E(s_i!) = \frac{\ord_E(p)}{p-1} (s_i - \sigma) <
\frac{\ord_E(p) s_i}{p-1},
\]
where $\sigma$ is the sum of the base-$p$ expansion digits of $s_i$, we have
\[
r' > r\sum\nolimits_i(s_i+t_i)
-\ord_E(p)
\left(
\log_p(d) + \log_p\left(\sum\nolimits_i(s_i+t_i)\right)
+ \frac{\sum\nolimits_i(s_i+t_i)}{p-1}
\right).
\]
If $r>\dfrac{\ord_E(p)}{p-1}$, then
we have $r'\to \infty$ as $d\to \infty$, because
$\sum^d_{i=1}(s_i+t_i) \ge d$.
Hence the infinite sum converges.
Since $\log_p(x)\le \dfrac{x}{2\log p}$ for every $x\ge 2$,
\[
r' > \left(r- \frac{\ord_E(p)}{\log p}- \frac{1}{p-1}\right)\sum\nolimits_i(s_i+t_i).
\]
Thus $\dfrac{\ord_E(p)}{\log p}+\dfrac1{p-1}\le \dfrac12 \le \dfrac{r}{2}$ implies $r'>r$,
and that $(\ast)$ is satisfied.

Consider again a totally tamely ramified extension $E/K$ of degree $e_{\mathcal R}$.
Then $\ord_E(p)=e_{\mathcal R} \ord_K(p)\le e_{\mathcal R} M_{\mathcal R}\ord_k(p)$.
The conclusion is that if
\[\tag{$\ast\ast$}
p > p_{\mathcal R},\quad\ p\nmid e_{\mathcal R}, \quad\mathrm{ and } \quad
\frac{e_{\mathcal R} M_{\mathcal R}\ord_k(p)}{\log p}+\frac1{p-1}\le \frac12,
\]
then the extension $E$ satisfies the second statement of Hypothesis $\scrE$.(iii).
We observe that the above calculation shows that $\exp(\blieG(E)_{x,r})$ is a group when
$p$ satisfies $(\ast\ast)$.

We finally need to verify the first statement of Hypothesis $\scrE$.(iii).
Under the condition that $p>\dim(G)\ord_E(p)+1$, the map $\exp:
\blieG(E)_{0^+}\rightarrow
\bG(E)_{0^+}$ is a homeomorphism by \cite[App.B]{Wald:tordue-asterisque} (An alternative
approach would be %
to consider a faithful representation $\bG\hookrightarrow \GL_n$ for some $n$, see
also~\cite[App.B]{DeBacker-Reeder}. Note that $\blieG(E)_{0^+}$ (resp. $\bG(E)_{0^+}$) are the set of topologically nilpotent (resp. unipotent) elements since $p>\dim(G)$.)
Hence, $\exp_r=\exp|_{\blieG(E)_{x,r}}$ is injective. We also need to show that $\exp_r$
maps
onto $\bG(E)_{x,r}$.
To see this, write $\blieG(E)=\blieT^0(E)\oplus (\oplus_\phi \blieG_\phi(E))$ where
$\lieG_\phi$ is the root space of $\phi\in\Phi(\bT^0,\bG,E)$. Upon fixing a pinning, $\blieG(E)_{x,r}$ is
the
$\CaO_E$-span of $\blieG_{\phi}(\mathfrak p_E^n)$ with $\phi(x)+n/e_{E/k}\ge r$ where
$e_{E/k}$ is the ramification index of $E$ over $k$.
Likewise, $\bG(E)_{x,r}$ is generated by $\bT^0(E)_r$ and $\bU_\phi(\mathfrak p_E^n)$
where $\bU_\phi$ is the root subgroup of $\phi$. Then, $\exp_r$ takes
$\blieG_\phi(\mathfrak p_E^n)$ onto $\bU_\phi(\mathfrak p_E^n)$, and the proof of assertion
(4) on p.258 of  \cite[App.B]{Wald:tordue-asterisque}
shows that $\blieT^0(E)_r$ is exponentiated onto
$\bT^0(E)_r$. Now, if $p$ satisfies $(\ast\ast)$, then $\exp(\blieG(E)_{x,r})$ is a group
and hence equals $\bG(E)_{x,r}$.
In sum, Hypothesis $\scrE$ holds when $p$ satisfies $(\ast\ast)$ and
$p>\dim(G)e_{\mathcal R}M_{\mathcal R}\ord_k(p) +1$.

The assertion that $\scrT$ holds for $p$ sufficiently large goes back
to~\cite[\S5]{Adler-Roche:intertwinning}, which treated the Lie algebra version.
A recent treatment of the existence of good elements for the group case is
\cite[Thm~3.6]{Fin18a}, with the sharp result that Hypothesis $\scrT$ holds when $\bG$
splits over a tamely ramified extension and $p$ does not divide the order of the Weyl
group of $\bG$.
\end{proof}

\begin{rem} \
\begin{itemize}
\item By definition, a finite extension $E$ of $k$ is tame if the residue characteristic
$p$
is coprime to the ramification index of $E$ over $k$.
Thus, Hypothesis $\scrE.(ii)$ implies $p$ is coprime to the order of every $\xo$
contained in a generic datum. As we have seen in the proof, this is satisfied if $p>h$,
the Coxeter number of $\bG$,
and $\bG$ splits over a tamely ramified extension of $k$.

\item
In Hypothesis $\scrE.(iii)$, if $\bG$ is a classical group and $p\neq 2$, one can use a
Cayley map
instead
of the exponential map.
When $p$ is very good (\cite[(8.9)]{BKV}), and $\bG$ is semisimple and simply
connected,  a quasi-logarithmic map
satisfying $\scrE.(iii)$ is constructed in~\cite[Lem~C.4]{BKV}.
In general the mock exponential map introduced by Adler could be used, compare also
with~\cite[Hyp.~3.2.1]{DeBacker:homogeneity}.
\end{itemize}
\end{rem}

For convenience, we shall work under Hypotheses $\scrT$ and $\scrE$ in
Section~\ref{s:proof} below, which is devoted to prove Proposition \ref{p: power-saving}.
However, inspecting
the recursive argument in Proposition~\ref{lem: decreasing}, it is sufficient to have
Hypothesis
$\scrE.(iii)$ for large enough $r$, and this is always satisfied by the same
argument. Also we shall need in
Lemma~\ref{lem:orbital-Lie} below that the isomorphisms in $\scrE.(iii)$ be
$G$-equivariant, which is also satisfied for large enough $r$.
Hypothesis $\scrE.(ii)$ is automatically satisfied for an extension of large
enough degree.
Hence a natural assumption for our setting is:

\medskip
\noindent $(\mathscr{T}')$
Every $k$-torus $\bT$ in $\bG$ is tame, and for every large enough $r$, every nontrivial
coset in $T_r/T_{r+}$ contains a $\bG$-good element.
\medskip

More precisely, the preceding paragraph explains that the arguments in Section~\ref{s:proof} prove
Proposition \ref{p: power-saving} under Hypothesis $(\mathscr{T}')$, instead of Hypotheses $\scrT$ and $\scrE$.
This in turn implies Theorems
\ref{t:asymptotic-orb-int} and \ref{t:asymptotic-char} under Hypothesis $(\mathscr{T}')$.

\begin{rem}\label{r:assumption}
To reach an optimal assumption, one could
also factor out the center $\bZ_{\bG}$, which
shouldn't play a role in estimating orbital integrals, and therefore it shouldn't be
necessary to assume that $\bZ_{\bG}$ is tame.
\end{rem}

\subsection{A uniform bound on orbital integrals of supercuspidal
coefficients}\label{sub:asymptotic-orbital}

For a semisimple element $\gamma\in G_{\mathrm{ss}}$ let $\fkg_\gamma$ denote the Lie
algebra of the connected centralizer of $\gamma$ in $G$. Define
$$D(\gamma)=D^G(\gamma):=\left |
\det(1-\mathrm{Ad}(\gamma)|_{\fkg/\fkg_\gamma})\right|~\in \R_{>0}.$$
Note that it is unnecessary to assume $\gamma$ to be regular.
Given a generic $G$-datum $\Sigma$, Lemma~\ref{l:sc-coeff} provides us with the
pseudocoefficient $f_{\pi_\Sigma}\in \cH(G)$ coming from $J_\Sigma$ and $\rho_\Sigma$.

The following is a key local result of this paper.

\begin{thm}\label{t:asymptotic-orb-int}
Assume Hypothesis $(\mathscr{T'})$.
  There exist constants $C,\nu>0$ depending only on $G$ such that for every generic
  $G$-datum $\Sigma$,
  $$D(\gamma)^{1/2}|\orbi_\gamma(f_{\pi_\Sigma})|\le C\cdot
  \deg(\pi_\Sigma)^{1-\nu},\quad \forall \gamma\in G_{\mathrm{ss}}\backslash Z.$$
  In fact we can take any $\nu < (\dim G)^{-1}$.
\end{thm}

  In particular the theorem implies that if $\gamma\notin Z$ then
\[
\lim_{\deg(\pi)\ra\infty}
\frac{\orbi_\gamma(f_\pi)}{\orbi_1(f_\pi)}=\lim_{\deg(\pi)\ra\infty}
\frac{\orbi_\gamma(f_\pi)}{\deg(\pi)}=0,
\]
where $\pi=\pi_\Sigma$ varies in $\Irr^\Yu(G)$. If, on the contrary $\gamma\in Z$ then
clearly $\orbi_\gamma(f_\pi)/\orbi_1(f_\pi)=\omega_{\pi}^{-1}(\gamma)$, where
$\omega_\pi$ is the central character of $\pi$.
So the above limit is never zero.

\begin{rem}\label{r:disc-analog-orb-int}
  An interesting question is whether the above Theorem~\ref{t:asymptotic-orb-int} remains
  valid if $\pi$ is allowed more generally to run over $\Irr^2(G)$. We are inclined to
  believe that it is at least true for every sequence in $\Irr^{\scusp}(G)$, possibly
  with a different value of $\nu\in \R_{>0}$. %
\end{rem}

\begin{proof}[Proof of Theorem \ref{t:asymptotic-orb-int}]
The orbital integral vanishes unless $\gamma$ is elliptic, so we assume that $\gamma$ is
noncentral and elliptic semisimple from now on.
Let $G_\gamma:=Z(\gamma)$ and $\orb(\gamma)$ the $G$-orbit of $\gamma$ in $G$.

For $\Sigma=(\vec{\bG},x,\vec r,\vec\phi,\rho)$ as in Definition~\ref{defn: generic
G-datum}, we let
$$\bG^{\prime}:=\bG^{d-1},\quad G^{\prime}:=G^{d-1},\quad\mbox{and}\quad
L_s:=G_{x,s}G^{\prime}_{x}~\mbox{for}~s\in\mathbb R_{\ge 0}.$$
Since $L_{s_\Sigma}$ is an open compact subgroup of $G$ containing $J_\Sigma$, where we
recall that $s_\Sigma=r_{d-1}/2$, it follows from Lemma~\ref{l:sc-coeff} that we have the
inequalities
\begin{align}\label{eq: orb}
\frac{\left|\orbi_\gamma(f_{\pi_\Sigma})\right|}{\deg(\pi_\Sigma)}
\le\vol_{\orb(\gamma)}(J_\Sigma\cap \orb(\gamma))
\le\vol_{\orb(\gamma)}(L_{s_\Sigma}\cap \orb(\gamma))=\orbi_\gamma(\triv_{L_{s_\Sigma}}).
\end{align}

Our strategy is to study a power-saving upper bound for $\orbi_\gamma(\triv_{L_s})$ as
$\gamma$ runs over semisimple elements of $G$ and as $s\to \infty$.
Indeed as $\Sigma$ moves along a sequence of generic $G$-data such that
$\deg(\pi_\Sigma)\ra\infty$, up to conjugacy,
there are only a finite number of choices for $(\bG^{\prime},x)$ appearing in $\Sigma$
with $x\in\Bd(G^{0})$.
The crucial estimate is the following, whose proof is postponed to Section~\ref{s:proof}.
(Regarding the hypothesis, see the discussion above Remark \ref{r:assumption}.)

\begin{prop}\label{p: power-saving}
Assume Hypothesis $(\mathscr{T'})$.
There exists a constant $C_1>0$ depending only on $G$ such that for all generic $G$-datum
$\Sigma$, $s\in \R_{>0}$ and all noncentral semisimple $\gamma\in G$,
\[
\orbi_\gamma(\triv_{L_s})\le C_1\cdot (s+1) \cdot q^{- s}\cdot D(\gamma)^{-1/2}.
\]
\end{prop}

To relate $\deg(\pi_\Sigma)$ and $s_\Sigma$, we deduce the following
from~\eqref{bound-on-d} and the fact that $J_\Sigma\supset G_{x,s_\Sigma}$:
\[
\deg(\pi_\Sigma)\le q^{\dim G}\vol(J_\Sigma)^{-1}\le q^{\dim G}\vol(G_{x,s_\Sigma})^{-1}.
\]
 Since $G_{x,r}/G_{x,r+1}$ is in bijection with  $\fkg_{x,r}/\fkg_{x,r+1}$ for $r>0$
 by hypothesis $\scrE$ and $\varpi \fkg_{x,r}=\fkg_{x,r+1}$, we have
 $[G_{x,r}:G_{x,r+1}]=q^{\dim G}$.  Thus for $s\in \R_{\ge 0}$,
\[
\vol(G_{x,s})\ge q^{-\lceil s\rceil \dim G}\vol(G_{x,1}).
\]
Hence
$\vol(G_{x,1}) \deg(\pi_\Sigma) \le q^{(s_\Sigma+2) \dim G}$.

We deduce from this, the inequality~\eqref{eq: orb}, and Proposition~\ref{p:
power-saving} above that
\[
\frac{ \left|\orbi_\gamma(f_{\pi_\Sigma})\right|}{\deg(\pi_\Sigma)}\le q^2 C_1 \cdot
(s_\Sigma+1)\cdot \left(\vol(G_{x,1}) \deg(\pi_\Sigma)\right)^{-\frac{1}{\dim
G}}D(\gamma)^{-1/2}.
\]
The proof of Theorem~\ref{t:asymptotic-orb-int} is complete.
\end{proof}

\begin{rem}\label{r:non-cpt-center}
Let us explain what can be done when the center $Z$ is not compact.
Theorem \ref{t:asymptotic-orb-int} and Proposition~\ref{p: power-saving} remain valid as
stated. The only modification to the proof in Section~\ref{s:proof} is that we
need to use the subgroups
$L_s:=G_{x,s}G^{\prime}_{[x]}$ which are only compact-mod-center, in place of
$G_{x,s}G'_x$.

Lemma \ref{l:sc-coeff} is still true in this case, if $\vol(J)$ is replaced with the
volume of $J/Z$ in $G/Z$, and verified by the same argument.
It may be troubling at first that the support of $f_\pi$ is compact only modulo center
but equality \eqref{e:char-orb-int} remains valid.\footnote{We recall that \eqref{e:char-orb-int}
follows from \cite[Thm 5.1]{Art93}, based on Arthur's local trace formula  \cite[Thm 4.2]{Art93}. In the case $Z$ is not compact,
 \cite[Prop 6.1]{ArtSTF3} gives the local trace formula with fixed central character and similarly implies the analogue of \eqref{e:char-orb-int}.}

Conjecture~\ref{c:asymptotic-char} below remains the same, even when the center $Z$ is
not compact. One only needs to remember that in our convention $\Irr^\scusp(G)$ consist
of \emph{unitary} supercuspidal representations, so the central character of $\pi$ is
also unitary.

Finally, in the global application we appeal to the simple trace formula with fixed
character on a closed central subgroup so as to allow $f_\pi$ as the local component of a
test function.
Alternatively one could work with a truncated (pseudo-)coefficient (cf.
\cite[1.9]{HL04}), which is compactly supported, in place of $f_\pi$.
\end{rem}

\subsection{Asymptotic behavior of supercuspidal
characters}\label{s:asymptotic-behavior-sc}
The main result, Theorem~\ref{t:asymptotic-char}, of this section may be rephrased as a
uniform upper bound for the characters of supercuspidal representations constructed by Yu
on elliptic regular elements via Proposition~\ref{p:char-orb-int}.

 Let us recall the context of the problem.
 Even though the precise character formulas for supercuspidal (and discrete series)
 representations of a $p$-adic group remain largely mysterious, we conjectured in
 \cite[Conjecture 4.1]{KST-localconstancy} that the character values behave in a
 controlled manner as the formal degree tends to infinity.
Recall that $\Irr^\scusp(G)$ denote the set
of isomorphism classes of supercuspidal representations.
Write $G_{\mathrm{ell}}$ (resp. $G_{\mathrm{rs}}$) for the set of elliptic (resp. regular
semisimple) elements in $G$.

\begin{conj}\label{c:asymptotic-char}
(i)  For every $\gamma\in G_{\mathrm{rs}}$,
\[
\lim_{\pi\in \Irr^{\scusp}(G)\atop \deg(\pi)\ra\infty}
\frac{\Theta_{\pi}(\gamma)}{\deg(\pi)}=0.
\]
In other words, for every $\gamma\in G_{\mathrm{rs}}$ and every $\epsilon>0$ there exists
$d_{\gamma,\epsilon}>0$ such that $|\Theta_\pi(\gamma)| < \epsilon \deg(\pi)$ for every
$\pi\in
\Irr^{\scusp}(G)$ with $\deg(\pi)>d_{\gamma,\epsilon}$.

\smallskip

(ii) Let $\mathcal{B}\subset G$ be a bounded subset. Then
   there exist constants $\nu>0$ depending only on $G$ and $C_\mathcal{B}>0$ depending
   only on $G$ and $\mathcal{B}$ such that
   \[
D(\gamma)^{1/2} |\Theta_{\pi}(\gamma)|\le C_\mathcal{B}\cdot \deg(\pi)^{1-\nu},\quad
   \forall \pi \in \Irr^{\scusp}(G),~\forall \gamma\in G_{\mathrm{rs}}\cap \mathcal{B}.
\]

\smallskip

(iii) There exist constants $\nu>0$ and $C_{\mathrm{ell}}>0$ depending only on $G$
   such that
   \[
D(\gamma)^{1/2} |\Theta_{\pi}(\gamma)|\le C_{\mathrm{ell}}\cdot
   \deg(\pi)^{1-\nu},\quad \forall \pi \in \Irr^{\scusp}(G),~\forall \gamma\in
   G_{\mathrm{rs}}\cap G_{\mathrm{ell}}.
\]
\end{conj}

     Note that (ii) implies (i). Since $G$ has finitely many (elliptic) maximal tori up to
     conjugacy (\cite[p.320, Cor 3]{PR94}), (ii) also implies (iii).  In \cite[Thm
     4.2]{KST-localconstancy}, we
     have
     proved a result for (ii), provided that the residue
     characteristic of $k$ is sufficiently large and that $\gamma$ runs over the set
     $G_{0^+}$ of
     topologically unipotent elements (which are not necessarily elliptic). The argument
     of that paper is based on an estimate of the number of fixed points of $\gamma$
     acting on certain coset spaces (without estimating orbital integrals).

Theorem \ref{t:asymptotic-orb-int} restricted to regular elements may be interpreted in
terms of character values to establish (iii) of the conjecture provided that the residue
characteristic of
$k$ is large enough. Note that the proof of Theorem \ref{t:asymptotic-orb-int} is
independent of what is done in
\cite{KST-localconstancy} and relies on quite a different method. %

\begin{thm}\label{t:asymptotic-char} Under Hypothesis $(\mathscr{T'})$, part (iii) of
Conjecture~\ref{c:asymptotic-char} holds true if $\pi\in\Irr^{\Yu}(G)$.
In particular, part (iii) of the conjecture is
true if the hypotheses of \cite[3.4]{Kim07} are also met (cf. the last paragraph of
\S\ref{sub:sc-cpt-ind}).
\end{thm}

\begin{proof}
The first assertion follows from equality~\eqref{e:char-orb-int} and Theorem
\ref{t:asymptotic-orb-int}.
The second is deduced from the first assertion and the exhaustion theorem
of~\cite{Kim07}.
\end{proof}

It is natural to wonder whether the method of this paper may be pushed further to cover
non-elliptic elements. Proposition \ref{p:char-orb-int} is only a special case of
Arthur's formula \cite{Art87} relating supercuspidal character values on non-elliptic
regular elements to weighted orbital integrals of supercuspidal coefficients. (This
extends to cover general discrete series via the local trace formula \cite{Art93}.) So
the problem is to bound such weighted orbital integrals.

A different question concerning trace characters is whether the constant $\nu>0$ can be
found independent of the field $k$ and the residue characteristic $p$.
In this direction we observe that there exist analogues for finite groups, notably a
general estimate by  Gluck~\cite{Gluck:character-sharp}.

It is natural to ask for a common generalization of the families in the depth aspect and
in the ``level aspect". Here, level aspect means that $q_S\to \infty$.
We would consider the multi-set $\cF(\xi,\sigma,K_{S_0})$ where both the finite
set of places $S$ and the discrete series representation $\sigma$ of $G(F_S)$ are varying.
In any non-trivial sequence, the formal degree $\deg(\sigma)$ with respect to the
canonical measure (see \cite{Gro97} or \cite[\S6.6]{ST:Sato-Tate} for definition) goes to
infinity, either because the depth of $\sigma$ goes to infinity or because the residue
characteristic $q_{S}$ goes to infinity.
Theorem~\ref{t:sato-tate} above corresponds to families in the depth aspect for which $S$
is fixed.
Theorem~\ref{th:steinbergs} below corresponds to a refinement of the families in
the level
aspect where
$\sigma=\mathrm{St}_S$.
To establish such a common generalization one would need to address the above question of
uniformity of the constant $\nu>0$, and one also would need to keep track of the
polynomial dependence in the constants $C_{\mathrm{ell}}$ in
Conjecture~\ref{c:asymptotic-char}.

\section{Proof of Proposition \ref{p: power-saving}}\label{s:proof}
In this section we work with a pair $(\bG,\bG^\prime)$ where $\bG^\prime$ is a tamely
ramified twisted Levi subgroup with $\bG^\prime\neq \bG$.
We shall also fix $x\in  \Bd(G^\prime) =
\Bd(\bG^\prime,k)$ and recall that $L_s=G_{x,s}G^{\prime}_{x}$.

We start by recalling some basic definitions.
For a maximal $k$-torus $\bT$ and $\gamma\in T\setminus Z$, the {\it singular depth} of
$\gamma$ is
defined as
\[
sd(\gamma):=\max \{ \val(\alpha(\gamma)-1)\mid\alpha\in\Phi,\ \alpha(\gamma)\ne 1 \},
\]
where $\Phi$ is the
set of $\mathbf T$-roots (see \cite{AK07}). We also define the {\it minimal
depth} of $\gamma$ as
\[
md(\gamma):=\min\{\val(\alpha(\gamma)-1)\mid\alpha\in\Phi,\ \alpha(\gamma)\ne 1 \}.
\]
Both $sd(\gamma)$ and $md(\gamma)$  are independent of the choice of $\mathbf T$
containing $\gamma$, thus are well-defined for every semisimple $\gamma\in
G\setminus Z$. Moreover, we have $md(\gamma)=md(g\gamma
g^{-1})$ and $sd(\gamma)=sd(g\gamma g^{-1})$ for any $g\in G$.

The following is a special case of the decomposition theorem in~\cite{AS08}.
Compared to \cite{AS08}, our situation is simpler and hypothesis
(C) in \cite{AS08} is not needed because we do not keep track of centralizers of
good
elements here.
We include a proof for
completeness.
\begin{lem}\label{lem: good product}
Suppose Hypothesis $\scrT$ holds. Let $\bT$ be a tame maximal
$k$-torus, $\gamma\in T_0\setminus Z_0$ a compact element, and $r=md(\gamma)$.
There exist elements $z$, $\gamma'$, $\gamma_+$, such that
 $\gamma=z\gamma'\gamma_+$, and
\begin{enumerate}[(i)]
\item either $z\in
 Z_0\setminus Z_r$, or $z=1$;
\item $\gamma'\in T_r\setminus Z_rT_{r^+}$ is $\bG$-good of depth $r$;
\item   $\gamma_+\in T_{r^+}$.
\end{enumerate}
\end{lem}
\begin{proof}
We first consider the case $r>0$. Let $b_1=\dth_{\bT}(\gamma)$ so that $\gamma\in
T_{b_1}\setminus T_{b_1^+}$.
By Hypothesis $\scrT$, there is a good element $\gamma_{b_1}\in T_{b_1}$ with
$\gamma_{b_1}T_{b_1^+}=\gamma T_{b_1^+}$, thus $\gamma=\gamma_{b_1}\gamma_+$ for some
$\gamma_+\in T_{b_1^+}$. If $b_1=md(\gamma)$, take $z=1$, $\gamma'=\gamma_{b_1}$.

Otherwise $\gamma_{b_1}\in Z$, and we then let $b_2 =
\dth_{\bT}(\gamma\gamma_{b_1}^{-1})$
so
that $b_2>b_1$ and
$\gamma\gamma_{b_1}^{-1} \in T_{b_2}\setminus T_{b_2^+}$. By Hypothesis $\scrT$, there
is a good element $\gamma_{b_2}$ such that
$\gamma\gamma_{b_1}^{-1} \gamma_{b_2}^{-1}\in T_{b_2^+}$, thus
$\gamma=\gamma_{b_1}\gamma_{b_2}\gamma_+$ for some $\gamma_+\in T_{b_2^+}$. Repeating the
process, one can write $\gamma=\gamma_{b_1}\gamma_{b_2}\cdots\gamma_{b_k}\gamma_+$ where
$b_1<b_2<\cdots<b_k=md(\gamma)$, $\gamma_{b_i}\in Z$, $i=1,\cdots, k-1$ and $\gamma_+\in
T_{b_k^+}$. Then, set $z=\gamma_{b_1}\gamma_{b_2}\cdots\gamma_{b_{k-1}}$,
$\gamma'=\gamma_{b_k}$.

Now suppose $r=0$. By \cite[Prop 2.36]{Spice08}, there is an absolutely semisimple
element $\gamma_{as}\in T_0$ such that $\gamma=\gamma_{as}\gamma_{tu}$ where
$\gamma_{tu}\in T_{0^+}$. If $\gamma_{as}\not\in Z$, take
$z=1$, $\gamma'=\gamma_{as}$ and $\gamma_+=\gamma_{tu}$. If $\gamma_{as}\in Z$,
one proceeds as in the first case for $\gamma_{tu}$ to reach the conclusion.
\end{proof}

\begin{cor}\label{cor:gamma1}
Suppose Hypothesis $\scrT$ holds. Let $\bT$ be a tame maximal
$k$-torus and $\gamma\in T_0\setminus Z_0$.
Then there exists $z\in Z_0$ and $\gamma_1\in T_0\setminus Z_0$ such that
$\gamma=z\gamma_1$ and
\[\dth_{\bT}(\gamma_1) = md(\gamma).
\]
\end{cor}

\begin{proof}
If $md(\gamma)=0$, then we can choose $z=1$ and $\gamma_1=\gamma$.
Suppose $md(\gamma)>0$.
The assertion follows from Lemma~\ref{lem: good product} by setting
$\gamma_1=\gamma'\gamma_+$: Indeed $\gamma'\in
T\setminus
Z$ is $\bG$-good of positive depth
$\dth_{\bT}(\gamma')=md(\gamma')=md(\gamma)$.
Moreover $\dth_{\bT}(\gamma')=\dth_{\bT}(\gamma_1)$, which concludes the proof.
\end{proof}

We will prove Proposition \ref{p: power-saving} at the end of this section by induction
based on the following two propositions. The first, Proposition \ref{p:Kottwitz}, which is
at the
base of the induction, is concerned with orbital integrals of a fixed test function and
proved by means of Shalika germ expansions. The second, Proposition~\ref{lem:
decreasing}, allows us
to proceed inductively in the parameter
$s\in \R_{\ge 0}$.

\begin{prop}\label{p:Kottwitz}
  For each test function $f\in \cH(G)$ there exists a constant $c(f)>0$ such that for
  every semisimple $\gamma\in G$,
  $|\orbi_\gamma(f)|\le c(f)D(\gamma)^{-1/2}$.
\end{prop}
\begin{proof}
  \cite[Thm A.1]{ST:Sato-Tate}.
\end{proof}

  Recall that $L_s=G_{x,s}G'_x$. The following power-saving bounds are essential ingredients in the proof of Proposition \ref{p: power-saving}.
  Proposition \ref{lem: decreasing} can be understood in comparison with its analogue Corollary~\ref{cor:decreasing2} below, which is easier to grasp and gives
  a bound for the function $\triv_{G_{x,s}}$ in place of $\triv_{L_s}$.

\begin{prop}\label{lem: decreasing} Suppose Hypotheses $\scrT$ and $\scrE$ hold.
Let $\gamma \in G\setminus Z$ be semisimple, and $s\in\mathbb Z_{\ge 2}$.
\begin{enumerate}
\item If $md(\gamma) \le  s-2$, then
$\orbi_\gamma(\triv_{L_{s}})\le\frac1{q}\,\orbi_\gamma(\triv_{L_{s-1}}).$
\item  If $md(\gamma)\ge s+1$, then
$\orbi_\gamma(\triv_{L_{s}})\le
\frac1{q}\,\orbi_\gamma(\triv_{L_{s-1}})+\orbi_\gamma(\triv_{G\rq_{x,s-2}G_{x,s}}).$
\end{enumerate}
\end{prop}

The proof of Proposition \ref{lem: decreasing} is postponed until we establish a
handful of technical lemmas.
Lemma \ref{lem: orth} is used in the proof of Lemma \ref{lem: normalize}.

\begin{lem}\label{lem: orth} Let $\sfG$ be a connected reductive group over a perfect
field $\bbF$.
Suppose $\sfM$ is a twisted Levi subgroup of $\sfG$.
Let $\mathsf g$ and $\mathsf m$ be the Lie algebras of $\sfG(\bbF)$ and $\sfM(\bbF)$
respectively.
\begin{enumerate}[(i)]
\item
For $X\in \mathsf m$, let ${\mathsf g}_X:=\{Y\in\mathsf g\mid [X,Y]=0\}$.
If $X\not\in Z_{\mathsf g}$, then $\mathsf m + {\mathsf g}_X$ is a proper $\F$-subspace
of $\mathsf g$.
\item
For $\delta\in \mathsf M$, let ${\mathsf g}_\delta:=\{Y\in\mathsf g\mid
\Ad(\delta)(Y)=Y\}$.
If $\delta\not\in Z_{\mathsf G}$, then $\mathsf m + {\mathsf g}_\delta$ is a proper
$\F$-subspace of $\mathsf g$.
\end{enumerate}
\end{lem}

\begin{proof}
(i) Let $X=X_{ss}+X_n$ be the Jordan decomposition of $X$ with $[X_{ss},X_n]=0$
and
$X_{ss}$ (resp. $X_n$) semisimple (resp. $X_n$ nilpotent).
Since the lemma may be proved after taking a finite extension of $\bbF$, we may assume that
 $\sfM$ (thus also $\mathsf G$) and $X_{ss}$ split over $\bbF$ and that $\sfM$ is a maximal proper Levi subgroup of $\mathsf G$.
Thus ${\mathsf m}$ is a maximal proper Levi subalgebra of ${\mathsf g}$. Without loss
of
generality, we may assume that the Dynkin diagram of $\mathsf G$ is connected.

Let $\sfT$ be a maximal $\bbF$-split torus in $\sfM$ whose Lie algebra contains $X_{ss}$.
Let $\Phi$ be the set of $\sfT$-roots.
Let $\Delta$ (resp. $\Phi^+$) be
the set of simple roots (resp. the set of positive roots) associated with $\sfT$ such that
$X_n\in\sum_{\alpha\in\Phi^+}{\mathsf g}_\alpha$, where $\mathsf g_\alpha$ is the root space of
$\alpha$ in $\mathsf
g$.
Let $\Delta_{\sfM}$ (resp. $\Phi_{\sfM}^+$) be the subset of $\Delta$ (resp. $\Phi^+$)
associated with $\sfM$. Let ${\sfU}$ and ${\sfU}^-$ be the unipotent and the opposite
unipotent subgroup respectively. Let $\beta\in\Delta$ such that
$\Delta=\Delta_{\sfM}\cup\{\beta\}$.

We will prove assertion (i) by showing its contrapositive that $\mathsf g = \mathsf m
+ \mathsf g_X$ implies that $X\in Z_{\mathsf g}$.
Assuming $\mathsf g = \mathsf m + \mathsf g_X$, we have that $\mathsf
g_\beta\subset {\mathsf u}\subset{\mathsf g}_X$.
We proceed in two steps, first showing that $X_{ss}$ is central, and then that
$X_n=0$.

We have  ${\mathsf u}\subset{\mathsf g}_{X_{ss}}$
since $[X_{ss},X_n]=0$ and $X_n$ is nilpotent. In particular, $\beta(X_{ss})=0$. Let
$\alpha_1\in\Delta_{\sfM}$ adjacent to $\beta$ in the Dynkin diagram. Then,
$\beta+\alpha_1\in\Phi^+\setminus\Phi^+_{\sfM}$ by \cite[Prop 8.4]{Hum78} (this follows
from the results about root strings) and $0\neq{\mathsf
g}_{\alpha_1+\beta}\subset{\mathsf u}$. Then, $(\beta+\alpha_1)(X_{ss})=0$, hence
$\alpha_1(X_{ss})=0$ and ${\mathsf g}_{\alpha_1}\subset{\mathsf g}_{X_{ss}}$. Similarly if
$\alpha_2\neq\beta$ is adjacent to $\alpha_1$,
$\beta+\alpha_1+\alpha_2\in\Phi^+\setminus\Phi_{\sfM}^+$ and ${\mathsf
g}_{\alpha_2}\subset{\mathsf g}_{X_{ss}}$. Since the Dynkin diagram is connected,
inductively, we conclude ${\mathsf g}_\alpha\subset{\mathsf g}_{X_{ss}}$ for all
$\alpha\in\Delta$. Thus, ${\mathsf m}\subset{\mathsf g}_{X_{ss}}$ and ${\mathsf
g}_{X_{ss}}={\mathsf g}$, therefore $X_{ss}$ is central.

\smallskip

Without loss of generality, we may assume that $X_{ss}=0$, thus $X=X_n$. Write
$X=\sum_{\alpha\in\Phi^+} X_\alpha$ with $X_\alpha\in{\mathsf g}_\alpha$. Fix $\alpha'\in\Phi$. Let $\Delta$ be a simple root system with $\alpha'\in\Delta$. Let $\Delta_M$, $\beta$, $\mathsf U$ and $\mathsf U^-$ be as in the previous case.

Let $\alpha_1, \cdots,\alpha_j=\alpha'$ be distinct simple roots in $\Delta_\sfM$ such that
$\beta$ and $\alpha_1$ (resp. $\alpha_i$ and $\alpha_{i+1}$) are adjacent to each other
in the Dynkin diagram. Then, we have the following:

(a) $\beta+\alpha_1+\cdots+\alpha_i\in\Phi^+\setminus\Phi^+_{\sfM}$ for $i\le j$. This
follows from $\beta\not\in\Delta_{\sfM}$ and by inductively applying \cite[Prop
8.4]{Hum78} to the root string.

(b) $[{\mathsf g}_{\beta+\alpha_1+\cdots+\alpha_i}, X]=0$ for $i\le j$. This follows
from  ${\mathsf g}_{\beta+\alpha_1+\cdots+\alpha_i}\subset\mathsf u\subset\mathsf g_X$.

(c) $[{\mathsf g}_{\beta+\alpha_1+\cdots+\alpha_{i-1}}, X]=0$ implies that
$[{\mathsf g}_{\beta+\alpha_1+\cdots+\alpha_{i-1}}, X_{\alpha_i}]=0$. Then,
$X_{\alpha_i}=0$ since $\beta+\alpha_1+\cdots+\alpha_i\neq0$. Note that we are using the assumption that $p$ is large enough (e.g. $p$ is large enough so that it does not divide any structural constants of $\mathsf g$, which is necessary for $\scrE$-(iii)).

By (a), (b) and (c), $X_{\alpha'}=0$ since the Dynkin diagram is connected. As
$\alpha'$ runs over $\Phi$, we have $X=0$, which is what we wanted, completing the
proof of (i).

\medskip

Assertion (ii) can be proved similarly as in (i) using the Jordan decomposition
$\delta=\delta_{ss}
\delta_n$ in $\mathsf M$.
\end{proof}

We recall the notation $[g,h]:=ghg^{-1}h^{-1}$ for $g,h\in G$.
We will frequently use the fact \cite[(2.6)]{MP94} that for $a,b\in \R_{\ge 0}$,
\begin{equation}\label{eq:commutator}
 \mbox{if}\quad g\in G_{x,a},~h\in G_{x,b} \quad \mbox{then} \quad [g,h]\in G_{x,a+b}.
\end{equation}
It is also useful for us to introduce functions $\mathfrak d_x:G_{x,0}\setminus \{1\}
\rightarrow
\mathbb
R_{\ge0}$ and $\mathfrak d_x^Z:G_{x,0}\setminus Z_0 \rightarrow \mathbb R_{\ge0}$ as follows:
\begin{align*}
&\mathfrak d_x(g):=\max\{r\in\mathbb R_{\ge0}\mid g\in G_{x,r}\}, \\
&\mathfrak d_x^Z(g):=\max\{r\in\mathbb R_{\ge0}\mid g\in ZG_{x,r}\}.
\end{align*}
Clearly $\mathfrak d_x(g)\le \mathfrak d_x^Z(g)$.
For $s\in \R_{\ge 0}$ and $\delta\in
L_{s}$, define
\[
\dth_{x,s}(\delta):=\max\{t\in [0,s] \mid \delta\in
ZG'_{x,t}G_{x,s}\}.
\]
Clearly $\mathfrak d_x(g)=\mathfrak d_x(g^{-1})$, $\mathfrak d_x^Z(g)=\mathfrak
d_x^Z(g^{-1})$, and $\dth_{x,s}(\delta)=\dth_{x,s}(\delta^{-1})$.
It is also clear that $\dth_{x,s}(\delta)\le \mathfrak d_x^Z(\delta)$, since $G'_{x,t}\subset
G_{x,t}$.

\begin{lem} \label{Ldxle}
Suppose Hypotheses $\scrT$ and $\scrE$ hold.
For every $s\in \R_{>0}$ and semisimple element $\delta\in L_s\cap G_0 \setminus Z_0$,
\[
 \min(\mathfrak d^Z_x(\delta),s)
\le
\dth_{x,s}(\delta) \le md(\delta).
\]
In particular, for a semisimple element $\delta\in L_s\cap G_0 \setminus Z_0G_{x,s^+}$,
we
have $\mathfrak d^Z_x(\delta)
=
\dth_{x,s}(\delta)$.
\end{lem}

\begin{proof}
Let $r\le s$.
Since $G'_x\cap G_{x,r}=G'_{x,r}$ (see e.g., \cite[\S4]{AS08}), we have
$G'_{x}G_{x,s}\cap
G_{x,r} =
G'_{x,r}G_{x,s}$.
The condition $\delta\in G_{x,r}$ thus implies $\delta\in
G'_{x,r}G_{x,s}$. Then $\delta\in Z G_{x,r}$ implies $\delta\in
ZG'_{x,r}G_{x,s}$.
In other words $\min(\mathfrak d^Z_x(\delta),s)\le \dth_{x,s}(\delta)$, which is the first
inequality.

By Hypothesis $\scrE$, there exists a tame maximal $k$-torus $\bT$ containing $\delta$.
It follows from Corollary \ref{cor:gamma1} that
\[
\max\{r\in\mathbb R_{\ge0}\mid \gamma\in ZT_r\} = md(\delta).
\]
The second inequality then follows from the definitions since
$G'_{x,t}G_{x,s}\subset G_{x,t}$, and $G_{x,t}\cap T\subset T_t$.

The last assertion is a corollary of the first inequality because $\delta\not\in
Z_0G_{x,s^+}$ is equivalent to $\mathfrak d_x^Z(\delta)\le s$.
\end{proof}

\begin{lem}\label{lem: normalize} Suppose Hypotheses $\scrT$ and $\scrE$ hold. Let
$s\in\bbZ_{\ge2}$ and $\delta\in L_s$ be such that
$\dth:=\dth_{x,s}(\delta)<s-1$.
Define the map
\[
C_{\delta}:G_{x,s-\dth-1}\left/G_{x,s-\dth}\right.\rightarrow
G_{x,s-1}\left/(G'_{x,s-1}G_{x,s})\right.
\]
by $C_{\delta}(g):=[\delta^{-1},g]\pmod{G'_{x,s-1}G_{x,s}}$. Then
the cardinality of the image of $C_{\delta}$ is at least $q$.
\end{lem}

\begin{proof}
The last assertion of Lemma~\ref{Ldxle} implies that $\mathfrak d^Z_x(\delta)=  \dth$.
Thus the map $C_{\delta}$ is well-defined in view of \eqref{eq:commutator}.
Since $C_{\delta}$ is unchanged if $\delta$ is multiplied by a central element, we may
assume without loss of generality that $\delta\in G_{x,\dth}$.

Let $E$ be a tamely ramified finite extension of $k$ that satisfies Hypothesis $\scrE$, and
fix a uniformizer $\varpi_E$ of $E$.
Let $m$ be the ramification index of $E$ over $k$. Hence $\val (\varpi_E)=1/m$, in
view of
our normalization of valuation and
filtration index in \S\ref{sub:notation Bd}.
Since the order of $x$ divides $m$ by $\scrE$.(ii), we have $\blieG(E)_{x,
t/m}=\varpi_E^t\blieG(E)_{x,0}$ for
$t\in\bbZ$ and there is no break for non-integral values of $t$.
Moreover,
$\dth\in\frac1m\Z$.
Let $i\in \Z_{>0}$ be such that
\[
\frac{i}{m}=s-\dth-1.
\]
We now divide the proof into the following four steps (1-4):

\smallskip

(1)   Write
$$\overline{\bG(E)}_{x,i/m}:=\bG(E)_{x,i/m}\left/\bG(E)_{x,(i+1)/m}\right.;\quad
\overline{\blieG(E)}_{x,i/m}:=\blieG(E)_{x,i/m}\left/\blieG(E)_{x,(i+1)/m}\right..$$
Similarly, we define $\overline G_{x,i/m}$, %
$\overline{\mathfrak g}_{x,i/m}$, $\overline{\bG'(E)}_{x,i/m}$,
$\overline{\blieG'(E)}_{x,i/m}$, $\overline{G'}_{x,i/m}$ and $\overline{\mathfrak
g'}_{x,i/m}$.
Since $i>0$,  we have abelian group isomorphisms
 $$
 \overline{\bG(E)}_{x,i/m}\simeq\overline{\blieG(E)}_{x,i/m}
 \quad\mbox{and}\quad\overline{\bG'(E)}_{x,i/m}\simeq\overline{\blieG'(E)}_{x,i/m}
 $$

\smallskip

(2) $\overline{\blieG(E)}_{x,i/m}$ and $\overline{\blieG'(E)}_{x,i/m}$ are vector spaces
over the residue field ${\mathbb E}$ of the maximal unramified extension $E^u$ of
$k$ in $E$.
Let $\mathbb F=\mathbb F_q$ be the residue field of $k$.
Let $\sfG$ (resp. $\sfG'$) be a reductive group defined over $\mathbb F$ which is the
reductive quotient at the building point $x$ of $\bG$ (resp. $\bG'$) given by
Bruhat-Tits theory.
Write $\mathsf g$ and $\mathsf g'$ for the Lie algebras of $\sfG(\mathbb F)$ and $\sfG'(\mathbb
F)$
respectively.
We have $\overline{\blieG(E)}_{x,0}\simeq\mathsf
 g(\mathbb E)$.
We claim $\dim_E \blieG(E)=\dim_{\mathbb E}\mathsf g(\mathbb E)$. This
equality
can be established from \cite[\S4.2]{DeBacker:parametrization}. More precisely, we
have the isomorphisms of $\mathbb E$-vector spaces
$$\overline{\blieG(E)}_x:=\oplus_{r\in\mathbb R/\frac1m\mathbb
Z}\blieG(E)_{x,r}/\blieG(E)_{x,r^+}=\overline{\blieG(E)}_{x,0},$$ where the first equality
is a definition given in \emph{loc. cit.}
Note that the range $r\in\mathbb R/\frac1m\mathbb Z$ comes from our normalization which
is such that
$\val(\varpi_E)=1/m$, and hence
$\varpi_E\blieG(E)_{x,r}=\blieG(E)_{x,r+\frac1m}$ as recalled above, so the second
equality is clear since all breaks occur when $r\in \frac1m \Z$.
Then, by the
result in \emph{loc.cit.}, we have $\dim_E\blieG(E)=\dim_E\overline{\blieG(E)}_x$, which
establishes the claim.

\smallskip

(3) If $\dth=0$, then $\frac im=s-1$, and the map $g\mapsto [\delta^{-1},g]$ induces
$\overline{\bG(E)}_{x,i/m}\rightarrow \overline{\bG(E)}_{x,i/m}$ thus a linear map
$\mathsf g(\bbE) \rightarrow \mathsf g(\bbE)$. Moreover, it maps
$\overline G_{x,i/m}\subset\overline{\bG(E)}_{x,i/m} $ into itself since $\delta^{-1}$
lies in $G_x$ (rather than $\bG(E)_x$). Its composition with the projection $\mathsf
g(\bbE) \ra \mathsf g(\bbE)\left/\mathsf g'(\bbE)\right.$ is seen to be equal to the map
$$\overline C_\delta:\mathsf g(\bbE)\ra \mathsf g(\bbE)\left/\mathsf
g'(\bbE)\right.\simeq \blieG(E)_{x,i/m}/({\blieG'(E)}_{x,i/m}+{\blieG(E)}_{x,(i+1)/m})$$
induced by the map $Y\mapsto \Ad(\delta^{-1})(Y)-Y$ from $\blieG(E)_{x,i/m}$ to
$\blieG(E)_{x,i/m}$. Equivalently, $\overline C_\delta$ is induced by the map $Y\mapsto
\Ad(\ol\delta^{-1})(Y)-Y$ from $\mathsf g(\bbE)$ to itself, where
\[
\ol\delta\in
G_{x,0}/G_{x,0^+}\subset\ol{\mathbf G(E)}_{x,0}\simeq \sfG(\bbE)
\]
is the image of
$\delta$.

\smallskip

(4)
If $\dth>0$, then $\frac im<s-1$. Write $X\in \mathfrak g_{x,\dth}$ for the element whose
exponential is
$\delta$. Then the map $g\mapsto \delta^{-1}g\delta g^{-1}$ induces a map
$\overline{\bG(E)}_{x,i/m}\rightarrow \overline{\bG(E)}_{x,s-1}$ and thus a
linear map $\mathsf g(\bbE) \rightarrow \mathsf g(\bbE)$ via (2). Composed with the
projection $\mathsf g(\bbE) \ra \mathsf g(\bbE)\left/\mathsf g'(\bbE)\right.$, this map
is equal to the map
$$\overline C_\delta:\mathsf g(\bbE)\ra \mathsf g(\bbE)\left/\mathsf
g'(\bbE)\right.\simeq \blieG(E)_{x,i/m}\left/({\blieG'(E)}_{x,s-1}+{\blieG(E)}_{x,s-1})\right.$$
induced by the map $Y\mapsto [Y,X]$ from $\blieG(E)_{x,i/m}$ to
$\blieG(E)_{x,s-1}$ (a proof of this assertion is similar to that of
Proposition~\ref{p:hypotheses}). Let $\overline X$ denote the image of $X$ under the
isomorphism $\overline{\blieG(E)}_{x,\dth}\simeq\mathsf g(\bbE)$ in (2). Then $\overline
X\in \mathsf g(\bbE)$. We see that $\overline C_\delta$ is given by $\overline
C_\delta(\overline Y)=[\overline Y,\overline X]$ modulo $\mathsf g'(\bbE)$ for each $Y\in
\mathsf g(\bbE)$.

In summary, we have a series of maps such that the following diagram commutes:
\begin{equation}\tag{$\dag$}
\xymatrix{
\overline{\bG'(E)}_{x,\frac{i}{m}}  \ar[r]^-{\sim} \ar@{^(->}[d] &
\overline{\blieG'(E)}_{x,\frac{i}{m}}  \ar[r]^-{\sim} \ar@{^(->}[d] & \mathsf g'(\mathbb E)
\ar@{^(->}[d] \\
\overline{\bG(E)}_{x,\frac{i}{m}} \ar[r]^-{\sim} \ar[d]^-{[\delta^{-1},~]} &
\overline{\blieG(E)}_{x,\frac{i}{m}} \ar[r]^-{\sim}   \ar[d]^-{\mathrm{ad}(X)-1} &
\mathsf g(\mathbb E) \ar[d]\\
\overline{\bG(E)}_{x,s-1}  \ar[r]^-{\sim} &
\overline{\blieG(E)}_{x,s-1}  \ar[r]^-{\sim} & \mathsf g(\mathbb E)
}
\end{equation}
Here the three horizontal isomorphisms on the right are chosen to make the
diagram commute.

\medskip

Consider the map $\overline C_\delta:\mathsf g(\bbE)
\ra \mathsf g(\bbE)/\mathsf g'(\bbE)$.
Since $\delta\notin ZG_{x,\dth^+}$ (because $\mathfrak d_x^Z(\delta)=\dth$), we see that
$\ol\delta$ is not in the center of
$\sfG(\bbE)$ when $\dth=0$ and similarly $\ol X$ is not in the center of $\mathsf g(\mathbb
E)$ when
$\dth>0$.
(Indeed $\delta\not\in \bZ(E)\bG(E)_{x,\dth^+}$, hence $X\not\in \boldsymbol{\mathfrak z}(E) +
\mathfrak
\blieG(E)_{x,\dth^+}$ in
view of
Hypothesis $\scrE$.(iii), therefore the image of $X$ in
$\overline{\mathfrak\blieG(E)}_{x,\dth}$ is noncentral).
We apply Lemma~\ref{lem: orth} with $\sfM=\sfG'$,
which yields $\dim_{\mathbb E}\left(\mathrm{Im}(\overline C_\delta)\right)\ge1$.

In both cases (either $\dth=0$ or $\dth>0$), the map $\overline C_\delta$ sends the
rational subspace $\mathsf g$ into the rational subspace $\mathsf g/\mathsf
g'$.
Via identification given in the diagram ($\dag$), after taking Galois invariants
of $\bbE$-vector spaces, we see that the dimension of $\mathrm{Im}(\overline
C_\delta)\cap\mathsf
g \pmod{\mathsf g'(\bbE)}$ over $\bbF$ is at least one.
Indeed for each component in the above diagram, we have
$\overline{\mathbf H(k)}_{x,r}\hookrightarrow\overline{\mathbf H(E)}_{x,r}$ where $\mathbf H=\bG, \bG'$
and $r=\frac im, s-1$, and observe also that the horizontal isomorphisms on the
right
can be also chosen such that $\overline{\mathbf H(k)}_{x,r}\hookrightarrow\overline{\mathbf
H(E)}_{x,r}$ induces $\mathsf h \hookrightarrow \mathsf h(\mathbb E)$ where
$\mathsf
h=\mathsf g, \mathsf g'$.
Lifting this fact to
$C_\delta$, we have $\sharp\left(\mathrm{Im}(C_\delta)\right)\ge q$.
\end{proof}

\begin{proof}[{Proof of Proposition \ref{lem: decreasing}}]
Without loss of generality, we may assume that $\gamma\in L_{s}$.
Let $\psi_\gamma:G\rightarrow G$ be given by $\psi_\gamma(g):=g\gamma g^{-1}$.
Then, $\psi_\gamma(G)=\orb(\gamma)$.
Moreover $g\in \psi_\gamma^{-1}(L_{s})$ (resp.
$g\in\psi_\gamma^{-1}(L_{s-1})$) if and only if $[\gamma^{-1}, g]\in L_{s}$ (resp.
$[\gamma^{-1}, g]\in L_{s-1}$).

We consider the disjoint decomposition
\[
L_s\cap\orb(\gamma)=\bigsqcup_{i=1}^n\delta_i G'_{x,s-1}G_{x,s}\cap \orb(\gamma)
\]
 for some $\delta_1,\cdots, \delta_n\in L_s\cap\orb(\gamma)$.
Set $\dth_i:=\dth_{x,s}(\delta_i)=\dth_{x,s}(\delta_i^{-1})$.

\medskip

\noindent
{\it Case (i).} $md(\gamma) \le s-2$. Then $\dth_i\le s-2$ by the second
inequality of Lemma~\ref{Ldxle}.
Furthermore, by the first inequality of Lemma~\ref{Ldxle}, $\delta_i, \delta_i^{-1}\in
Z G_{x,\dth_i}$.

Put $V_{i,s}:=\psi_\gamma^{-1}(\delta_i G'_{x,s-1}G_{x,s})$. We have that
$V_{i,s}\subset \psi_\gamma^{-1}(L_s)\subset \psi_\gamma^{-1}(L_{s-1})$.
 We claim that for any $u\in G_{x,s-\dth_i-1}$,
\[u V_{i,s}\subset \psi_\gamma^{-1}(L_{s-1}).\]
Indeed for $v\in V_{i,s}$ we have
\[
\tilde v:= \psi_\gamma(v)= v\gamma v^{-1}\in \delta_i G'_{x,s-1}G_{x,s},
\]
 so in particular $\tilde v\in L_{s}\cap Z G_{x,\dth_i}$.
Hence $[\tilde v^{-1},u]\in G_{x,s-1}$ by \eqref{eq:commutator}, and therefore
\[
\psi_\gamma(uv)=uv \gamma v^{-1}u^{-1}=\tilde v[\tilde v^{-1},u]  \in L_{s-1},
\]
 verifying the claim.

Next we want to show that the sets $G_{x,s-\dth_i-1} V_{i,s}$ are disjoint for $1\le i
\le n$.
Indeed suppose that there exist $u\in G_{x,s-\dth_i-1}$, $u'\in G_{x,s-\dth_{i'}-1}$,
$v\in V_{i,s}$ and $v'\in V_{i',s}$ such
that $uv=u'v'$.
As before $\tilde v:=v\gamma v^{-1}\in \delta_i G'_{x,s-1}G_{x,s}$ so we may write
$\tilde v=\delta_i g$ for $g\in G'_{x,s-1}G_{x,s}\subset G_{x,s-1}$. Then
 $$uv\gamma (uv)^{-1}=(u \delta_i u^{-1})(u g
 u^{-1})=(\delta_i[\delta_i^{-1},u])(g[g^{-1},u]).$$
 Again $[\delta_i^{-1},u],[g^{-1},u]\in G_{x,s-1}$. Hence
\[
\psi_\gamma(uv)=uv\gamma (uv)^{-1}\in \delta_i G_{x,s-1}.
\]
The same reasoning shows that $u'v'\gamma (u'v')^{-1}\in \delta_{i'} G_{x,s-1}$.
Since $uv=u'v'$, it implies that $\delta_i \equiv \delta_{i'}\pmod{ G_{x,s-1}}$.
This is promoted to $\delta_i \equiv
\delta_{i'}\pmod{G'_{x,s-1}G_{x,s}}$ thanks to the fact that $\delta_i,\delta_{i'}\in
L_s$, thus $i=i'$, verifying the
disjointness.

Define a map $C_{\delta_{i}}$ as follows:
\[
C_{\delta_{i}}:G_{x,s-\dth_i-1}\left/G_{x,s-\dth_i}\right.\rightarrow
G_{x,s-1}\left/(G'_{x,s-1}G_{x,s})\right.
\]
given by $C_{\delta_{i}}(g):=[\delta_{i}^{-1},g]\pmod{G'_{x,s-1}G_{x,s}}$.
By Lemma \ref{lem: normalize},
for each $i$ there exist $q$ elements $u_{i1}, u_{i2},\cdots,u_{iq}\in G_{x,s-\dth_i-1}$
such that $C_{\delta_{i}}(u_{ij})$ are distinct,
$j=1,\cdots,q$.

As consequence of the above claim we have
$$
\psi_\gamma^{-1}(L_{s-1})\supset
\bigcup\limits_{i=1}^n \bigcup\limits_{j=1}^q  u_{ij}V_{i,s}.
$$
To finish the proof of (i) it is enough to prove that the terms on the right hand side
are mutually disjoint. Indeed, if for each open compact subset $U\subset G$ we write
$\vol_{G/G_\gamma}(U)$ to denote the volume of the image of $U$ in $G/G_\gamma$ then we
will have
 $$ \orbi_\gamma(\triv_{L_{s-1}})
 =\vol_{G/G_\gamma}(\psi_\gamma^{-1}(L_{s-1}))
 \ge \sum_{i=1}^n \sum_{j=1}^q\vol_{G/G_\gamma}(u_{ij}V_{i,s})$$
 $$= q  \sum_{i=1}^n \vol_{G/G_\gamma}(V_{i,s})=q \orbi_\gamma(\triv_{L_s}).$$

Since the sets $G_{x,s-\dth_i-1} V_{i,s}$ are disjoint, it only remains to show that
$u_{ij}V_{i,s}$ and $u_{ij'}V_{i,s}$ are
disjoint for $j\neq j'$ and $1\le i \le n$. Suppose $u_{ij}V_{i,s}\cap
u_{ij'}V_{i,s}\neq\emptyset$ for some $i,j,j'$.
There are $v, v'\in V_{i,s}$ such that $u_{ij}v=u_{ij'}v'$.
As before, $v\gamma v^{-1}=\delta_ig$ for some $g\in G'_{x,s-1}G_{x,s}$.
Hence,
$$u_{ij}v\gamma (u_{ij}v)^{-1}=(u_{ij} \delta_i u_{ij}^{-1})(u_{ij} g
u_{ij}^{-1})=(\delta_i[\delta_i^{-1},u_{ij}])(g[g^{-1},u_{ij}]).$$
Since $\dth_i\le s-2$, we apply \eqref{eq:commutator} to obtain $[g^{-1},u_{ij}]\in
G_{x,2s-\dth_i-2}\subset G_{x,s}$.
Thus the term $ u_{ij}v\gamma (u_{ij}v)^{-1} $ above belongs to $\delta_i
C_{\delta_i}(u_{ij}) G'_{x,s-1}G_{x,s}$.
Since $u_{ij}v = u_{ij'}v'$, we deduce similarly that it also belongs to $\delta_i
C_{\delta_i}(u_{ij'}) G'_{x,s-1}G_{x,s}$. This implies $C_{\delta_i}(u_{ij}) =
C_{\delta_i}(u_{ij'})$,
hence $j=j'$.

\medskip

\noindent
{\it Case (ii). $md(\gamma)\ge s+1$.}

\smallskip

Observe that $\orb(\gamma)\cap Z G_{x,s}\neq\emptyset$. This follows from combining
Corollary~\ref{cor:gamma1} and the fact that for any torus $\bT$, there exists $g\in
G$ with
$ZT_{s+1}\subset ZG_{gx,s}$. Hence, one can assume $\gamma\in Z G_{x,s}$. Without
loss of generality, we assume that $\gamma\in G_{x,s}$.

Write $\gamma=z\gamma_1$ with
$z\in Z$ and $\gamma_1\in T_0\setminus Z_0$ with $\dth_{\bT}(\gamma_1)=md(\gamma)$ as in
Corollary~\ref{cor:gamma1}. Then, since
\[
\min(\dth_{\bT}(z), md(\gamma_1))=\dth_{\bT}(\gamma)\ge\mathfrak d_x(\gamma)\ge s,
\]
 we
have $\dth_{\bT}(z)\ge s$, hence $z\in Z_s\subset G_{x,s}$.

If $\dth_i=\dth_{x,s}(\delta_i)\ge s-2$, we have $\delta_i\in Z G'_{x,s-2}G_{x,s}$. We
claim that $\delta_i\in G'_{x,s-2}G_{x,s}$.
To prove this claim, write $\delta_i=z\, ^h\!\gamma_1$ for some $h\in G$. Since $z\in
G_{x,s}$, it is enough to show that $^h\!\gamma_1\in G'_{x,s-2}G_{x,s}$. Suppose
$^h\!\gamma_1\not\in G'_{x,s-2}G_{x,s}$, that is, $\mathfrak d_x(\,^h\!\gamma_1)<s-2$.
Then, since $\dth_{x,s}(\,^h\!\gamma_1)\ge s-2$, there is $z'\in Z$ of depth $\mathfrak
d_x(\, ^h\!\gamma_1)$ such that $z'\,^h\!\gamma_1\in G'_{x,s-2}G_{x,s}$. However,
$\dth_{^h\!\bT}(z'\,^h\!\gamma_1)=\dth_{^h\!\bT}(z')=\mathfrak d_x(\, ^h\!\gamma_1)<s-2$,
hence $z'\,^h\!\gamma_1\not\in G_{x,s-2}$, which is a contradiction. Hence, the claim
follows.

Now we can arrange the decomposition such that $\dth_i <  s-2$ for $1\le i \le n'$ and
$\delta_i\in  G'_{x,s-2}G_{x,s}$ for $n'+1\le i\le n$:
\[
L_{s}\cap\orb(\gamma)=
\bigsqcup\limits_{i=1}^{n'}\left(\delta_i G'_{x,s-1}G_{x,s}\cap \orb(\gamma)\right)
\bigcup
\left(G'_{x,s-2}G_{x,s}\cap \orb(\gamma)\right).
\]
For $1 \le i \le n'$, choose $u_i$ and define $V_{i,s}$ as in Case (i).
Then,
$\psi_\gamma^{-1}(L_{s-1})$ contains
$$
\bigcup_{i=1}^{n'}\bigcup_{j=1}^{q} u_{ij} V_{i,s}$$
and the summands are mutually disjoint by a similar argument to Case (i). Arguing as in
Case (i) but keeping in mind that $G'_{x,s-2}G_{x,s}\cap \orb(\gamma)$ accounts for
$\orbi_\gamma(\triv_{G\rq_{x,s-2}G_{x,s}})$, we complete the proof of Case (ii) as
follows.
 $$ \orbi_\gamma(\triv_{L_{s-1}})\ge q  \sum_{i=1}^{n'} \vol_{G/G_\gamma}(V_{i,s})\ge q
 \left(\orbi_\gamma(\triv_{L_s})-\orbi_\gamma(\triv_{G\rq_{x,s-2}G_{x,s}})\right).\vspace{-.25in}$$
\end{proof}

In preparation for the end of the proof of Proposition~\ref{p: power-saving} we shall
need a final bound from Corollary \ref{cor:decreasing2} below.
\begin{lem}\label{lem:orbital-Lie} Assume that the
isomorphisms of Hypothesis $\scrE$.(iii) are $G$-equivariant.
 Let $s\in \R_{\ge 1}$, and $\gamma\in G_{x,s}$. Write $X:=\exp^{-1}(\gamma)$. Then
  $$\orbi^G_\gamma(\triv_{G_{x,s}})=\orbi^{\fkg}_X(\triv_{\fkg_{x,s}}).$$
\end{lem}

\begin{proof}
We are going to follow the idea as in the proof of Theorem 3.2.3 of \cite{Fer07}.
  Let $G_\gamma$ (resp. $G_X$) denote the connected centralizer of $\gamma$ (resp. $X$)
  in $G$.  Write $K:=G_{x,s}$ and $\fkk:=\fkg_{x,s}$. By the definition of orbital
  integrals,
  $$\orbi^G_\gamma(\triv_{K})=\sum_{g\in G_\gamma\bs G/K} \vol(G_\gamma\bs G_\gamma g
  K)\triv_K(g^{-1}\gamma g),$$
  $$\orbi^{\fkg}_X(\triv_{\fkk})=\sum_{g\in G_X\bs G/K} \vol(G_X\bs G_X g
  K)\triv_{\fkk}(g^{-1}X g).$$
  Since $G_\gamma=G_X$ and $\triv_K(g^{-1}\gamma g)=\triv_{\fkk}(g^{-1}X g)$, the
  equality in the lemma holds.
\end{proof}

\begin{cor}\label{cor:decreasing2} Assume
that the
isomorphisms of Hypothesis $\scrE$.(iii) are $G$-equivariant.
There is a constant $C>0$ depending only on $G$ such that the following holds:
for every
$s\in \R_{\ge 0}$, and
every semisimple $\gamma\in G\setminus Z$,
\[
|\orbi_\gamma(\triv_{G_{x,s}})|\le C\cdot q^{-s}\cdot D(\gamma)^{-1/2}.
\]
\end{cor}

\begin{proof}

In view of Proposition~\ref{p:Kottwitz}, we may assume without loss of generality that
$s\in \R_{\ge 2}$. It is enough to prove the bound for $s\in \Z_{\ge 2}$. Indeed, if the
corollary
is known for $s\in \Z_{\ge 2}$ then
the corollary holds for $s\le t<s+1$ at the expense of increasing the constant:
$$|\orbi_\gamma(\triv_{G_{x,t}})|\le |\orbi_\gamma(\triv_{G_{x,s}})|\le C\cdot q^{-s}
D(\gamma)^{-1/2}\le (C\cdot q)q^{-t}D(\gamma)^{-1/2}.$$
For the rest of this proof, $s\in \Z_{\ge 2}$ and $\gamma= \exp(X) \in G_{x,s}$. Then
\[
\orbi^G_\gamma(\triv_{G_{x,s}})=\orbi^{\fkg}_X(\triv_{\fkg_{x,s}})
=\orbi^{\fkg}_{\varpi^{1-s}X}(\triv_{\fkg_{x,1}})
=\orbi^G_{\exp(\varpi^{1-s}X)}(\triv_{G_{x,1}}).
\]
 Here the first and third equalities are from Lemma \ref{lem:orbital-Lie} and the
 second
 follows from a direct computation. Applying Proposition \ref{p:Kottwitz} to
 $f=\triv_{G_{x,1}}$ we obtain
   $$\orbi^G_{\exp(\varpi^{1-s}X)}(\triv_{G_{x,1}})\le C\cdot
  D(\exp(\varpi^{1-s}X))^{-1/2}= C\cdot q^{\frac{s-1}{2}(\dim G_\gamma - \dim
  G)}D(\gamma)^{-\frac12}\le C\cdot q^{1-s}D(\gamma)^{-\frac12},$$
  where the constant $C=c(\triv_{G_{x,1}})\in \R_{>0}$ depends only on $G$ and $x$ which
  had been
  fixed since the beginning of this section.

 We can arrange that $c(\triv_{G_{x,1}})$ depend only on the $G$-orbit
 of
  the facet containing~$x$: If $x'=gx$ with $g\in G$ then $G_{x',1}=g G_{x,1} g^{-1}$ so
  $\triv_{G_{x,1}}$ and $\triv_{G_{x',1}}$ have the same orbital integral on each
  conjugacy class.
 There are only finitely many $G$-orbits of facets, so the lemma holds true for a
 constant $C>0$ depending only on $G$.
\end{proof}

\proof[Proof of Proposition \ref{p: power-saving}]
As in the proof of Corollary \ref{cor:decreasing2}, once we prove the proposition
for a fixed $x$, the proposition also holds when $x$ varies by a similar
finiteness argument.
So we may and will keep $x$ fixed in this proof.

Let
$ a_s := D(\gamma)^{\frac12} \orbi_\gamma(\triv_{L_s}), $
which is a decreasing function of $s\in \R_{\ge 0}$.
We want to show that $a_s \le C_1\cdot q^{-s}$.
It is sufficient to verify the inequality for $s\in \Z_{\ge 0}$ at the expense of
replacing $C_1$ by $qC_1$.
Applying Proposition \ref{p:Kottwitz} to $f=\triv_{L_0}$ we have that $a_1\le a_0 \le
c(\triv_{L_0})$.

Set $m$ to be the largest integer such that $m\le md(\gamma)$.
Proposition~\ref{lem: decreasing}.(ii) combined with Corollary~\ref{cor:decreasing2}
provides us with the recursive inequality
(replacing $C$ by $q^2 C$)
  $$
a_s \le \frac1q a_{s-1}+Cq^{-s},\quad  2\le s\le m-1.
$$
  This implies
that for any $2 \le s \le m-1$,
\begin{align*}
a_s &\le q^{1-s} a_1 + C
\left(
q^{-s} + \frac{q^{1-s}}{q} + \cdots + \frac{q^{-1}}{q^{s-1}}
\right)
\le  q^{1-s} c(\triv_{L_0}) + C s q^{-s}.
\end{align*}
We have in particular
\[
a_{m+1} \le a_m \le a_{m-1} \le \left(   c(\triv_{L_0})q +  Cm \right)  q^{1-m}.
\]
Next, Proposition~\ref{lem: decreasing}.(i) shows that for $ s \ge m+2$, we have the
inequality $a_s\le \frac1q a_{s-1}$. Hence
\[
a_s \le q^{m+1-s} a_{m+1} \le  \left(   c(\triv_{L_0}) q +  Cm \right)  q^{2-s}.
\]
 Proposition~\ref{p: power-saving} is verified with $C_1:=(c(\triv_{L_0}) q+C)q^2$, which
 is indeed a constant depending only on $G$.
  \qed

\section{Automorphic Plancherel equidistribution with error terms}\label{s:plancherel}

  In this section our asymptotic formula for supercuspidal characters and orbital integrals is applied to produce an equidistribution
theorem for a family of automorphic representations. The theorem can be informally summarized as follows: Consider the set of $L^2$-discrete
automorphic representations with supercuspidals at a fixed finite place (suitably weighted). As the formal degree of the supercuspidal at
the fixed place moves toward infinity, the local components (away from the fixed place) of the automorphic representations are
equidistributed with respect to the Plancherel measure. We prove the theorem in the case where a suitable condition at infinite places
simplifies the trace formula so that the technical difficulties with general terms in the trace formula do not blur the close relationship
between the asymptotic formula for supercuspidal characters and the equidistribution.

\subsection{Preliminaries}\label{sub:prelim-unitary-dual}

In the rest of the article the following global setup will be in effect.
Let $G$ be a connected reductive group over a totally real number field $F$. (We used the symbol $G$ differently in the preceding sections.) Put $F_\infty:=F\otimes_\Q \R$. Write $\cA_{\disc}(G)$ for the set of discrete automorphic representations of $G(\A_F)$ up to isomorphism (i.e. without multiplicity). The automorphic multiplicity of $\pi\in \cA_{\disc}(G)$ is denoted $m_\disc(\pi)$. Let $S$ be a nonempty finite set of finite places of $F$. Fix a Haar measure $\mu_S$ on $G(F_S)$. Recall that the unitary dual $G(F_S)^\wedge$ is equipped with a positive Borel measure $\pl_S$, the Plancherel measure.

  Write $\Ram(G)$ for the set of finite places $v$ of $F$ such that $G$ is ramified over $F_v$. For each finite place $v\notin \Ram(G)$ let $K_v^\hs$ be a hyperspecial subgroup of $G(F_v)$. We choose $K_v^\hs$ such that at all but finitely many $v\notin \Ram(G)$, the group $K_v^\hs$ consists of the $\cO_{F_v}$-points of some reductive integral model of $G$ over $\cO_F[1/N]$ for a sufficiently large integer $N$.

  Let $v$ be a place of $F$.
Write $\mu^{\can}_v$ for the canonical measure on $G(F_v)$ (denoted by $L(M^\vee(1))\cdot |\omega_{G}|$ in \cite{Gro97}), and if $G(F_v)$ has compact center, denote by $\mu_v^{\EP}$ the Euler-Poincar\'e measure on $G(F_v)$, cf.
\cite[\S5, \S7]{Gro97}.
Assuming $G(F_\infty)$ has compact center, put $\mu^{\EP}_\infty:=\prod_{v|\infty} \mu^{\EP}_v$.
Similarly $\mu^{\EP}_S:=\prod_{v\in S} \mu^{\EP}_v$ and $\mu^{\can,\Sigma}:=\prod_{v\notin \Sigma} \mu^{\can}_v$.
When $G$ is unramified over $F_v$ it is known that $\mu^{\can}_v$ assigns volume 1 to hyperspecial subgroups.
From \S\ref{sub:counting-measures} on we will fix a finite set of places $S$ and consider the (possibly negative) measure
  $$\mu^{\can,\EP}:=\left(\prod_{v\notin S\cup S_\infty} \mu_v^{\can}\right)\mu^{\EP}_S \mu^{\EP}_\infty.$$
  (This is different from the convention of \cite{ST:Sato-Tate}; there we used $\mu^{\can}_v$ at all finite places. Also note that $\mu^{\can,\EP}$ depends on the set $S$.) Define the volume of the adelic quotient $$\tau'(G,S):=\mu^{\can,\EP}(G(F)\bs G(\A_F)),$$
  relative to the counting measure on the discrete subgroup $G(F)$.
  This volume is finite if $G(F_\infty)$ has compact center, which will always be the case by part (iii) of the assumptions made in the next subsection.\footnote{To have finite volume in general, one has to take a further quotient of $G(F)\bs G(\A_F)$ by the $\R$-split part of the center of $G(F_\infty)$.}
The Tamagawa volume of $G(F)\bs G(\A_F)$ is denoted by $\tau(G)$.

\subsection{The simple trace formula}\label{sub:simple-TF}

  As the trace formula is going to play a central role in the proof, we recall some basic facts.
  Let $T_{\mathrm{ell}},T_{\disc}:\cH(G(\A_F))\ra \C$ designate the invariant distributions consisting of contributions from the elliptic conjugacy classes and the discrete automorphic spectrum, respectively. Arthur's trace formula is the equality of the two invariant distributions $I_{\geom}$ and $I_{\spec}$ on the geometric and spectral sides. In general $I_{\geom}$ (resp. $I_{\spec}$) is the sum of $T_{\mathrm{ell}}$ (resp. $T_{\disc}$) and other very complicated terms, but we will always be in the situation where the simple trace formula applies, i.e. $T_{\mathrm{ell}}=I_{\geom}=I_{\spec}=T_{\disc}$.

Suppose that $\phi\in \cH(G(\A_F))$ admits a decomposition
   $\phi=\phi^{S\cup S_\infty}\phi_S \phi_\infty$ according to $G(\A_F)=G(\A_F^{S\cup S_\infty})G(F_S) G(F_\infty)$ such that
   $\phi_\infty$ is an Euler-Poincar\'e function on $G(F_\infty)$ as in \cite[Thm 3.(ii)]{CD90} up to a nonzero scalar.
   For the rest of the paper we make the following overarching assumptions.
  \begin{enumerate}
    \item $G(F_S)$ has compact center, %
    \item the function $\phi_S$ is cuspidal in the sense that orbital integrals vanish on non-elliptic regular semisimple conjugacy classes of $G(F_S)$, and
  \item $G(F_\infty)$ contains a compact maximal torus.
\end{enumerate}
   By (iii), the real group $G(F_\infty)$ admits discrete series spectrum and the function $\phi_\infty$ is nonzero. Condition (ii) is equivalent to the condition that the trace of any (fully) induced representation from any proper parabolic
subgroups vanishes against $\phi_S$. Typical examples of such $\phi_S$ are matrix coefficients of supercuspidal representations
(\S\ref{s:orbital-integrals}) and Kottwitz's Euler-Poincar\'e functions, cf. \S\ref{s:Steinberg} below.

   For a semisimple $\gamma\in G(F)$ write $G_\gamma$ for its centralizer and $I_\gamma$ for the neutral component of $G_\gamma$. Put
$\iota(\gamma):=[G_\gamma(F): I_\gamma(F)]\in \Z_{\ge 1}$. Let $\mu_{G(\A_F)}$ (resp. $\mu_{I_\gamma(\A_F)}$) denote a Haar measure on $G(\A_F)$ and $I_\gamma(\A_F)$, respectively. The elliptic part of the trace formula is the expansion
  \begin{equation}\label{e:Igeom}
    T_{\mathrm{ell}}(\phi,\mu_{G(\A_F)}):=\sum_{\gamma\in G(F)/\sim\atop \textrm{elliptic}} \iota(\gamma)^{-1}\mu_{I_\gamma}(I_\gamma) \orbi_\gamma(\phi,\mu_{G(\A_F)}/\mu_{I_\gamma(\A_F)}),
  \end{equation}
    where the sum runs over the set of $F$-elliptic conjugacy classes in $G(F)$, and $\mu_{I_\gamma}(I_\gamma)$ is the volume of $I_\gamma(F)\bs I_\gamma(\A_F)$ for the quotient measure of $\mu_{I_\gamma(\A_F)}$. (Note that $I_\gamma(F_\infty)$ has compact center by ellipticity of $\gamma$ and assumption (iii) above.) The discrete part of the trace formula is
  \begin{equation}\label{e:Ispec}
    T_{\disc}(\phi,\mu_{G(\A_F)}):=\sum_{\pi\in \cA_{\disc}(G)}m_\disc(\pi)\tr \pi(\phi,\mu_{G(\A_F)}),
  \end{equation}
  where $m_\disc(\pi)$ denotes the multiplicity of $\pi$ in the discrete automorphic spectrum.
  Under the above hypotheses Arthur~\cite[Cor 7.3, Cor 7.4]{Art88b} provides us with the simple trace formula
  \begin{equation}\label{e:simple-TF}
  T_{\mathrm{ell}}(\phi,\mu_{G(\A_F)})=T_\disc(\phi,\mu_{G(\A_F)}).\end{equation}
  Indeed the assumption at $v_1$ (resp. at $v_1$ and $v_2$) in Corollary 7.3 (resp. 7.4) of that paper is satisfied by any $v_1\in S$ (resp. any $v_1\in S$ and any $v_2\in S_\infty$) by (ii) and (iii) above. Here we use the property of Euler-Poincar\'e functions~\cite[p.270, p.281]{Art89} that their orbital integrals vanish outside elliptic conjugacy classes.

\subsection{Counting measures for automorphic representations}\label{sub:counting-measures}

  Let $G$ be a connected reductive group over a totally real field $F$ as in the preceding subsection.
 Let $S_0$, $S$, $\fkS$, $\xi$, $\Pi_\infty(\xi)$, $K_{S_0}$, and $K^\fkS$ be as in the introduction. (We allow $S_0$ to be empty.) Throughout this section $G$ is assumed to be unramified away from the finite set of places $\fkS$. This is always ensured by increasing the set $S_0$ if necessary.
  We make the following additional hypothesis, which is technically helpful as it was in \cite{ST:Sato-Tate}.
\bit
\item The highest weight of $\xi$ is regular.
\eit
 Write $\Irr^{\Yu}(G(F_S))$ for the set of $\sigma_S=\otimes_{v\in S} \sigma_v$ such that $\sigma_v\in \Irr^{\Yu}(G(F_v))$. Given $\sigma_S\in  \Irr^{\Yu}(G(F_S))$,
 define $$\cF=\cF(\xi,\sigma_S,K_{S_0})$$
 to be the multi-set of $\pi\in \cA_{\disc}(G)$ whose multiplicity is zero unless $\pi^\fkS$ is unramified, $\pi_S\simeq \sigma_S$, and $\pi_\infty\in \Pi_\infty(\xi)$,  in which case the multiplicity of $\pi$ is
 $$a_{\cF}(\pi):=m_{\disc}(\pi)\dim (\pi_{S_0})^{K_{S_0}}.$$%
    By Harish-Chandra's finiteness theorem, $a_{\cF}(\pi)\neq 0$ only for finitely many $\pi$.
    We may replace $m_{\disc}(\pi)$ by the multiplicity in the cuspidal spectrum since every automorphic representation with a supercuspidal component (or with $\pi_\infty$ in discrete series) is cuspidal.
    For each $\sigma_\infty\in \Pi_\infty(\xi)$ let $f_{\sigma_\infty}$ be a pseudo-coefficient for $\sigma_\infty$. Set
    $$f_\xi:=\sum_{\sigma_\infty\in \Pi_\infty(\xi)} f_{\sigma_\infty}.$$ Then $\tr \pi_\infty(f_\xi)\neq 0$ if and only if $\pi_\infty\in \Pi_\infty(\xi)$, in which case the trace equals $1$. (The only if part follows from the results of Vogan--Zuckerman on Lie algebra cohomology.) Moreover $f_{\xi}$ is an Euler-Poincar\'e function up to a nonzero constant, cf. \cite[Lem 3.2]{Kot92b}.
    Write $f_{\sigma_S}\in C^\infty_c(G(F_S))$ for the product of explicit supercuspidal coefficients $f_{\sigma_v}$ (see \S\ref{sub:explicit-coeff}).

\begin{lem}\label{l:LHS=T_ell} Let $\phi^\fkS\in \cH^{\ur}(G(\A_F^\fkS))$. Put $\phi:=\phi^\fkS f_{\sigma_S} \triv_{K_{S_0}} f_\xi$. Then
  $$\sum_{\pi\in \cF(\xi,\sigma_S,K_{S_0})}
\tr \pi^\fkS\left(\phi^\fkS, \mu^{\can,\fkS}\right)=\mu^{\can}_{S_0}(K_{S_0})^{-1}T_\elp(\phi,\mu^{\can,\EP}).$$
\end{lem}

\begin{proof}
  It follows from the definition that the left hand side equals
$$\sum_{\pi\in \cA_{\disc}(G)} m_{\disc}(\pi) \tr \pi^\fkS(\phi^\fkS)\tr \pi_S(f_{\sigma_S}) \frac{\tr \pi_{S_0}(\triv_{K_{S_0}})}{\mu^{\can}_{S_0}(K_{S_0})} \tr \pi_\infty(f_\xi),$$
which is none other than $\mu^{\can}_{S_0}(K_{S_0})^{-1}T_{\disc}(\phi)$. We conclude by \eqref{e:simple-TF}.

\end{proof}

\subsection{Bounds on the geometric side}\label{sub:bounds}

  Here we recollect various bounds on the terms appearing on the geometric side, mostly from \cite{ST:Sato-Tate}.
Given each semisimple element $\gamma\in G(F)$, fix a maximal torus $T_\gamma$ in $G$ over $\ol{F}$ containing $\gamma$ and write $\Phi_\gamma$ for the set of roots of $T_\gamma$ in $G$ outside $I_\gamma$, namely the set of roots $\alpha$ in $G$ such that $\alpha(\gamma)\neq 1$. Thus $\Phi_\gamma$ is nonempty if and only if $\gamma \notin Z(F)$.
Define $S_\gamma$ for the following set of finite places of $F$:
$$S_\gamma:=\{v\,:\,  |1-\alpha(\gamma)|\neq 1~\textup{for~some}~\alpha\in \Phi_\gamma\}.$$
Evidently $S_\gamma$ is independent of the choice of $T_\gamma$.
In the same way we defined $\Ram(G)$, we have the set $\Ram(I_\gamma)$.
For each $v\notin \Ram(G)$, we have a maximal split torus $A_v$ in $G\otimes_F F_v$ such that $K^\hs_v$  is in a good relative position to $A_v$.

Let $\Sigma\supset \Ram(G)\cup S_\infty$ be a finite set of places of $F$.
Choose $\underline{\kappa}=(\kappa_v)_{v\notin \Sigma}$ with $\kappa_v\in \Z_{\ge 0}$ such that $\kappa_v\neq 0$ for only finitely many $v$.
Define
$$Q:=\prod_{v} q_v^{\min(1,\kappa_v)},\qquad Q^{\underline{\kappa}}:=\prod_{v} q_v^{\kappa_v},$$
 where $v$ runs over places of $F$ outside $\Sigma$.
For simplicity we will write $Q^{a+b\underline{\kappa}}$ to mean $Q^a (Q^{\underline{\kappa}})^b$.
Put $U_v^{\le \kappa_v}:=\cup_{\|\lambda\|\le \kappa_v} K_v^\hs \lambda(\varpi_v)
K_v^\hs$, where $v\notin \Sigma$ and $\lambda\in X_*(A_v)$.
Here $||\cdot||$ is an $\Omega_G$-invariant norm on $X_*(A_v)$, which depends on the
choice of an $\R$-basis of
$X_*(A_v)_\R$.

Write $d(G_\infty)$ for the cardinality of $\Pi_\infty(\xi)$, which is independent of $\xi$.
Let $q(G_\infty)$ denote the real dimension of $G(F_\infty)$ modulo (any) maximal compact subgroup.
Given a constant $C\ge1$ and $\star\in \{\mathrm{alg},\mathrm{reg}\}$, we define a set
  $$\Irr^{\star}_C(G(F_\infty)):=\{\xi\in \Irr^{\star}(G(F_\infty))\,:\,\max(\xi)/\min(\xi)\le C\},$$
  where $\max(\xi)$ and $\min(\xi)$ are given as follows.
Let $T$ be a maximal torus in $G$ over $\C$ and choose a Borel subgroup $B$ containing $T$.
Let $\lambda_\xi\in X^*(T)$ denote the $B$-dominant weight of $\xi$.
Write $\Phi^+$ for the set of $B$-positive roots of $T$ in $G$, and $\rho$ for the half sum of roots in $\Phi^+$.
Then $\max(\xi)$ (resp.
$\min(\xi)$) is the maximum (resp.
minimum) value of the natural pairing $\lg \alpha,\lambda_\xi+\rho\rg$ as $\alpha$ runs over $\Phi^+$.
The value is independent of the choice of $T$ and $B$.
(We introduce $\Irr^{\star}_C(G(F_\infty))$ so that we can vary $\xi$ in a controlled manner in that set.)

\begin{prop}\label{prop:bounds-prelim}
\begin{enumerate}
  \item There exists $A_1>0$ depending only on $G$ such that the following holds for any $\Sigma$ and $\underline{\kappa}=(\kappa_v)$ as above: Choose $U_\Sigma$ to be a compact subset of $G(F_\Sigma)$. Write $\mathcal{Y}(\underline\kappa)$ for the set of $G(\A_F)$-conjugacy classes of $\gamma\in G(F)_{\semis}$ which meet $\prod_{v\notin \Sigma} U_v^{\le \kappa_v} \times U_\Sigma$. Then $|\mathcal{Y}(\underline\kappa)|=O(Q^{A_1\underline{\kappa}})$.
  \item
Let $\xi\in \Irr^{\alg}(G(F_\infty))$.
If $\gamma\notin G(F_\infty)_{\elp}$ then $\orbi^{G(F_\infty)}_\gamma(f_\xi)=0$. Moreover,
      $$  \orbi^{G(F_\infty)}_z(f_\xi,\mu^{\EP}_\infty)= (-1)^{q(G_\infty)}\omega_\xi(z) d(G_\infty) \dim \xi,\quad\mbox{if}~z\in Z(F_\infty).$$

For every $C\ge1$, $\xi\in \Irr^{\alg}_C(G(F_\infty))$ and $\gamma\in
G(F_\infty)_{\semis}$ with $\gamma\not\in Z(F_\infty)$, we have
 $$D_\infty(\gamma)^{1/2}|\orbi^{G(F_\infty)}_\gamma(f_\xi,\mu^{\EP}_\infty)|=
 O_C(\dim(\xi)^{1-\nu_\infty}),$$
      where $\nu_\infty\in \R_{>0}$ depends only on $G(F_\infty)$ and the implicit constant is independent of $\gamma$ and $\xi$.
One can choose $\nu_\infty$ to be the minimum of $1-(\dim_\R I_\gamma-\rk_{\R}I_{\gamma})/(\dim_\R G_\infty-\rk_{\R} G_\infty)$ as $\gamma$ runs over noncentral elements in $G(F_\infty)_{\semis}$.
  \item Let $v\notin S_\gamma\cup \Ram(G)$. Suppose that $\gamma\in G(F_v)_{\semis}$ is conjugate to an element of $K_v^{\hs}$. Then $I_\gamma$ is unramified over $F_v$, and $\orbi_\gamma(\triv_{K^\hs_v},\mu^{\can}_v)=1$.
  \item There exists a lower bound $p_0>0$ and $A_3,B_3>0$ depending only on $G$ such that for every finite place $v$ whose residue characteristic is greater than $p_0$, for every $\gamma\in G(F_v)_{\semis}$, and for every $\lambda\in X_*(A_v)$ with $\|\lambda\|\le \kappa_v$,
      $$D(\gamma)^{1/2}\orbi_\gamma(\triv_{K^\hs_v \lambda(\varpi_v) K^\hs_v},\mu^{\can}_v)\le q_v^{A_3+B_3\kappa_v}.$$

\end{enumerate}

\end{prop}

\begin{proof}
  Part (i) follows from \cite[Prop 8.7]{ST:Sato-Tate}. (Take $S_0$ and $S_1$ there to be our $\Sigma\backslash S_\infty$ and $\{v\notin \Sigma: \kappa_v\neq 0\}$, respectively. The proposition there assumes that the nonzero values of $\kappa_v$ are all equal, but the same proof works when $\kappa_v$ are different. Finally observe that $A_3$ can be absorbed into $B_3$ in that proposition.)

 Let us prove (ii). %
  It is a standard fact (\cite[p.659]{Kot92b}) for a discrete series representation $\pi_\infty$ that $\orbi_\gamma(f_{\pi_\infty})$ vanishes unless $\gamma$ is elliptic semisimple, in which case $$\orbi_\gamma(f_{\pi_\infty})=(-1)^{q(G_\infty)}\tr\xi(\gamma)$$
  if the Euler-Poincar\'e measures are used on $G(F_\infty)$ and $I_\gamma(F_\infty)$. This implies everything but the last bound in (ii) as $f_\xi$ is the sum of $f_{\sigma_\infty}$ over $\sigma_\infty\in \Pi_\infty(\xi)$. It remains to bound $D^G_\infty(\gamma)^{1/2}|\tr\xi(\gamma)|=O(\dim(\xi)^{1-\nu_\infty})$ with $\nu_\infty$ described as in the proposition. We get the bound from \cite[Lem 6.10.(ii)]{ST:Sato-Tate}, observing that $m(\xi)$ there is equal to $\min(\xi)$ and that $\dim (\xi)/\min(\xi)^{|\Phi^+|}$ is bounded both above and below (in terms of $C$).

  We get (iii) from \cite[Prop 7.1, Cor 7.3]{Kot86b}. Finally (iv) is proved by motivic integration in \cite[Thm 14.1]{ST:Sato-Tate}.

\end{proof}

  Write $\Mot_{I_\gamma}$ for the Artin-Tate motive associated to $I_\gamma$ by Gross \cite{Gro97}. We have the decomposition $\Mot_{I_\gamma}=\oplus_{d\in \Z_{\ge1}} \Mot_{I_\gamma,d}(1-d)$, where $\Mot_{I_\gamma,d}$ is an Artin motive, and $(1-d)$ denotes the Tate twist. For any Artin-Tate motive $M$ over $F$, denote by $L(\Mot_{I_\gamma})$ (resp. $L_v(\Mot_{I_\gamma})$) the global (resp. local) $L$-function evaluated at $s=0$. Write $\Omega_{I_\gamma}$ for the absolute Weyl group of $I_\gamma$. There is a certain local cohomological invariant $c_v(I_\gamma)\in \Q_{>0}$ defined in \cite[(8.1)]{Gro97}. We do not need the definition but only the property that
  \begin{equation}\label{eq:c_v}
  c_v(I_\gamma)=|H^1(F_v,I_\gamma)|\ge 1~\textup{if}~v\nmid\infty,\quad  c_v(I_\gamma)\ge |\Omega_{I_\gamma}|^{-1},~\textup{if}~v|\infty.
  \end{equation}
  Let $\gamma\in G(F)_{\elp}$ so that $I_\gamma$ has $F$-anisotropic center. (This ensures that $\mu^{\can,\EP}_{I_\gamma}(I_\gamma)$ is finite, cf. \cite[Prop 9.4]{Gro97}.)
   We have the identity \cite[Thm 9.9]{Gro97}
\begin{equation}\label{e:volume-I_gamma}
  |\mu^{\can,\EP}_{I_\gamma}(I_\gamma)|= \frac{|L(\Mot_{I_\gamma})|}{\prod_{v\in S} |L_v(\Mot_{I_\gamma})|} \frac{\tau(I_\gamma)\cdot |\Omega_{I_\gamma}|}{\prod_{v\in S\cup S_\infty} c_v(I_\gamma)}.
\end{equation}
\begin{lem}\label{l:bound-volume} Let $\gamma$ vary over the set $G(F)_{\elp}$ and retain the above notation. Let $S_{\mathrm{iso}}$ denote the set of finite places $v$ of $F$ such that $Z$ contains a nontrivial $F_v$-split torus. (So $S\cap S_{\mathrm{iso}}=\emptyset$.)
  \begin{enumerate}
  \item There exist constants $0<c_1 < c_2 $ such that for all $\gamma\in G(F)_{\semis}$,
  \begin{eqnarray}
    |L_v(\Mot_{I_\gamma})|^{-1}& \le& c_2 q_v^{\frac12(\dim I_\gamma-\rk I_\gamma)}\quad \forall v\notin S_{\mathrm{iso}}\cup S_\infty,\nonumber\\
    |L_v(\Mot_{I_\gamma})|^{-1} &\ge& c_1 q_v^{\frac12(\dim I_\gamma-\rk I_\gamma)}\quad \forall v\notin S_{\mathrm{iso}}\cup S_\infty~\mbox{such that}~I_\gamma~\mbox{is unramified at}~v.\nonumber
  \end{eqnarray}
    \item There exist constants $c_0,c_3,A_2>0$ depending only on $G$ such that the following holds: for all $S$ such that $G$ is unramified at all places in $S$, we have
  $$\frac{|\mu^{\can,\EP}_{I_\gamma}(I_\gamma)|}{|\mu^{\can,\EP}_{G}(G)|} \le c_0c_3^{|S|}\cdot q_S^{\frac12(\dim I_\gamma-\rk I_\gamma)-\frac12(\dim G-\rk G)}\prod_{v\in \Ram(I_\gamma)} q_v^{A_2} ,\quad \forall\gamma\in G(F)_{\semis}.$$
  \end{enumerate}
\end{lem}

\begin{proof}
  (i) Observe that
  $$|L_v(\Mot_{I_\gamma})|^{-1}=\prod_{d\ge 1} |L_v(\Mot_{I_\gamma,d}(1-d))|^{-1}=\prod_{d\ge1}\prod_{i\in I_d} |1-a_{d,i}|,$$
  where $I_d$ is a finite index set, and $a_{d,i}\in \C$ has absolute value $q^{d-1}$ for all $i$. (Since $v\notin S_{\mathrm{iso}}$ we always have $a_{d,i}\neq 1$.) For each $d$ and $i$ we have the obvious bounds $q^{d-1}-1\le |1-a_{d,i}|\le q^{d-1}+1$. Part (i) is now easily deduced from the following facts \cite[\S1]{Gro97}, cf. \cite[Prop 6.3]{ST:Sato-Tate}:
  \begin{itemize}
    \item $|I_d|\le \dim \Mot_{I_\gamma,d}$ with equality when $I_\gamma$ is unramified at $v$,
    \item $\sum_{d\ge 1}|I_d|\le \dim \Mot_{I_\gamma}=\rk I_\gamma$,
    \item $\sum_{d\ge 1} (d-1)\dim \Mot_{I_\gamma,d}=\frac12 (\dim I_\gamma-\rk I_\gamma)$.
  \end{itemize}
(ii)  From \eqref{e:volume-I_gamma} (applied one more time with $\gamma=1$) we have the bound
  $$\frac{|\mu^{\can,\EP}_{I_\gamma}(I_\gamma)|}{|\mu^{\can,\EP}_{G}(G)|} = O\left( |L(\Mot_{I_\gamma})|\prod_{v\in S}\frac{c_v(G)|L_v(\Mot_{G})|}{ |L_v(\Mot_{I_\gamma})| }\right).$$
  Using (i) we bound
  $$|L_v(\Mot_{G})|/ |L_v(\Mot_{I_\gamma})|=O\left(c_2 q_v^{\frac12(\dim I_\gamma-\rk I_\gamma)-\frac12(\dim G-\rk G)}\right).$$
   Note that the implicit constants in both $O(\cdot)$ depend only on $G$, not on $\gamma$ or $S$. (This is obvious in the latter. For the former, it is enough to observe that $|\Omega_{I_\gamma}|$ and $\tau(I_\gamma)$ are uniformly bounded as $\gamma$ varies. The Weyl group is bounded by $|\Omega_{I_\gamma}|\le |\Omega_G|$. The Tamagawa measure formula \cite{Kot88} tells us that $\tau(I_\gamma)\le |Z(\hat I_\gamma)^{\Gal(\ol F/F)}|$, and the latter is uniformly bounded by \cite[Cor 8.12]{ST:Sato-Tate}.)
  From \eqref{eq:c_v} and \cite[Thm 1.2]{Kot86b}, we have
  $$c_v(G)=|H^1(F_v,G)|=|\pi_0(Z(\hat G)^{\Gal(\ol F_v/F_v)})|,\qquad v\nmid \infty.$$
  Since $\Gal(\ol F/F)$ acts on $Z(\hat G)$ through a finite quotient, say
  $\Gal(F'/F)$, if we take $c'_2\in \Z_{\ge 1}$ be the maximum of $|\pi_0(Z(\hat
  G)^{H})|$ as $H$ runs over all subgroups of $\Gal(F'/F)$, then clearly $c_v(G)\le
  c'_2$. Note that $c'_2$ is independent of $v$.
  Finally we have $$|L(\Mot_{I_\gamma})|=O\left(\prod_{v\in \Ram(I_\gamma)} q_v^{A_2}\right)$$ for a uniform constant $A_2>0$, with an implicit constant of $O(\cdot)$ which is uniform for all $\gamma$, by \cite[Cor 6.16]{ST:Sato-Tate}. Now the bound of (ii) follows by putting $c_3:=c_2c'_2$.
\end{proof}

\subsection{Equidistribution results}\label{sub:equidistribution}
  Fix $S$, $S_0$, and $K_{S_0}$. %
  We keep the notation from the previous subsection with $\mathfrak{S}=S_\infty \cup S_0\cup S$. Throughout this subsection we suppose that
\begin{itemize}
  \item the residue characteristic of every $v\in S$ is sufficiently large such that Theorem \ref{t:asymptotic-orb-int} applies to $G(F_v)$ at each $v\in S$,
  \item $S\cup S_0$ contains the places of $F$ with small residue characteristics such that the result by Cluckers--Gordon--Halupczok on uniform bound on orbital integrals \cite[Thm 14.1]{ST:Sato-Tate} applies to places outside $S\cup S_0$.
\end{itemize}
  Note that the lower bound on the residue characteristic for Theorem
  \ref{t:asymptotic-orb-int} can be made effective, whereas the lower bound for \cite[Thm
  14.1]{ST:Sato-Tate} to hold is ineffective by the nature of its proof.

  Our interest lies in statistics of the family $\cF=\cF(\xi,\sigma_{S},K_{S_0})$ as
  $\xi$ and $\sigma_S$ vary.
   Write $K^{\mathfrak{S}}:=\prod_{v\notin {\mathfrak{S}}} K_v^\hs$. Since we are
   assuming $Z(F_S)$ and $Z(F_\infty)$ are compact, the intersection $Z(F)\cap
   K_{S_0}K^{\mathfrak{S}}$ (taken in $G(\A_F)$) is finite. (The same is true with any
   compact $K^{\mathfrak{S}}$-bi-invariant subset of $G(\A_F^{\mathfrak{S}})$ in place of
   $K^{\mathfrak{S}}$.) As in~\cite[\S2.3]{ST:Sato-Tate}, define the truncated Hecke
   algebra
$\cH^{\ur,\le \underline{\kappa}}(G(\A_F^{\mathfrak{S}}))$ as the space of locally
constant
bi-$K^{\mathfrak{S}}$-invariant functions on $G(\A_F^{\mathfrak{S}})$ whose support is
inside the compact subset $\prod_{v\not\in \mathfrak{S}}U_v^{\le \kappa_v}$.

\begin{thm}\label{t:disc-series-case}
There exist constants
$\nu_S,\nu_\infty>0$ and $A>0$ such that for every $\xi\in \Irr^{\reg}_C(G(F_\infty))$,
for every $\sigma_S\in \Irr^{\Yu}(G(F_S))$, for every
$\underline\kappa=(\kappa_v)_{v\notin \mathfrak{S}}$, and for every
$\phi^{\mathfrak{S}}\in \cH^{\ur,\le \underline{\kappa}}(G(\A_F^{\mathfrak{S}}))$
 which is the characteristic function of a bi-$K^{\mathfrak{S}}$-invariant compact subset,
  \begin{eqnarray}
  \sum_{\pi\in \cF}
\tr \pi^{\mathfrak{S}}(\phi^{\mathfrak{S}})
&=& (-1)^{q(G_\infty)}d(G_\infty)\dim (\xi)\deg(\sigma_S)\frac{\tau'(G,S)}{\mu^{\can}_{S_0}(K_{S_0})} \sum_{z\in Z(F)\cap K_{S_0}} \frac{\omega_\xi(z)}{\omega_{\sigma_{S}}(z)}\phi^{\mathfrak{S}}(z)\nonumber\\
&&+~ O(\dim(\xi)^{1-\nu_\infty}\deg(\sigma_S)^{1-\nu_S} Q^{A\underline{\kappa}} ).\label{e:in-thm-disc-series-case}
\end{eqnarray}
  The implicit constant in $O(\cdot)$ depends on $G$, $S$, $S_0$, $K_{S_0}$, and $C$ (but is independent of $\xi$, $\sigma_S$, $\underline{\kappa}$, and $\phi^{\mathfrak{S}}$).
\end{thm}

\begin{rem}\label{r:epsilon_infty}
  The proof shows that $\nu_\infty$ can be chosen to be as in Proposition \ref{prop:bounds-prelim}. We have restricted to the set $\xi\in \Irr^{\reg}_C(G(F_\infty))$ to underline the analogy between the finite places $S$ and the infinite places $S_\infty$. Without the restriction the error bound could be stated in terms of $\min(\xi)$ as in \cite[Thm 9.19]{ST:Sato-Tate}. The same remark applies to Theorem \ref{th:steinbergs} below.
\end{rem}

\begin{rem}
  If we fix $\underline\kappa$ and $\phi^{\mathfrak{S}}$ (while allowing $\xi$ and
  $\sigma_S$ to vary) then the same proof shows the asymptotic formula with error bound
  $O(\dim(\xi)^{1-\nu_\infty}\deg(\sigma_S)^{1-\nu_S})$. This holds under a weaker
  assumption on $S$, namely that the residue characteristic of each $v\in S$ has to be
  large enough so that only Theorem \ref{t:asymptotic-orb-int}, but not \cite[Thm
  14.1]{ST:Sato-Tate}, applies. Hence the lower bound for the residue characteristic can
  be explicitly given. Thus an explicit lower bound is possible for Corollaries
  \ref{cor:limit-multiplicity}, \ref{c:existence-supercuspidal}, and \ref{c:St-case}.
\end{rem}

\begin{proof}
  By Lemma \ref{l:LHS=T_ell} the left hand side is equal to
  $$\mu^{\can}_{S_0}(K_{S_0})^{-1}\sum_{\gamma\in G(F)/\sim\atop \mathrm{elliptic}} \iota(\gamma)^{-1}\mu^{\can,\EP}_{I_\gamma}(I_\gamma)\orbi^{G(\A_F)}_\gamma(\phi^{\mathfrak{S}} f_{\sigma_S} \triv_{K_{S_0}} f_\xi).$$
  For each $v\in S$, write $U_v$ for the (finite) union of a set of representatives for $G(F_v)$-conjugacy classes of elliptic maximal tori in $G(F_v)$. Take $U_v$ to be an elliptic maximal torus in $G(F_v)$ for infinite places $v\in S_\infty$. Thus $U_v$ is compact in both cases.
  Clearly the summand in the preceding formula vanishes unless $\gamma\in \mathcal{Y}(\underline\kappa)$, where $\mathcal{Y}(\underline\kappa)$ is as in Proposition \ref{prop:bounds-prelim} taking $\Sigma=\mathfrak{S}$ and $U_\Sigma=K_{S_0}\times\prod_{v\in S\cup S_\infty} U_v$.

  The contribution from central elements $z\in Z(F)$ is computed as in the first line on the right hand side of \eqref{e:in-thm-disc-series-case}.
For this it is enough to observe that $\mu_G^{\can,\EP}(G)=\tau'(G,S)$ by definition, $f_{\sigma_S}(z)=\omega_{\sigma_S}^{-1}(z)\deg(\sigma_S)$ by Harish-Chandra's Plancherel theorem, and $f_\xi(z)=(-1)^{q(G_\infty)}\omega_\xi(z)\dim \xi$ by Proposition \ref{prop:bounds-prelim}.
This implies that the contribution of $\gamma\in Z(F)$ in the above sum equals the main term in the right hand side of \eqref{e:in-thm-disc-series-case}.
  Hence it suffices to show that for some uniform constants $\nu_S,\nu_\infty,A>0$,
  \begin{equation}\label{e:disc-series-case}
  \sum_{\gamma\in \mathcal{Y}(\underline\kappa)\atop \mathrm{s.t.}~ \gamma\notin Z(F)} \left|\mu^{\can,\EP}_{I_\gamma}(I_\gamma)\orbi^{G(\A_F)}_\gamma(\phi^{\mathfrak{S}} f_{\sigma_S} \triv_{K_{S_0}} f_\xi)\right|
  = O(\dim\xi^{1-\nu_\infty}\deg(\sigma_S)^{1-\nu_S} Q^{A\underline{\kappa}} ).
  \end{equation}

  We will bound the summand for each $\gamma\in \mathcal{Y}(\underline\kappa)$.
Without loss of generality we assume that $\gamma$ belongs to $\prod_{v\notin \mathfrak{S}} U_v^{\le \kappa_v} \times K_{S_0}\times \prod_{v\in S\cup S_\infty} U_v$.
Define $\Phi_\gamma$ and $S_\gamma$ as in the last subsection (with ${\mathfrak{S}}=S_0\cup S\cup S_\infty$).
Write $S'_\gamma$ for the set of $v\in S_\gamma$ with $v\notin \Sigma$.
The subset of $v\in S'_\gamma$ with $\kappa_v=0$ is written by $S'_{\gamma,0}$.

According to Lemma \ref{l:bound-volume} (iii), the group $I_\gamma$ is unramified at $v$ if $v\notin \Ram(G)$, $v\notin S_\gamma$, and $\kappa_v=0$. (The last condition ensures that the $v$-component of $\phi^{\mathfrak{S}}$ is supported on $K_v^{\hs}$ so that $\gamma$ is conjugate to $K_v^{\hs}$.) Hence
$$\Ram(I_\gamma)\subset \Ram(G)\cup S_\gamma\cup \{v\notin \Sigma: \kappa_v\neq 0\}\subset  S\cup S_0\cup S'_{\gamma} \cup \{v\notin \Sigma: \kappa_v\neq 0\}.$$
Now Lemma \ref{l:bound-volume}, Proposition \ref{prop:bounds-prelim}, Theorem~\ref{t:asymptotic-orb-int} and Proposition~\ref{p:Kottwitz} tell us that there exist positive constants $c_0,c_S,c_{S_0},c_\infty,A_2,A_3,B_3,\nu_S,\nu_\infty$ such that
\begin{eqnarray}
  |\mu^{\can,\EP}_{I_\gamma}(I_\gamma)| &\le& c_0 c_3^{|S|} \prod_{v\in \Ram(G)\cup S_\gamma\cup \{v\notin \Sigma: \kappa_v\neq 0\}} q_{v}^{A_2}\le c_0 q_S^{A_2} q_{S_0}^{A_2} q_{S'_{\gamma}}^{A_2}Q^{A_2},\label{e:disc-series-case-1}\\
  D^{\mathfrak{S}}(\gamma)^{1/2} |\orbi_\gamma(\phi^{\mathfrak{S}})| &\le& q_{S'_\gamma,0}^{A_3}Q^{A_3+B_3 \underline{\kappa}}, \nonumber \\
  D_S(\gamma)^{1/2}  |\orbi_\gamma(f_{\sigma_S})| &\le & c_S \deg(\sigma_S)^{1-\nu_S},\label{e:disc-series-case-2}\\
  D_{S_0}(\gamma)^{1/2} |\orbi_\gamma(\triv_{K_{S_0}})| & \le & c_{S_0}, \nonumber\\
  D_\infty(\gamma)^{1/2} |\orbi_\gamma(f_\xi)| & \le & c_\infty \dim (\xi)^{1-\nu_\infty}.
\nonumber
\end{eqnarray}
At this point, the main remaining task is to bound $q_{S'_{\gamma}}$ and $q_{S'_{\gamma,0}}$ independently of $\gamma$.

  We introduce some invariants of the group $G$ over $F$.
Write $d_G$ for the dimension of $G$, $w_G$ for the order of the absolute Weyl group, and $s_G$ for the minimal degree over $F$ of an extension field over which $G$ splits.

  Define $\delta_\infty$ to be the supremum of $ \prod_{v\in S_\infty}|1-\alpha(\gamma_\infty)|_v$ as $\gamma_\infty$ runs over $\prod_{v\in S_\infty} U_v$, and as $\alpha$ runs over the set of absolute roots for the compact maximal torus $\prod_{v\in S_\infty} U_v$ in $G(F_\infty)$. (It makes no difference if we impose $\alpha(\gamma_\infty)\neq 1$.) Then $\delta_\infty<\infty$ as the supremum of a continuous function on a compact set is finite.
     Likewise define
  \begin{equation}\label{eq:delta_v}
  \delta_v:=\sup_{\gamma_v,\, \alpha}  |1-\alpha(\gamma_v)|_v,  \qquad v\in S\cup S_0,
  \end{equation}
   as $\gamma_v$ runs over the $v$-component of $K_{S_0}$ if $v\in S_0$, and as $\gamma_v$ runs over elliptic semisimple elements of $G(F_v)$ if $v\in S$,
   and $\alpha$ runs over the set of absolute roots of a maximal torus containing $\gamma_v$ in $G$. By definition $\delta_v$ with $v\in S\cup S_0$ depends only on $G$, $S$, $S_0$, and $K_{S_0}$. We see that $\delta_v<\infty$ for $v\in S\cup S_0$ by continuity and compactness as for $\delta_\infty$. (For instance when $v\in S_0$, there are only finitely many $G(F_v)$-conjugacy classes of elliptic maxial tori, which are compact as $G(F_v)$ has compact center. For each compact maximal torus, the value $|1-\alpha(\gamma)|_v$ is bounded as $\alpha$ varies over the set of absolute roots and $\gamma$ on the maximal torus.) For our purpose below, we may and will arrange that $\delta_v,\delta_\infty\ge 1$ by redefining each of them to be 1 if smaller than 1.

  Noting that $|1-\alpha(\gamma)|_v=1$ at $v\notin S_\gamma\cup S_\infty$ with $\alpha\in \Phi_\gamma$, we deduce from the product formula that
  \begin{equation}\label{e:product-1-alpha(gamma)}
  1=\prod_v |1-\alpha(\gamma)|_v \le \delta_\infty \prod_{v\in S_\gamma} |1-\alpha(\gamma)|_v, \quad \forall \alpha\in \Phi_\gamma.
  \end{equation}
  By \cite[Lem 2.18]{ST:Sato-Tate}, for some $B_1>0$ which is independent of $\gamma$ and~$\underline{\kappa}$,
  \begin{equation}\label{e:bound-1-alpha(gamma)-by-Q}
  \prod_{v\in S'_\gamma\backslash S'_{\gamma,0}} |1-\alpha(\gamma)|_v \le Q^{B_1\underline\kappa},\quad \forall \alpha\in \Phi_\gamma.
   \end{equation}
   For each $v\in S'_{\gamma,0}$, we have
   $|1-\alpha(\gamma)|_v\le 1$ for every $\alpha\in \Phi_\gamma$.
   By the definition of $S'_{\gamma,0}$, there exists $\alpha\in \Phi_\gamma$ such that $|1-\alpha(\gamma)|_v<1$. The argument as in the proof of \cite[Prop 8.7]{ST:Sato-Tate} (also see the proof of Theorem 9.19 there) shows that a fortiori $|1-\alpha(\gamma)|_v\le q_{v}^{-\frac{1}{w_G s_G}}$. Hence
  \begin{equation}\label{e:bound-1-alpha(gamma)}
  \prod_{\alpha\in \Phi_\gamma} |1-\alpha(\gamma)|_v \le q_{v}^{-\frac{1}{w_G s_G}}, \quad v\in S'_{\gamma,0}.
  \end{equation}
 Taking the product of \eqref{e:product-1-alpha(gamma)} over $\alpha\in \Phi_\gamma$ and applying the estimates \eqref{eq:delta_v} at $v\in S_\gamma\backslash S'_\gamma$ (which is contained in $S\cup S_0$), \eqref{e:bound-1-alpha(gamma)-by-Q} at $v\in S'_\gamma\backslash S'_{\gamma,0}$, and \eqref{e:bound-1-alpha(gamma)} at $v\in S'_{\gamma,0}$,  we see that
  $$1\le q_{S'_{\gamma,0}}^{-\frac{1}{w_G s_G}} \left( Q^{B_1\underline\kappa} \delta_\infty\prod_{v\in S_\gamma\backslash S'_{\gamma}} \delta_v\right)^{|\Phi_\gamma|} \le q_{S'_{\gamma,0}}^{-\frac{1}{w_G s_G}} \left( Q^{B_1\underline\kappa} \delta_\infty\prod_{v\in S\cup S_0} \delta_v\right)^{d_G}.$$
  Therefore $q_{S'_{\gamma,0}}=O(Q^{w_G s_G d_G B_1\underline\kappa})$, implying that
  $$q_{S'_\gamma} \le  Q\cdot q_{S'_\gamma,0} = O(Q^{1+w_G s_G d_G B_1\underline\kappa}) .$$
   To summarize so far, the absolute value of each summand in \eqref{e:disc-series-case} is bounded by, if we set $C:=c_0 c_3^{|S|}c_S c_{S_0}c_\infty (q_S q_{S_0})^{A_2}$, the following:
     $$C q_{S'_{\gamma}}^{A_2}q_{S'_\gamma,0}^{A_3}Q^{A_2+A_3+B_3 \underline{\kappa}}\deg(\sigma_S)^{1-\nu_S}\dim (\xi)^{1-\nu_\infty}.$$
     Applying the above bounds on $q_{S'_\gamma}$ and $q_{S'_\gamma,0}$, the summand admits a bound of the form $O(Q^{A_4+B_4\underline\kappa}\deg(\sigma_S)^{1-\nu_S}\dim (\xi)^{1-\nu_\infty})$.
  The number of nonzero summands is bounded as $O(Q^{A_1\underline\kappa})$ by Proposition \ref{prop:bounds-prelim}. All in all, the absolute value of the left hand side of \eqref{e:disc-series-case} is
  $$ O(Q^{A_4+(A_1+B_4)\underline\kappa}\deg(\sigma_S)^{1-\nu_S}\dim (\xi)^{1-\nu_\infty}).$$
  The proof is complete by taking $A=A_4+A_1+B_4$. (Observe that $Q\le Q^{\underline{\kappa}}$.)
\end{proof}

\begin{rem}
  An affirmative answer to the question in Remark \ref{r:disc-analog-orb-int} would immediately improve Theorem \ref{t:disc-series-case} with the hypotheses relaxed accordingly, by exactly the same argument.
\end{rem}

   Consider the set of pairs $(\xi,\sigma_S)\in \Irr^{\reg}_C(G(F_\infty))\times \Irr^{\Yu}(G(F_S))$.
  We partition the set into $\mathcal P_=$ and $\mathcal P_{\neq}$ according as whether $\omega_{\xi}=\omega_{\sigma_S}$ or not on $Z(F)\cap K_{S_0}K^{\mathfrak{S}}$. Recall that $Z(F)\cap K_{S_0}K^{\mathfrak{S}}$ is finite.

\begin{cor}\label{cor:limit-multiplicity}
  We have the limit multiplicity formulas
\begin{eqnarray}
  \lim_{(\xi,\sigma_S)\in \mathcal P_{\neq}\atop \dim(\xi)|\deg(\sigma_S)|\ra\infty} \frac{m(\xi,\sigma_S,K_{S_0})}{d(G_\infty)\dim(\xi)\deg(\sigma_S)} &=&0,\nonumber\\
\lim_{(\xi,\sigma_S)\in \mathcal P_{=}\atop \dim(\xi)|\deg(\sigma_S)|\ra\infty} \frac{m(\xi,\sigma_S,K_{S_0})}{d(G_\infty)\dim(\xi) \deg(\sigma_S)}
&= & (-1)^{q(G_\infty)}\frac{\tau'(G,S)}{\mu^{\can}_{S_0}(K_{S_0})} |Z(F)\cap K_{S_0}K^{\mathfrak{S}}|.\nonumber
\end{eqnarray}

\end{cor}

\begin{rem}
  It is clear from the definition that $m(\xi,\sigma_S,K_{S_0})\ge 0$. This is consistent with the signs in the second formula above. Indeed the sign of the measure $\mu^{\EP}_\infty\mu^{\EP}_S $ is the same as that of $\tau'(G,S)$ since the canonical measure $\mu^{\can}_v$ is a positive measure for $v\notin S\cup S_\infty$. The signs of $\mu^{\EP}_\infty$ and $\mu^{\EP}_S$ are the same as those of $(-1)^{q(G_\infty)}$ and $\deg(\sigma_S)$, respectively.
\end{rem}

\begin{proof}
  We start by claiming that $\dim(\xi)\deg(\sigma_S)\ra\infty$ if and only if $\dim(\xi)^{\nu_\infty}\deg(\sigma_S)^{\nu_S}$ tends to infinity. To see this, we partition the set of all $\sigma_S$ into two sets where $|\deg(\sigma_S)|\ge 1$ and $|\deg(\sigma_S)|<1$. The claim is obvious in the first set. (Recall that $\nu_\infty,\nu_S>0$ and note that $\dim(\xi)\ge 1$.) In the second set, the claim follows from the fact \cite[Thm 7]{HC99} that $\deg(\sigma_S)$ is bounded below by a positive constant. (More precisely $\deg(\sigma_S)$ is an integral multiple of a constant depending only on a Haar measure on $G(F_S)$.)

  Now the corollary readily follows from the preceding theorem by plugging in $\kappa_v=0$ for all $v\notin {\mathfrak{S}}$ and $\phi^{\mathfrak{S}}=\triv_{K^{\mathfrak{S}}}$.
\end{proof}

  We can restate Theorem \ref{t:disc-series-case} in terms of $m(\xi,\sigma_S,K_{S_0})$, assuming $G$ is split and semisimple for simplicity. The $L^1$-norm of $\phi^{\mathfrak{S}}\in \cH^{\ur}(G(\A_F^{\mathfrak{S}}))$ is given by $\|\phi^{\mathfrak{S}}\|_1:=\int_{G(\A_F^{\mathfrak{S}})} \left|\phi^{\mathfrak{S}}(g)\right| d\mu^{\can,{\mathfrak{S}}}$.

\begin{cor}\label{c:disc-seires-case-restated} Suppose that $G$ is a split semisimple group over $F$. There exist constants $\epsilon,A>0$ such that for every $(\xi,\sigma_S)\in \mathcal P_=$ and for every $\phi^{\mathfrak{S}}=\triv_{K^{\mathfrak{S}} g K^{\mathfrak{S}}}\in \cH^{\ur}(G(\A_F^{\mathfrak{S}}))$ with some $g\in G(\A_F^{\mathfrak{S}})$,
 $$
  \sum_{\pi\in \cF(\xi,\sigma_S,K_{S_0})}
\tr \pi^{\mathfrak{S}}(\phi^{\mathfrak{S}})
=  \frac{m(\xi,\sigma_S,K_{S_0})}{|Z(F)\cap K_{S_0}K^{\mathfrak{S}}|}\sum_{z\in Z(F)\cap K_{S_0}} \phi^{\mathfrak{S}}(z)
+ O(m(\xi,\sigma_S,K_{S_0})^{1-\nu} \|\phi^{\mathfrak{S}}\|_1^{A} ).
$$
  The implicit constant in $O(\cdot)$ depends on $G$, $S$, $S_0$, $K_{S_0}$, and $C$ (but is independent of $\xi$, $\sigma_S$, and $\phi^{\mathfrak{S}}$).
\end{cor}

\begin{proof}
  Since $G$ has finite center, the center is contained in every maximal compact subgroup. So $Z(F)\cap K_{S_0}=Z(F)\cap K_{S_0}K^{\mathfrak{S}}$ and $\omega_{\sigma_S}=\omega_\xi$ on $Z(F)\cap K_{S_0}$ for $(\xi,\sigma_S)\in \mathcal P_=$. Our task is to turn the right hand side of \eqref{e:in-thm-disc-series-case} to the right hand side as in the corollary. Let $\underline\kappa=(\kappa_v)_{v\notin {\mathfrak{S}}}$ be chosen such that $\phi^{\mathfrak{S}}\in \cH^{\ur,\le\underline\kappa}(G(\A_F^{\mathfrak{S}}))$. We know that
  $$\left| m(\xi,\sigma_S,K_{S_0}) - (-1)^{q(G_\infty)}d(G_\infty)\dim (\xi)\deg(\sigma_S)\frac{\tau'(G,S)|Z(F)\cap K_{S_0}K^{\mathfrak S}|}{\mu^{\can}_{S_0}(K_{S_0})}\right|$$
  $$=~ O(\dim(\xi)^{1-\nu_\infty} \deg(\sigma_S)^{1-\nu_S})$$
  from Theorem \ref{t:disc-series-case}, cf. the proof of Corollary \ref{cor:limit-multiplicity}. So it is enough to show that
   for some constants $A',B'>0$ (whose independence is as in the corollary),
  \begin{enumerate}
    \item $\sum_{z\in Z(F)\cap K_{S_0}} \phi^{\mathfrak{S}}(z) =O( Q^{B'\underline\kappa})$,
    \item $Q^{A\underline\kappa}\le \|\phi^{\mathfrak{S}}\|_1^{A'}$.
  \end{enumerate}

  Part (i) is equivalent to $|Z(F)\cap K_{S_0}(K^{\mathfrak{S}} g K^{\mathfrak{S}})K_\infty|=O( Q^{A'\underline\kappa})$ for any maximal compact subgroup $K_\infty\subset G(F_\infty)$. This immediately follows from \cite[Prop 8.7]{ST:Sato-Tate}.

  It remains to check (ii).
  We adopt the notation about truncated Hecke algebras from \cite[\S2.3]{ST:Sato-Tate}. By Cartan decomposition we may write $g=(g_v)_{v\notin {\mathfrak{S}}} \in G(\A_F^{\mathfrak{S}})$ with $g_v=\mu_v(\varpi_v)$ for the cocharacter $T_v$ of a maximal $F_v$-split torus $T_v$ of $G$ and a uniformizer $\varpi_v$ of $F_v$. We take $\underline\kappa=(\kappa_v)_{v\notin {\mathfrak{S}}}$ such that $\kappa_v=\|\mu_v\|_{\mathcal B}$ for a suitable $\R$-basis $\mathcal B=\{e_1,...,e_r\}$ of $X_*(T)\otimes_\Z \R$, where $X_*(T)$ is the cocharacter group of a maximal torus $T$ over $\ol{F}$. (Thus $r=\dim T=\rk G$.) Of course $\kappa_v=0$ for all but finitely many $v$. Choose a Borel subgroup containing $T$ so that we have a set of positive coroots $\Phi^{\vee,+}$ for $T$. We take $\mathcal B$ to consist of simple coroots in $\Phi^{\vee,+}$. Similarly we choose a Borel subgroup $B_v\supset T_v$ and a set of positive coroots $\Phi^{\vee,+}_v$ for $T_v$. Without loss of generality we assume that $\mu_v$ is $B_v$-dominant. Set $\rho_v^\vee:=\sum_{\alpha^\vee\in \Phi^{\vee,+}_v}\alpha^\vee$. The equality $\kappa_v=\|\mu_v\|_{\mathcal B}$ means that $\kappa_v=\lg \mu_v,i \alpha^\vee\rg$ for some $\alpha^\vee\in \Phi^{\vee,+}$, where $i$ is an inner automorphism of $G$ sending $T$ to $T_v$. Hence for each $v\notin {\mathfrak{S}}$, we see that $\lg \mu_v,\beta_v^\vee\rg=\kappa_v$ for some coroot $\beta_v\in \Phi^{\vee,+}_v$.

  We claim that there exists a constant $c>1$ such that for every $v\notin {\mathfrak{S}}$ we have
  $$\|\triv_{K_v g_v K_v} \|_1 \ge c^{-1} q^{\lg \mu_v,\rho^\vee\rg}.$$
  Indeed \cite[Prop 7.4]{Gro98} tells us that $\|\triv_{K_v g_v K_v} \|_1/q^{\lg \mu_v,\rho^\vee\rg}$ is equal to $|(G/P_{\mu_v})(\F_{q_v})|/q_v^{\dim (G/P_{\mu_v})}$ for a suitable parabolic subgroup $P_{\mu_v}$ of $G$. Of course there are finitely many parabolic subgroups (up to conjugation). We see that the quotient (which is a quotient of two polynomials in $q_v$) tends to one as $q_v\ra\infty$. The claim follows.

  For each $v$ such that $\mu_v\neq 0$ (so $K_v g_v K_v\neq K_v$) we have $\|\triv_{K_v g_v K_v} \|_1\ge 2= c^{\log 2/\log c}$. Setting $c':=\log c/\log 2\in \R_{>0}$, we have (whether $\mu_v=0$ or not)
  $$\|\triv_{K_v g_v K_v} \|^{1+c'}_1 \ge  q^{\lg \mu_v,\rho^\vee\rg},\quad \forall v\notin {\mathfrak{S}}.$$
  Since $\mu_v$ is $B_v$-dominant, it is clear that $\lg \mu_v,\rho^\vee\rg\ge \lg \mu_v,\beta_v^\vee\rg=\kappa_v$.  In conclusion
  $$\|\phi^{\mathfrak{S}} \|^{1+c'}_1 \ge  Q^{\underline{\kappa}}.$$
  The proof of (ii) is complete.
\end{proof}

  Let us record a sample application to the existence of certain automorphic representations. For simplicity we assume that $G$ is split over $F$. We fix a reductive model over the ring of integers $\cO_F$, giving rise to hyperspecial subgroups $K_v^{\hs}$ at each finite place $v$.

\begin{cor}\label{c:existence-supercuspidal} Suppose that $G$ is a split reductive group over $F$.
Fix $\xi\in \Irr^{\reg}(G(F_\infty))$ and $S$ a nonempty finite set of finite places. Suppose that the residue characteristic of each $v\in S$ is sufficiently large in the sense at the start of \S\ref{sub:equidistribution}. Then there exists $d_0>0$ with the following property: For every $\pi^0_S\in \Irr^{\Yu}(G(F_S))$ with $|\deg(\pi^0_S)|\ge d_0$ and $\omega_{\pi^0_S}=\omega_\xi$ on $Z(F)\cap K_{S_0}K^{\mathfrak{S}}$, there exists a cuspidal automorphic representation $\pi$ of $G(\A_F)$ such that
  \begin{itemize}
    \item $\pi_S\simeq \pi^0_S$,
        \item $\pi_\infty\in \Pi_\infty(\xi)$,
    \item $\pi^{S,\infty}$ is unramified.
      \end{itemize}
\end{cor}

\begin{proof}
  In the preceding corollary, it is enough to take $S_0$ to be sufficiently large and contain all places of small residue characteristics, and set $K_{S_0}:=\prod_{v\in S_0} K_v^\hs$. Then fix $\xi$ and let $\deg(\sigma_S)$ go to infinity.
\end{proof}

  According to Corollary \ref{cor:limit-multiplicity}, it is reasonable to restrict our attention to $(\xi,\sigma_S)\in \mathcal P_=$ when studying equidistribution problems on the following counting measure for the multi-set $\cF=\cF(\xi,\sigma_S,K_{S_0})$, where $\delta_{\pi^{{\mathfrak{S}}}}$ denotes the Dirac delta measure supported at $\pi^{\mathfrak{S}}$ on the unramified unitary dual $G(\A_F^{\mathfrak{S}})^{\wedge,\ur}:=\prod_{v\notin {\mathfrak{S}}} G(F_v)^{\wedge,\ur}$:
  $$\hat\mu^{\cnt}_{\cF}:=\frac{1}{|\cF|} \sum_{\pi\in \cF} \delta_{\pi^{{\mathfrak{S}}}}.$$
  Of course this makes sense if $|\cF|\neq 0$. To obtain a clean formula we will further assume that $Z(F)\cap K_{S_0}=Z(F)\cap K_{S_0}K^{\mathfrak{S}}$. (Alternatively we may instead restrict to the pairs $(\xi,\sigma_S)\in \mathcal P_=$ such that $\omega_\xi=\omega_{\sigma_S}$ on $Z(F)\cap K_{S_0}$.)

  Given $z\in Z(F)$ let $\omega_z$ denote the function on $G(\A_F^{\mathfrak{S}})^{\wedge,\ur}$ whose value on each representation is its central character evaluated at $z$. Define a measure $\hat\mu^{\mathrm{pl},\ur,{\mathfrak{S}}}_z$ on $G(\A_F^{\mathfrak{S}})^{\wedge,\ur}$ to be $\omega_z\cdot \hat\mu^{\mathrm{pl},\ur,{\mathfrak{S}}}$, where $\hat\mu^{\mathrm{pl},\ur,{\mathfrak{S}}}:=\prod_{v\notin {\mathfrak{S}}} \hat\mu^{\mathrm{pl},\ur}_v$ is the product of the Plancherel measure  $\hat\mu^{\mathrm{pl},\ur}_v$ on the unramified unitary dual of $G(F_v)$.\footnote{Since we consider only those $\hat\phi^{\mathfrak{S}}$ coming from $\cH^{\ur}(G(\A_F^{\mathfrak{S}}))$, the formulas remain valid if we use the Plancherel measure on the whole unitary dual.} Recall that $\phi^{\mathfrak{S}}$ defines a function $\hat\phi^{\mathfrak{S}}:\pi\mapsto \tr \pi(\phi^{\mathfrak{S}})$ on $\prod_{v\notin {\mathfrak{S}}} G(F_v)^{\wedge,\ur}$.
 Note that we can integrate $\hat\phi^{\mathfrak{S}}$ against the (possibly infinite) sum measure $\hat\mu^{\mathrm{pl},\ur,{\mathfrak{S}}}_{Z(F)\cap K_{S_0}}:=\sum_{z\in Z(F)\cap K_{S_0}} \hat\mu^{\mathrm{pl},\ur,{\mathfrak{S}}}_z$. Indeed $\hat\mu^{\mathrm{pl},\ur,{\mathfrak{S}}}_z(\hat\phi^{\mathfrak{S}})$ is nonzero for only finitely many $z\in Z(F)\cap K_{S_0}$ since $Z(F)$ intersects an open compact subset of $Z(\A_F)$ at only finitely many points.

\begin{cor} Assume that $Z(F)\cap K_{S_0}=Z(F)\cap K_{S_0}K^{\mathfrak{S}}$. For every $\phi^{\mathfrak{S}}\in \cH^{\ur}(G(\A_F^{\mathfrak{S}}))$,
  $$\lim_{\dim(\xi)\deg(\sigma_S)\ra\infty\atop (\xi,\sigma_S)\in \mathcal P_{=}} \hat\mu^{\cnt}_{\cF(\xi,\sigma_S,K_{S_0})} (\hat \phi^{\mathfrak{S}})=\frac{\hat\mu^{\mathrm{pl},\ur,{\mathfrak{S}}}_{Z(F)\cap K_{S_0}}(\hat \phi^{\mathfrak{S}})}{|Z(F)\cap K_{S_0}K^{\mathfrak{S}}|}.$$
 (The counting measure is defined when $\dim(\xi)|\deg(\sigma_S)|\gg 1$ by Corollary \ref{cor:limit-multiplicity}.)
\end{cor}

\begin{proof}
  Choose $\underline\kappa\ge 0$ such that $\phi^{\mathfrak{S}}\in \cH^{\ur,\le \underline\kappa}(G(\A_F^{\mathfrak{S}}))$.
  By \eqref{e:in-thm-disc-series-case}, $\hat \mu^{\cnt}_{\cF(\xi,\sigma_S,K_{S_0})} (\hat \phi^{\mathfrak{S}})$ equals
    $$(-1)^{q(G_\infty)}\frac{\tau'(G,S)d(G_\infty)\dim (\xi)\deg(\sigma_S)}{\mu^\can_{S_0}(K_{S_0})|\cF(\xi,\sigma_S,K_{S_0})|} \sum_{z\in Z(F)\cap K_{S_0}} \phi^{\mathfrak{S}}(z)
  + O\left(\frac{\deg(\sigma_S)^{1-\nu_S}\dim\xi^{1-\nu_\infty}}{|\cF(\xi,\sigma_S,K_{S_0})|} Q^{A\underline{\kappa}} \right).$$
  By Plancherel theorem, $\sum_{z\in Z(F)\cap K_{S_0}} \phi^{\mathfrak{S}}(z)=\hat\mu^{\mathrm{pl},\ur,{\mathfrak{S}}}_{Z(F)\cap K_{S_0}}(\hat \phi^{\mathfrak{S}})$.
  We apply Corollary \ref{cor:limit-multiplicity} to finish the proof.
  \end{proof}

  \begin{rem}\label{r:ST-equidistribution}
  In particular if $Z(F)\cap K_{S_0}=\{1\}$ (e.g. if $Z(F)$ is trivial) then we have
    $$\lim_{\dim(\xi)+\deg(\sigma_S)\ra\infty\atop (\xi,\sigma_S)\in \mathcal P_{=}} \hat\mu^{\cnt}_{\cF(\xi,\sigma_S,K_{S_0})} (\hat \phi^{\mathfrak{S}})
    =\hat\mu^{\mathrm{pl},\ur,{\mathfrak{S}}}(\hat \phi^{\mathfrak{S}}).$$
    This confirms Conjecture 1 in \cite{SST}, or more precisely its analogue as explained in the remark below it. In our case the limiting measure is the product of the unramified Plancherel measures so (i) and (ii) of the conjecture are true. Part (iii) is essentially \cite[Prop 5.3]{ST:Sato-Tate}, from which (iv) follows immediately.
  \end{rem}

  \begin{rem}\label{r:nontrivial-center}
  The results above should carry over to the case where neither $G(F_S)$ nor
  $G(F_\infty)$ has compact center, at least if $G$ is a cuspidal group in the
  sense that the center of $\Res_{F/\Q}G$ has the same split $\Q$-rank and split
  $\R$-rank. This requires some modification in the statements (e.g.
  pseudo-coefficients of a supercuspidal representation have compact support
  only modulo center) but would not lead to any significant change in the
  proof.  Alternatively one could work with the trace formula with fixed
  central character (one could use \cite[\S6]{Dal19} for instance), in which case representations and test functions also have
  fixed central characters which are inverses of each other.
  \end{rem}

\subsection{Possible generalizations} It is sensible to ask whether the method of this paper applies to non-supercuspidal discrete series representations $\pi$ but there are difficulties. In that case we still have a somewhat explicit construction of pseudo-coefficients for $\pi$, cf. \cite[\S3.4]{SS97} but they are not as simple as in \S\ref{sub:explicit-coeff} to be useful. In the trace formula, if we impose a pseudo-coefficient of $\pi$ at a local place $v$ then one still has the simple trace formula, but the spectral side picks up automorphic representations whose $v$-components are not only $\pi$ but possibly nontempered representations in a finite list. This means that one has to control these spectral error terms. Alternatively one could allow more general test functions at $v$ but then the trace formula will have more terms to be dealt with (unless the global reductive group is anisotropic modulo center).

\section{Steinberg representations and horizontal families}\label{s:Steinberg}

The results in this section do not rely on Sections \ref{s:Yu-construction} through \ref{s:proof}. However the reader will see a strong analogy both in the statements and proofs between the vertical families in the last section and the horizontal families in this one.

Let $\xi$, $S$, $S_0$, $S_\infty$, and $\fkS$ be as in \S\ref{sub:counting-measures}, cf. \S\ref{s:quantitative}. In particular the finite sets $S$, $S_0$, and $S_\infty$ are mutually disjoint, and $G$ is unramified outside $\fkS$. We will fix $S_0$ and assume that
\begin{itemize}
  \item $S\cap \Ram(G)=\emptyset$  (so that $\Ram(G)\subset S_0$),
  \item the residue characteristic of every $v\in S$ is sufficiently large.
\end{itemize}
Throughout this section we keep a large residue characteristic assumption as in \S\ref{sub:equidistribution}:
\begin{itemize}
  \item $S\cup S_0$ contains the places of $F$ with small residue characteristics such that the result by Cluckers--Gordon--Halupczok on uniform bound on orbital integrals \cite[Thm 14.1]{ST:Sato-Tate} applies to places outside $S\cup S_0$.
\end{itemize}
We also assume that (to use Lemma \ref{l:1-dim} relying on strong approximation)
\begin{itemize}
  \item the simply connected cover of $G^{\mathrm{der}}$ has no $F$-simple factor that is $F_v$-anisotropic for every $v\in S$.
\end{itemize}

Let $\cF(\xi,\mathrm{St}_S,K_{S_0})$ denote the multi-set of $\pi\in \cA_{\disc}(G)$, with $m_{\mathrm{disc}}(\pi)\dim (\pi_{S_0})^{K_{S_0}} $ as the multiplicity of $\pi$, which satisfies the following conditions: first, $\pi_\infty\in \Pi_\infty(\xi)$; second, $\pi_v$ is isomorphic to the Steinberg representation for all $v\in S$; third, $\pi_v$ is unramified for all finite
places $v\not\in S\cup S_0$.  In this section we study $\cF(\xi,\mathrm{St}_S,K_{S_0})$ as we vary the set $S$. The situation is somewhat
complementary to that in the previous section. We refer to $\cF(\xi,\mathrm{St}_S,K_{S_0})$ as a \emph{horizontal family}.\footnote{Even though $\xi$ is allowed to vary ``vertically'' (as in the last section), the main novel feature is to allow $S$ to vary, so the family deserves the name.}

Kottwitz~\cite{Kot88} constructed Euler-Poincar\'e functions for $p$-adic groups. For any place $v\in S$, we denote it by $\phi^{\EP}_{v}\in C_c^\infty(G(F_v))$.
We have that~\cite[Thm~2]{Kot88}
\begin{equation}\label{e:orb-int-EP}
O_\gamma(\phi^{\EP}_v)=
\begin{cases}1, & \gamma\in G(F_v)_{\elp},\\
0, & \gamma\in G(F_v)_{\semis}\backslash G(F_v)_{\elp}.
\end{cases}
\end{equation}
  We will assume that
\begin{itemize}
  \item $G$ is a simple\footnote{Such a group is often said to be absolutely almost simple, e.g. in \cite{Gross:prescribed}.} algebraic group, i.e. every proper normal subgroup of $G$ over an algebraic closure of $F$ is finite.
\end{itemize}
In that case Casselman's theorem, cf.~\cite[Thm 2]{Cas} and~\cite[Thm~2']{Kot88}, tells us that an irreducible unitary representation $\pi_v$ with $\tr \pi_v(\phi^{\EP}_v)\neq 0$ can occur only in the following two cases (which may not be mutually exclusive):
\begin{enumerate}[(i)]
\item $\pi_v$ is the trivial representation and $\tr \pi_v(\phi^{\EP}_v)=1$, or
\item $\pi_v\simeq \St_v$ and $\tr \pi_v(\phi^{\EP}_v)=(-1)^{q(G_v)}$, where $q(G_v)$ is the semisimple rank of $G_{F_v}$. In particular $f_{\St_v}:=(-1)^{q(G_v)}\phi^{\EP}_v$ is a pseudo-coefficient for $\St_v$.
\end{enumerate}
Set $f_{\St_S}:=\prod_{v\in S} f_{\St_v}$, $q_S:= \prod_{v\in S} q_v$, and $q(G_S):=\sum_{v\in S}q(G_v)$.

\begin{remark}
The formal degree of $\St_v$ is equal to $(-1)^{q(G_v)}$ for the Euler-Poincar\'e measure by the Plancherel theorem, (ii) above, and $
\phi^{\EP}_v(1)=1$ recalled above. See also~\cite{Borel:Iwahori}. %
\end{remark}
We will state the analogue of Lemma \ref{l:LHS=T_ell} for $\cF(\xi,\St_S,K_{S_0})$ after recalling the following well-known result on one-dimensional automorphic representations.
\begin{lem}\label{l:1-dim}
  Let $H$ be a connected reductive group over a number field $E$. Let $v$ be a place of $E$, and assume that the simply connected cover of $H^{\mathrm{der}}$ has no $E$-simple factor that is $E_v$-anisotropic (in particular, $H(E_v)$ is not compact modulo center). If $\pi$ is an automorphic representation of $H(\A_E)$ with $\dim\pi_v=1$ then $\pi$ is one-dimensional.
\end{lem}

\begin{proof}
  Using a $z$-extension of $H$ (\cite[\S1]{Kot82}) we can reduce to the case when the derived subgroup $H^{\mathrm{der}}$ of $H$ is simply connected. Now the groups $H^{\mathrm{der}}(E)$ and $H^{\mathrm{der}}(E_v)$ are subgroups of $H(\A_E)$ so they act on $\pi$ viewed as a subspace of the regular representation of $H(\A_E)$ on automorphic forms of $H(\A_E)$ (with a fixed central character). The one-dimensionality of $\pi_v$ implies that $H^{\mathrm{der}}(E_v)$ acts trivially on $\pi$. Then any smooth vector of $\pi$, as a function $\phi: H(\A_E)\ra \C$, has the property that $\phi(h_0 h h_v)=\phi(h)$ for $h_0\in H^{\mathrm{der}}(E)$, $h\in H(\A_E)$, and $h_v\in H^{\mathrm{der}}(E_v)$. Hence $\phi(hh')=\phi(h)$ for $h'\in h^{-1} H^{\mathrm{der}}(E) h H^{\mathrm{der}}(E_v)=h^{-1} H^{\mathrm{der}}(E)H^{\mathrm{der}}(E_v) h$.
  By the strong approximation theorem, $h'$ runs over a dense subset of $H^{\mathrm{der}}(\A_E)$. Therefore $\phi$ is constant on left $H^{\mathrm{der}}(\A_E)$-cosets, i.e. $H(\A_E)$ acts through an abelian quotient. Hence $\pi$ is one-dimensional.
\end{proof}

\begin{lem}\label{l:LHS=T_ell-Stcase} Let $\mathfrak{S}:=S_\infty\cup S_0\cup S$ and $\phi^{\mathfrak{S}}\in \cH^{\ur}(G(\A_F^{\mathfrak{S}}))$. Put $\phi:=\phi^{\mathfrak{S}} f_{\St_S} \triv_{K_{S_0}} f_\xi$. Then
  $$\sum_{\pi\in \cF(\xi,\St_S,K_{S_0})}
\tr \pi^{\mathfrak{S}}\left(\phi^{\mathfrak{S}}, \mu^{\can,{\mathfrak{S}}}\right)=\mu^{\can}_{S_0}(K_{S_0})^{-1}T_\elp(\phi,\mu^{\can,\EP}).$$
\end{lem}

\begin{proof}
  We have the simple trace formula \eqref{e:simple-TF} for $\phi$. Indeed $f_{\St_S}$ enjoys property (ii) in \S\ref{sub:simple-TF}.
 It suffices to show that the left hand side equals $\mu^{\can}_{S_0}(K_{S_0})^{-1} T_{\disc}(\phi,\mu^{\can,\EP})$, which expands as %
\begin{equation}\label{e:LHS=T_ell-Stcase}
\sum_{\pi\in \cA_{\disc}(G)} m_{\disc}(\pi) \tr \pi^{\mathfrak{S}}(\phi^{\mathfrak{S}})\tr \pi_S(f_{\St_S})
\frac{\tr \pi_{S_0}(\triv_{K_{S_0}})}{\mu^{\can}_{S_0}(K_{S_0})} \tr \pi_\infty(f_\xi).
\end{equation}
Suppose that the summand for $\pi$ is nonzero. Then $\pi_\infty$ cannot be one-dimensional by regularity of $\xi$.
Let $v\in S$ so that $\tr \pi_v(f_{\St_v})\neq 0$. If $G(F_v)$ is compact modulo center
then $\St_v$ is the trivial representation, so $\pi_v\simeq \St_v$ by Casselman's
theorem above. If $G(F_v)$ is not compact modulo center then Lemma \ref{l:1-dim}
tells us that $\pi_v$ cannot be one-dimensional so $\pi_v\simeq \St_v$ by the same
theorem. In either case $\tr \pi_v(f_{\St_v})=1$. We see that
\eqref{e:LHS=T_ell-Stcase} equals the left hand side of the lemma.
\end{proof}

\begin{thm}\label{th:steinbergs}
There exist real constants
$\nu_\infty,A>0$ and an integer $B\in \Z_{\ge 1}$ such that for every $\xi\in \Irr^{\reg}_C(G(F_\infty))$, for every $\underline\kappa=(\kappa_v)_{v\notin {\mathfrak{S}}}$, and for every $\phi^{\mathfrak{S}}\in \cH^{\ur,\le \underline{\kappa}}(G(\A_F^{\mathfrak{S}}))$ which is the characteristic function of a bi-$K^{\mathfrak{S}}$-invariant compact subset,
  \begin{eqnarray}
 \frac{1}{\tau'(G,S)} \sum_{\pi\in \cF(\xi,\St_S,K_{S_0})}
\tr \pi^{\mathfrak{S}}(\phi^{\mathfrak{S}})
&=& \frac{(-1)^{q(G_S)+q(G_\infty)}d(G_\infty)\dim \xi}{\mu^{\can}_{S_0}(K_{S_0})} \sum_{z\in Z(F)\cap K_{S_0}} \omega_\xi(z)\phi^{\mathfrak{S}}(z)\nonumber\\
&&+~ O(q_S^{-B}\dim(\xi)^{1-\nu_\infty} Q^{A\underline{\kappa}} ).\nonumber
\end{eqnarray}
  The implicit constant in $O(\cdot)$ depends on $G$, $S_0$, $K_{S_0}$, and $C$ (but is independent of $\xi$, $S$, $\underline{\kappa}$, and $\phi^{\mathfrak{S}}$).

\end{thm}
\begin{proof}
  The proof proceeds exactly as for Theorem \ref{t:disc-series-case}. The main term coming from $\gamma\in Z(F)$ is computed similarly. (Note that $f_{\St_S}(z)=(-1)^{q(G_S)}=\deg(\St_S)$ for $z\in Z(F_S)$. So $\deg(\sigma_S)$ in \eqref{e:in-thm-disc-series-case} is replaced by $(-1)^{q(G_S)}$ here.) The issue is to bound the contribution from noncentral elements. To explain the mild modifications in the argument we freely use the notation from the proof there. It suffices to show the following analogue of \eqref{e:disc-series-case} for uniform constants $\nu_\infty,A>0$ and $B\in \Z_{\ge1}$:
    \begin{equation}\label{e:St-case}
    \frac{1}{\tau'(G,S)}
  \sum_{\gamma\in \mathcal{Y}(\underline\kappa) \atop \gamma\notin Z(F)} \left|\mu^{\can,\EP}_{I_\gamma}(I_\gamma)\orbi^{G(\A_F)}_\gamma(\phi^{\mathfrak{S}} f_{\St_S} \triv_{K_{S_0}} f_\xi)\right|
  = O(q_S^{-B}\dim(\xi)^{1-\nu_\infty} Q^{A\underline{\kappa}} ).
  \end{equation}
  The only nontrivial change is to replace \eqref{e:disc-series-case-1} and \eqref{e:disc-series-case-2} with the following inequalities for noncentral elements $\gamma\in \mathcal{Y}(\underline\kappa)$ for suitable uniform constants $c_0,A_5>0$:
  \begin{eqnarray}
    \frac{|\mu^{\can,\EP}_{I_\gamma}(I_\gamma)|}{\tau'(G,S)} &\le& c_0 q_S^{-B} Q^{A_5\underline{\kappa}},\label{e:St-case-1}\\
  D_S(\gamma)^{1/2}  |\orbi_\gamma(f_{\St_S})| &\le & 1 \label{e:St-case-2}.
  \end{eqnarray}
  The bounds for orbital integrals away from $S$ as well as the rest of the proof are exactly the same as in the proof of Theorem \ref{t:disc-series-case}. So we content ourselves with justifying the two inequalities, beginning with \eqref{e:St-case-2}. (Along the way we obtain a bound on $q_{S_\gamma}$. Evidently the same bound works for $q_{S'_{\gamma,0}}$. Together with \eqref{e:disc-series-case-2}, this is used to bound the orbital integral away from $\mathfrak{S}$.)

   As we recalled in \eqref{e:orb-int-EP}, $|\orbi_\gamma(f_{\St_S})|$ is bounded by $1$ and nonzero only on $\gamma$ that is elliptic in $G(F_v)$ for all $v\in S$. Then $\gamma$ is contained in an $F_v$-anisotropic maximal torus $T_v$ of $G(F_v)$ so $|\alpha(\gamma)|_v=1$ for all absolute roots $\alpha$ of $T_v$ in $G$ by compactness, implying that $|1-\alpha(\gamma)|_v\le 1$. Therefore $D_S(\gamma)\le 1$ and \eqref{e:St-case-2} holds.

  It remains to check the uniform bound \eqref{e:St-case-1} for noncentral elements $\gamma$.
   If $\gamma$ is not in the center of $G$, we have that $\frac{1}{2}(\dim I_\gamma-\rk I_\gamma)-\frac{1}{2}(\dim G-\rk G)$ is a negative integer. Take $B$ to be the maximum of such integers as $\gamma$ varies. Our initial assumption tells us that $\Ram(G)\subset S_0$.   At each $v\in S_0$ we define $\delta_v<\infty$ as in \eqref{eq:delta_v} or $\delta_v=1$, whichever is bigger. We define $d_G,s_G,w_G\in \Z_{>0}$ as in the proof of Theorem \ref{t:disc-series-case}.

     From Lemma \ref{l:bound-volume} and the fact that
     $$\Ram(I_\gamma)\subset \Ram(G)\cup S_\gamma \cup \{v\notin\mathfrak{S}: \kappa_v\neq 0\}
     \subset S_0 \cup S_\gamma \cup \{v\notin\mathfrak{S}: \kappa_v\neq 0\},$$
    it follows that
  \begin{eqnarray}
  \frac{|\mu^{\can,\EP}_{I_\gamma}(I_\gamma)|}{\tau'(G,S)}& \le & c_0 q_S^{\frac{1}{2}(\dim I_\gamma-\rk I_\gamma)-\frac{1}{2}(\dim G-\rk G)}\prod_{v\in \Ram(I_\gamma)} q_{v}^{A_2}. \nonumber\\
  & \le & c_0 q_S^{-B} q_{S_0}^{A_2} q_{S_\gamma}^{A_2}Q^{A_2} = O(q_S^{-B}q_{S_\gamma}^{A_2}Q^{A_2}).\label{eq:mu-over-tau'}
  \end{eqnarray}

  We will bound $q_{S_\gamma}$ independently of $\gamma$, by considering the partition
  $$S_\gamma=(S_\gamma\cap S)\coprod (S_\gamma\cap S_0) \coprod (S'_\gamma\backslash S'_{\gamma,0}) \coprod S'_{\gamma,0}.$$
    When $v\in S_\gamma\cap S$, we have $|1-\alpha(\gamma)|_v\le 1$ as explained above, and the inequality is strict for some $\alpha\in \Phi_\gamma$, thus as in  \eqref{e:bound-1-alpha(gamma)}, we have
    $$\prod_{\alpha\in \Phi_\gamma} |1-\alpha(\gamma)|_v\le q_v^{-\frac{1}{w_Gs_G}},\qquad v\in S_\gamma\cap S.$$
     At $v\in S_\gamma\cap S_0$ we have $|1-\alpha(\gamma)|_v\le \delta_v$. For $v\in S'_\gamma\backslash S'_{\gamma,0}$ and $v\in S'_{\gamma,0}$ the bounds of \eqref{e:bound-1-alpha(gamma)-by-Q} and \eqref{e:bound-1-alpha(gamma)} remain valid without change. Applying these bounds to the product formula for $1-\alpha(\gamma)$ with $\alpha\in \Phi_\gamma$ (compare with \eqref{e:product-1-alpha(gamma)})
  \begin{eqnarray*}
  1&=&\prod_{\alpha\in \Phi_\gamma} \prod_v |1-\alpha(\gamma)|_v ~\le ~\delta_\infty^{|\Phi_\gamma|} \prod_{\alpha\in \Phi_\gamma} \prod_{v\in S_\gamma} |1-\alpha(\gamma)|_v \\
  & \le &(\delta_\infty\prod_{v\in S_0} \delta_v)^{d_G} q_{S_{\gamma}}^{-\frac{1}{w_Gs_G}} \left(\prod_{v\in S'_\gamma\backslash S'_{\gamma,0}}q_v^{\frac{1}{w_Gs_G}} \right) Q^{d_G  B_1\underline\kappa}.
   \end{eqnarray*}
      Increasing $B_1$ if necessary, we can disregard the second bracketed term (the product over $v\in S'_\gamma\backslash S'_{\gamma,0}$). We deduce that $q_{S_\gamma}=O(Q^{d_G w_Gs_G B_1\underline\kappa})$.
      Using this bound in \eqref{eq:mu-over-tau'} we conclude:
      $$\frac{|\mu^{\can,\EP}_{I_\gamma}(I_\gamma)|}{\tau'(G,S)}= O\left(q_S^{-B}q_{S_\gamma}^{A_2}Q^{A_2}\right)
= O\left(q_S^{-B} Q^{(d_Gw_Gs_G A_2 B_1+A_2)\underline\kappa} \right).$$

\end{proof}

  Denote by $\Xi_1$ (resp. $\Xi_{\neq1}$) the subset of $\xi\in \Irr^{\alg}_C(G(F_\infty))$ whose central character is trivial (resp.
non-trivial).

\begin{cor}\label{c:St-case}
  We have the limit multiplicity formulas
\begin{eqnarray}
  \lim_{q_S\dim(\xi)\ra\infty\atop \xi\in \Xi_{\neq 1}} \frac{m(\xi,\St_S,K_{S_0})}{d(G_\infty)\dim(\xi) \tau'(G,S)} &=&0,\nonumber\\
\lim_{q_S\dim(\xi)\ra\infty\atop \xi\in \Xi_{ 1}} \frac{(-1)^{q(G_S)}m(\xi,\St_S,K_{S_0})}{d(G_\infty)\dim(\xi) \tau'(G,S)}
&= &\frac{(-1)^{q(G_\infty)} |Z(F)\cap K_{S_0}K^{\mathfrak{S}}|}{\mu^{\can}_{S_0}(K_{S_0})}.\label{e:mult-St-case}
\end{eqnarray}
More precisely, for each $\epsilon>0$, there exists $\delta_\epsilon>0$ with the following property: for every finite set of finite places
$S$ such that $S\cap \Ram(G)=\emptyset$ and for every $\xi\in \Xi_{\neq 1}$ such that $q_S\dim(\xi)>\delta_\epsilon$ (while $S_0$ is fixed),
we have $|m(\xi,\St_S,K_{S_0})|<\epsilon \cdot d(G_\infty)\dim(\xi) \tau'(G,S)$. The second limit formula is interpreted in a similar way.
\end{cor}

\begin{proof}
  This follows from the preceding theorem exactly as Corollary \ref{cor:limit-multiplicity} does from Theorem \ref{t:disc-series-case}.
\end{proof}

\begin{ex}\label{ex:PGL2}
Consider the case when $G=\PGL(2)$ over a totally real field $F$ and $S_0=\emptyset$. Then $d(G_\infty)=1$, $q(G_\infty)=(-1)^{[F:\Q]}$, and the right hand side of \eqref{e:mult-St-case} is $(-1)^{[F:\Q]}$. (When $F=\Q$, corresponding to classical holomorphic modular forms of even weight $k\in \Z_{\ge 2}$\footnote{The regularity condition on $\xi$ excludes $k=2$ but we can easily work out the case $k=2$. The simple trace formula is still valid for the same test function $\phi$. The only extra work is to bound the extra spectral terms from one-dimensional automorphic representations, cf. Lemma \ref{l:1-dim}, which do not show up when $k>2$.}  with trivial Nebentypus character is the representation $\xi_k=\mathrm{Sym}^{k-2}$ of $\PGL(2)$ so that $\dim(\xi_k)=k-1$. Similarly $\dim(\xi)$ is computed for general $F$.) Gross's motive for $\PGL(2)$ is $\Q(-1)$. We can easily compute $\tau(G)=2$ and
$$\tau'(G,S)=\zeta_F(-1) 2^{1-[F:\Q]}\prod_{v\in S} \frac{1-q_v}{2},$$ where $\zeta_F$ is the Dedekind zeta function (in particular $\zeta_\Q(-1)=-1/12$). Since $\dim G-\rk G=2$, we can take $B=1$ in Theorem \ref{th:steinbergs} by its proof and $\nu_\infty=1$ by Remark \ref{r:epsilon_infty}. So we have the asymptotic formula
\begin{equation}\label{e:example-ILS}
 m(\xi,\St_S,\emptyset) = |\zeta_F(-1)|2^{1-[F:\Q]-|S|} \prod_{v\in S} (q_v-1) + O(1).
 \end{equation}
 Here the bound $O(1)$ comes from $q_S^{-1}\prod_{v\in S}(q_v-1)=O(1)$.
In the special case when $F=\Q$ and $\xi_k=\mathrm{Sym}^{k-2}$ is fixed, we have
 $$m(\xi_k,\St_S,\emptyset)\sim \frac{(k-1)\phi(q_S)}{12\cdot 2^{|S|}}\quad\mbox{as}~q_S\ra\infty,$$
where $\phi(\cdot)$ is Euler's phi-function. The difference from \cite[Cor 2.14]{ILS00} by the factor $2^{|S|}$ is explained as follows. At each $v\in S$ there is another square-integrable representation $\St'_v$ such that $\St_v$ and $\St'_v$ differ as representations of $\GL_2(\Q_v)$ by the unique nontrivial unramified quadratic character of $\Q_v^\times$. The result in \emph{loc. cit.} can be interpreted as $\sum_{\sigma_S} m(\xi_k,\sigma_S,\emptyset)$, where the sum runs over $\sigma_S$ such that $\sigma_v\in \{\St_v,\St'_v\}$. Thus their count is $2^{|S|}$ times ours.
In this special case, observe that our $O(1)$ in \eqref{e:example-ILS} improves on the error bound $O((kq_S)^{2/3})$ obtained in \emph{loc. cit.}

\end{ex}

\begin{cor}\label{c:St-case-restated} Suppose that $G$ is a split simple reductive group over $F$. There exist constants $\epsilon,A,B>0$ such that for every $\xi\in \Irr^{\reg}_C(G(F_\infty))$ and for every $\phi^{\mathfrak{S}}=\triv_{K^{\mathfrak{S}} g K^{\mathfrak{S}}}\in \cH^{\ur,\le \underline{\kappa}}(G(\A_F^{\mathfrak{S}}))$ with some $g\in G(\A_F^{\mathfrak{S}})$,
 $$
  \sum_{\pi\in \cF(\xi,\St_S,K_{S_0})}
\tr \pi^{\mathfrak{S}}(\phi^{\mathfrak{S}})
=  m(\xi,\St_S,K_{S_0})\sum_{z\in Z(F)\cap K_{S_0}} \phi^{\mathfrak{S}}(z)
+ O(m(\xi,\St_S,K_{S_0})^{1-\nu}q_S^{-B} \|\phi^{\mathfrak{S}}\|_1^{A} ).
$$
  The implicit constant in $O(\cdot)$ depends on $G$, $S$, $S_0$, $K_{S_0}$, and $C$ (but is independent of $\xi$ and $\phi^{\mathfrak{S}}$).
\end{cor}

\begin{proof}
  This is proved in the same way as Corollary \ref{c:disc-seires-case-restated}.
\end{proof}

As an immediate consequence of Theorem~\ref{th:steinbergs}, we deduce the existence of representations with very mild ramification (e.g. one can take $S$ to be a singleton) and fixed weight.
\begin{cor}
Let $G$ be a split simple reductive group over $F$. Fix $\xi\in \Irr^{\reg}(G(F_\infty))$. There exists a constant $q_0>0$ with the following property: for every finite set of places $S$ such that $G$ is unramified away from $S$, if $q_S>q_0$ then there exists a cuspidal automorphic representation $\pi$ such that
\begin{itemize}
  \item $\pi_\infty\in \Pi_\infty(\xi)$ (in particular it is a discrete series representation),
  \item $\pi_v$ is the Steinberg representation at each $v\in S$,
  \item $\pi$ is unramified at every finite place $v\notin S$.
\end{itemize}
\end{cor}

\begin{proof}
  We take $S_0=\emptyset$.
  Theorem \ref{th:steinbergs} implies that $\cF(\xi, \mathrm{St}_S,\emptyset)$ is nonempty if $q_S$ is sufficiently large since the main term in the right hand side of the theorem is nonzero. Any $\pi\in \cF(\xi,\mathrm{St}_S,\emptyset)$ is (not only discrete but) cuspidal since it has a Steinberg component, cf. \cite[Prop 4.5.4]{Lab99}.
\end{proof}

\newcommand{\etalchar}[1]{$^{#1}$}
\def\cprime{$'$} \def\cprime{$'$} \def\cprime{$'$} \def\cprime{$'$}
  \def\cprime{$'$} \def\cprime{$'$}
\providecommand{\bysame}{\leavevmode\hbox to3em{\hrulefill}\thinspace}
\providecommand{\MR}{\relax\ifhmode\unskip\space\fi MR }
\providecommand{\MRhref}[2]{%
  \href{http://www.ams.org/mathscinet-getitem?mr=#1}{#2}
}
\providecommand{\href}[2]{#2}


\begin{thebibliography}{ABB{\etalchar{+}}11}

%
%
%
%
%

\bibitem[ABB{\etalchar{+}}11]{BG7:betti11}
Miklos Abert, Nicolas Bergeron, Ian Biringer, Tsachik Gelander, Nikolay
  Nikolov, Jean Raimbault, and Iddo Samet, \emph{On the growth of $L^2$-invariants
  for sequences of lattices in Lie groups}, Ann. Math., \textbf{185} (2017), no.
  3, 711--790.

\bibitem[Adl98]{Adl98} J.Adler,
\emph{Refined anisotropic $K$-types and supercuspidal representations},
Pacific J. Math. \textbf{185} (1998), no.~1, 1--32.

\bibitem[AK07]{AK07}  J. Adler and J. Korman,
\emph{The local character expansion near a tame, semisimple element,}
Amer. J. Math. \textbf{129} (2007), no. 2, 381--403.

\bibitem[AR00]{Adler-Roche:intertwinning}
J. Adler, and A. Roche, \emph{An intertwining result for $p$-adic
groups},
Canad. J. Math. \textbf{52} (2000), no.~3, 449--467.

  \bibitem[AS08]{AS08}
J. Adler and L. Spice, \emph{Good product expansions for tame elements
of p-adic groups},
Int. Math. Res. Pap. IMRP (2008), no. 1, Art. ID rp. 003, 95 pp.

\bibitem[AS09]{AS09}
\bysame, \emph{Supercuspidal characters of reductive
  {$p$}-adic groups}, Amer. J. Math. \textbf{131} (2009), no.~4, 1137--1210.
  \MR{2543925}

\bibitem[Art87]{Art87}
James Arthur, \emph{The characters of supercuspidal representations as weighted
  orbital integrals}, Proc. Indian Acad. Sci. Math. Sci. \textbf{97} (1987),
  no.~1-3, 3--19 (1988).

\bibitem[Art88]{Art88b}
\bysame, \emph{The invariant trace formula. {II}. {G}lobal theory}, J. Amer.
  Math. Soc. \textbf{1} (1988), no.~3, 501--554.

\bibitem[Art89]{Art89}
\bysame, \emph{The {$L^2$}-{L}efschetz numbers of {H}ecke operators}, Invent.
  Math. \textbf{97} (1989), no.~2, 257--290.

\bibitem[Art93]{Art93}
\bysame, \emph{On elliptic tempered characters}, Acta Math. \textbf{171}
  (1993), no.~1, 73--138.

\bibitem[Art03]{ArtSTF3}
\bysame, \emph{A stable trace formula. {III}. {P}roof of the main theorems},
  Ann. of Math. (2) \textbf{158} (2003), no.~3, 769--873.

\bibitem[Art13]{Arthur}
\bysame, \emph{The endoscopic classification of representations}, American
  Mathematical Society Colloquium Publications, vol.~61, American Mathematical
  Society, Providence, RI, 2013, Orthogonal and symplectic groups.

\bibitem[BKV]{BKV}
Roman Bezrukavnikov, David Kazhdan, and Yakov Vashavsky, \emph{On the depth $r$
  {B}ernstein projector}, preprint, arXiv:1504.01353.

\bibitem[Bor76]{Borel:Iwahori}
A.~Borel, \emph{Admissible representations of a semi-simple group over a
  local field with vectors fixed under an {I}wahori subgroup}, Invent. Math.
  \textbf{35} (1976), 233--259. \MR{0444849 (56 \#3196)}

\bibitem[Car12]{Car12}
Ana Caraiani, \emph{Local-global compatibility and the action of monodromy on
  nearby cycles}, Duke Math. J. \textbf{161} (2012), no.~12, 2311--2413.
  \MR{2972460}

\bibitem[Cas]{Cas}
W.~Casselman, \emph{Introduction to the theory of admissible representations of
  $p$-adic reductive groups},
  \verb"www.math.ubc.ca/~cass/research/pdf/p-adic-book.pdf".

\bibitem[Clo86]{Clo86}
L.~Clozel, \emph{On limit multiplicities of discrete series
  representations in spaces of automorphic forms}, Invent. Math. \textbf{83}
  (1986), no.~2, 265--284. \MR{818353 (87g:22012)}

\bibitem[Clo91]{Clo91}
\bysame, \emph{Repr\'esentations galoisiennes associ\'ees aux repr\'esentations
  automorphes autoduales de {${\rm GL}(n)$}}, Inst. Hautes \'Etudes Sci. Publ.
  Math. (1991), no.~73, 97--145. \MR{1114211 (92i:11055)}

\bibitem[Clo13]{Clozel:purity}
\bysame, \emph{Purity reigns supreme}, Int. Math. Res. Not. IMRN (2013), no.~2,
  328--346. \MR{3010691}

\bibitem[CD90]{CD90}
L.~Clozel and P.~Delorme, \emph{Le th\'{e}or\`{e}me de {P}aley-{W}iener
  invariant pour les groupes de lie r\'{e}ductifs. {II}}, Ann. Sci. \'{E}cole
  Norm. Sup. \textbf{23} (1990), 193--228.

\bibitem[CMS90]{CMS90}
L. Corwin, A. Moy and P. Sally Jr.,
\emph{Degrees and formal degrees for division algebras and GLn over a p-adic field},
Pacific J. Math. \textbf{141} (1990), no. 1, 21--45.

\bibitem[CM04]{Cogdell2004}
J.~Cogdell and Ph. Michel, \emph{On the complex moments of symmetric power
  {$L$}-functions at {$s=1$}}, Int. Math. Res. Not. (2004), no.~31, 1561--1617.

\bibitem[Dal19]{Dal19}
R.~Dalal, \emph{Sato-{T}ate Equidistribution for Families of Automorphic Representations through the Stable Trace Formula},
	arXiv:1910.10800

\bibitem[dGW78]{DeGeor-Wall}
D.~L. de~George and N.~R. Wallach, \emph{Limit formulas for
  multiplicities in {$L^{2}(\Gamma \backslash G)$}}, Ann. of Math. (2)
  \textbf{107} (1978), no.~1, 133--150.

\bibitem[DW79]{DeGeor-Wall-II}
\bysame, \emph{Limit formulas for multiplicities
  in {$L^{2}(\Gamma \backslash G)$}. {II}. {T}he tempered spectrum}, Ann. of
  Math. (2) \textbf{109} (1979), no.~3, 477--495.

\bibitem[DeB]{DeBacker:parametrization} S.~DeBacker,
\emph{Parametrizing nilpotent orbits via Bruhat-Tits theory},
Ann. of Math. (2) \textbf{156} (2002), no.~1, 295--332.


\bibitem[DeB02]{DeBacker:homogeneity}
\bysame,
\emph{Homogeneity results for invariant distributions of a reductive p-adic group}, Ann.
Sci. École Norm. Sup. (4) \textbf{35} (2002), no.~3, 391–-422.



\bibitem[DR]{DeBacker-Reeder}
S.~DeBacker and M.~Reeder, \emph{Depth-zero supercuspidal L-packets and their
stability},
Ann. of Math. (2) \textbf{169} (2009), no.~3, 795-–901.

\bibitem[DS]{DS:supercuspidal}
S.~DeBacker and L.~Spice, \emph{Stability of character sums for
  positive-depth, supercuspidal representations}, arXiv:1310.3306, to appear in J. Reine Angew. Math.

\bibitem[DS00]{DeBacker-Sally}
S.~DeBacker and P. Sally, Jr., \emph{Germs, characters, and the
  {F}ourier transforms of nilpotent orbits}, The mathematical legacy of
  {H}arish-{C}handra ({B}altimore, {MD}, 1998), Proc. Sympos. Pure Math.,
  vol.~68, Amer. Math. Soc., Providence, RI, 2000, pp.~191--221. \MR{1767897
  (2001i:22022)}


\bibitem[DKV84]{DKV84}
P.~Deligne, D.~Kazhdan, and M.-F. Vign{\'e}ras, \emph{Repr\'esentations des
  alg\`ebres centrales simples {$p$}-adiques}, Representations of reductive
  groups over a local field, Travaux en Cours, Hermann, Paris, 1984,
  pp.~33--117. \MR{771672 (86h:11044)}

\bibitem[Fin18a]{Fin18a} J.~Fintzen, \emph{Tame tori in $p$-adic groups and good semisimple
elements}, arXiv:1801.04955.


\bibitem[Fin18b]{Fin18b} J.~Fintzen, \emph{Types for tame $p$-adic groups},
  arXiv:1810.04198.

\bibitem[Fer07]{Fer07}
Axel Ferrari, \emph{Th\'eor\`eme de l'indice et formule des traces},
  Manuscripta Math. \textbf{124} (2007), no.~3, 363--390. \MR{2350551
  (2008j:22026)}

\bibitem[Glu95]{Gluck:character-sharp}
David Gluck, \emph{Sharper character value estimates for groups of {L}ie type},
  J. Algebra \textbf{174} (1995), no.~1, 229--266. \MR{1332870 (96m:20021)}

\bibitem[Gro97]{Gro97}
B.H.~Gross, \emph{On the motive of a reductive group}, Invent. Math. \textbf{130}
  (1997), 287--313.

\bibitem[Gro98]{Gro98}
\bysame, \emph{On the {S}atake isomorphism}, Galois representations
  in arithmetic algebraic geometry ({D}urham, 1996), London Math. Soc. Lecture
  Note Ser., vol. 254, Cambridge Univ. Press, Cambridge, 1998, pp.~223--237.
  \MR{1696481 (2000e:22008)}

\bibitem[Gro11]{Gross:prescribed}
\bysame, \emph{Irreducible cuspidal representations with prescribed local
  behavior}, Amer. J. Math. \textbf{133} (2011), no.~5, 1231--1258.

\bibitem[HC99]{HC99}
Harish-Chandra, \emph{Admissible invariant distributions on reductive
  {$p$}-adic groups}, University Lecture Series, vol.~16, American Mathematical
  Society, Providence, RI, 1999, Preface and notes by Stephen DeBacker and Paul
  J. Sally, Jr. \MR{1702257 (2001b:22015)}

\bibitem[HL04]{HL04}
Michael Harris and Jean-Pierre Labesse, \emph{Conditional base change for
  unitary groups}, Asian J. Math. \textbf{8} (2004), no.~4, 653--683.
  \MR{2127943 (2006g:11098)}



\bibitem[Hum78]{Hum78}
James~E. Humphreys, \emph{Introduction to {L}ie algebras and representation
  theory}, Graduate Texts in Mathematics, vol.~9, Springer-Verlag, New York,
  1978, Second printing, revised. \MR{499562 (81b:17007)}

\bibitem[ILS00]{ILS00}
Henryk Iwaniec, Wenzhi Luo, and Peter Sarnak, \emph{Low lying zeros of families
  of {$L$}-functions}, Inst. Hautes \'Etudes Sci. Publ. Math. (2000), no.~91,
  55--131 (2001). \MR{1828743 (2002h:11081)}

\bibitem[Kal]{Kaletha:supercuspidal}
Tasho Kaletha, \emph{Regular supercuspidal representations}, preprint,
  arXiv:1602.03144.

\bibitem[Kaz86]{Kaz86}
D.~Kazhdan, \emph{Cuspidal geometry of $p$-adic groups}, J. Analyse Math.
  \textbf{47} (1986), 1--36.

%
%
%

\bibitem[Kim07]{Kim07}
J.-L. Kim, \emph{Supercuspidal representations: an exhaustion theorem}, J.
  Amer. Math. Soc. \textbf{20} (2007), no.~2, 273--320 (electronic).
  \MR{2276772 (2008c:22014)}

\bibitem[KST16]{KST-localconstancy}
J.-L. Kim, S.~W. Shin, and N.~Templier, \emph{Asymptotics and local constancy
  of characters of $p$-adic groups}, Proceedings of Simons Symposia, Families
  of Automorphic Forms and the Trace Formula, Springer Verlag (2016), 259--295.

\bibitem[Kot82]{Kot82}
R.~E. Kottwitz, \emph{Rational conjugacy classes in reductive groups}, Duke
  Math. J. \textbf{49} (1982), no.~4, 785--806. \MR{683003 (84k:20020)}

%
%
%

\bibitem[Kot86]{Kot86b}
\bysame, \emph{Stable trace formula: elliptic singular terms},
Math. Ann. \textbf{275} (1986), no.~3, 365--399. \MR{858284 (88d:22027)}

\bibitem[Kot88]{Kot88}
\bysame, \emph{Tamagawa numbers}, Ann. of Math. (2) \textbf{127} (1988), no.~3,
  629--646. \MR{942522 (90e:11075)}

\bibitem[Kot92]{Kot92b}
\bysame, \emph{On the $\lambda$-adic representations associated to some
  simple {S}himura variaties}, Invent. Math. \textbf{108} (1992), 653--665.

\bibitem[Lab99]{Lab99}
J.-P. Labesse, \emph{Cohomologie, stabilisation et changement de base},
  Ast\'{e}risque, no. 257, 1999.

\bibitem[Mar05]{Martin:dimension}
Greg Martin, \emph{Dimensions of the spaces of cusp forms and newforms on
  {$\Gamma_0(N)$} and {$\Gamma_1(N)$}}, J. Number Theory \textbf{112} (2005),
  no.~2, 298--331. \MR{2141534 (2005m:11069)}

\bibitem[Mok15]{Mok}
Chung~Pang Mok, \emph{Endoscopic classification of representations of
  quasi-split unitary groups}, Mem. Amer. Math. Soc. \textbf{235} (2015),
  no.~1108, vi+248. \MR{3338302}

\bibitem[MP94]{MP94}
Allen Moy and Gopal Prasad, \emph{Unrefined minimal {$K$}-types for {$p$}-adic
  groups}, Invent. Math. \textbf{116} (1994), no.~1-3, 393--408. \MR{1253198
  (95f:22023)}

\bibitem[MT]{MT:gln}
J.~Matz and N.~Templier, \emph{Sato-{T}ate equidistribution for families of
  {H}ecke--{M}aass forms on {SL(n,R)/SO(n)}}, arXiv:1505.07285.

\bibitem[PR94]{PR94}
V.~Platonov and A.~Rapinchuk, \emph{Algebraic groups and number theory},
Translated from the 1991 Russian original by Rachel Rowen. Pure and Applied Math.,
139. Academic Press, Inc., Boston, MA, 1994, xii+614 pp. \MR{1278263}


\bibitem[Pra01]{Pra01}
Gopal Prasad, \emph{Galois-fixed points in the {B}ruhat-{T}its building of a
  reductive group}, Bull. Soc. Math. France \textbf{129} (2001), no.~2,
  169--174.

\bibitem[Rou77]{Rou77}
Guy Rousseau, \emph{Immeubles des groupes r\'educitifs sur les corps locaux},
  U.E.R. Math\'ematique, Universit\'e Paris XI, Orsay, 1977, Th{\`e}se de
  doctorat, Publications Math{\'e}matiques d'Orsay, No. 221-77.68.

\bibitem[RS87]{RS87}
J.~Rohlfs and B.~Speh, \emph{On limit multiplicities of representations with
  cohomology in the cuspidal spectrum}, Duke Math. \textbf{55} (1987),
  199--211.

\bibitem[RY14]{RY14}
Mark Reeder and Jiu-Kang Yu, \emph{Epipelagic representations and invariant
  theory}, J. Amer. Math. Soc. \textbf{27} (2014), no.~2, 437--477.
  \MR{3164986}

\bibitem[Sar05]{Sar05}
Peter Sarnak, \emph{Notes on the generalized {R}amanujan conjectures}, Harmonic
  analysis, the trace formula, and {S}himura varieties, Clay Math. Proc.,
  vol.~4, Amer. Math. Soc., Providence, RI, 2005, pp.~659--685. \MR{2192019
  (2007a:11067)}

\bibitem[Sav89]{Sav89}
Gordan Savin, \emph{Limit multiplicities of cusp forms}, Invent. Math.
  \textbf{95} (1989), no.~1, 149--159. \MR{969416 (90c:22035)}

\bibitem[Ser97]{Serre:pl}
Jean-Pierre Serre, \emph{R\'epartition asymptotique des valeurs propres de
  l'op\'erateur de {H}ecke {$T_p$}}, J. Amer. Math. Soc. \textbf{10} (1997),
  no.~1, 75--102.

\bibitem[Ser07]{Serre:bounds-finite}
\bysame,
\emph{Bounds for the orders of the finite subgroups of $G(k)$},
Group representation theory, EPFL Press, Lausanne,
(2007), 405--450.

\bibitem[Sha11]{Shahidi:packet-Ramanujan}
Freydoon Shahidi, \emph{Arthur packets and the {R}amanujan conjecture}, Kyoto
  J. Math. \textbf{51} (2011), no.~1, 1--23. \MR{2784745}

\bibitem[Shi11]{Shi11}
S.~W. Shin, \emph{Galois representations arising from some compact {S}himura
  varieties}, Ann. of Math. (2) \textbf{173} (2011), no.~3, 1645--1741.
  \MR{2800722}

\bibitem[Shi12]{Shi-plan}
S.~W. Shin, \emph{Automorphic {P}lancherel density theorem}, Israel J. Math.
  \textbf{192} (2012), 83--120. \MR{3004076}

\bibitem[Sp08]{Spice08}
L. Spice, \emph{Topological Jordan decompositions}, J. Algebra \textbf{319} (2008), no. 8,
3141--3163. \MR{2408311}

\bibitem[SS97]{SS97}
Peter Schneider and Ulrich Stuhler, \emph{Representation theory and sheaves on
  the {B}ruhat-{T}its building}, Inst. Hautes \'Etudes Sci. Publ. Math. (1997),
  no.~85, 97--191. \MR{1471867 (98m:22023)}

\bibitem[SST16]{SST}
P.~Sarnak, S.~W. Shin, and N.~Templier, \emph{Families of {$L$}-functions and
  their symmetry}, Proceedings of Simons Symposia, Families of Automorphic
  Forms and the Trace Formula, Springer Verlag (2016), 531--578.

\bibitem[ST16]{ST:Sato-Tate}
S.~W. Shin and N.~Templier, \emph{Sato-{T}ate theorem for families and
  low-lying zeros of automorphic {$L$}-functions}, Invent. Math. \textbf{203}
  (2016), no.~1, 1--177, Appendix A by Robert Kottwitz, and Appendix B by Raf
  Cluckers, Julia Gordon and Immanuel Halupczok. \MR{3437869}

\bibitem[Wald08]{Wald:tordue-asterisque}
J.-L.~Waldspurger, \emph{L'endoscopie tordue n'est pas si tordue}, Mem. Amer. Math. Soc.
\textbf{194}
(2008),
no.~908.


\bibitem[Wei09]{Weinstein09}
Jared Weinstein, \emph{Hilbert modular forms with prescribed ramification},
  Int. Math. Res. Not. IMRN (2009), no.~8, 1388--1420.

\bibitem[Yu01]{Yu01}
Jiu-Kang Yu, \emph{Construction of tame supercuspidal representations}, J.
  Amer. Math. Soc. \textbf{14} (2001), no.~3, 579--622 (electronic).


\bibitem[Yu09]{Yu09}
\bysame, \emph{Bruhat-{T}its theory and buildings}, Ottawa lectures on
  admissible representations of reductive {$p$}-adic groups, Fields Inst.
  Monogr., vol.~26, Amer. Math. Soc., Providence, RI, 2009, pp.~53--77.
  \MR{2508720 (2010b:20051)}

\end{thebibliography}
\end{document}